\PassOptionsToPackage{pagebackref}{hyperref}

\documentclass[reqno]{amsart}

\usepackage{purrsonal}
\numberwithin{equation}{section}
\usepackage{tikz-cd}

\makeatletter
\renewcommand{\tocsection}[3]{%
  \indentlabel{\@ifnotempty{#2}{\bfseries\ignorespaces#1 #2\quad}}\bfseries#3}
\renewcommand{\tocsubsection}[3]{%
  \indentlabel{\@ifnotempty{#2}{\ignorespaces#1 #2\quad}}#3}
\newcommand\@dotsep{4.5}
\def\@tocline#1#2#3#4#5#6#7{\relax
  \ifnum #1>\c@tocdepth % then omit
  \else
    \par \addpenalty\@secpenalty\addvspace{#2}%
    \begingroup \hyphenpenalty\@M
    \@ifempty{#4}{%
      \@tempdima\csname r@tocindent\number#1\endcsname\relax
    }{%
      \@tempdima#4\relax
    }%
    \parindent\z@ \leftskip#3\relax \advance\leftskip\@tempdima\relax
    \rightskip\@pnumwidth plus1em \parfillskip-\@pnumwidth
    #5\leavevmode\hskip-\@tempdima{#6}\nobreak
    \leaders\hbox{$\m@th\mkern \@dotsep mu\hbox{.}\mkern \@dotsep mu$}\hfill
    \nobreak
    \hbox to\@pnumwidth{\@tocpagenum{\ifnum#1=1\bfseries\fi#7}}\par% <-- \bfseries for \section page
    \nobreak
    \endgroup
  \fi}
\AtBeginDocument{%
\expandafter\renewcommand\csname r@tocindent0\endcsname{0pt}
}
\def\l@section{\@tocline{1}{9pt}{0pc}{10pc}{}}
\def\l@subsection{\@tocline{2}{0pt}{1.5pc}{3pc}{}}
\makeatother
\renewcommand{\sI}{I}
\renewcommand{\sJ}{J}
\newcommand{\sig}[1]{\prescript{\sigma}{}{#1}}

\newcommand{\rss}{\textnormal{rss}}

\newcommand{\tf}{\omega} %transfer factor notation

\DeclareMathOperator{\Nm}{Nm}
\DeclareMathOperator{\Orb}{Orb}

\DeclareMathOperator{\diag}{diag}

\usepackage[foot]{amsaddr}

\title[Twisted GGP conjecture for unramified quadratic extensions]{Twisted Gan--Gross--Prasad conjecture for unramified quadratic extensions}
\author{Danielle Wang}
\email{daniellew@berkeley.edu}
\address{Department of Mathematics, University of California, Berkeley, Berkeley, CA 94720}
\subjclass{Primary 11F67, 11F70; Secondary 22E55}
\date{\today}

\begin{document}

\begin{abstract}
Using a relative trace formula approach, we prove the twisted
global Gan--Gross--Prasad conjecture for 
$\U(V) \subseteq \GL(V)$, as well as its refinement,
under some unramifiedness assumptions and local conditions
on the quadratic extension and the automorphic representation.
\end{abstract}

\maketitle

\tableofcontents

\section{Introduction}
In \cite{MR3202556}, Gan, Gross, and Prasad formulated a
series of conjectures relating period integrals of 
automorphic forms on classical groups to the central values
of certain $L$-functions. In the unitary groups case,
one considers the diagonal embedding $\U(W) \hookrightarrow
\U(V) \times \U(W)$ for a pair of Hermitian or skew-Hermitian spaces
$W \subseteq V$.

In the Hermitian case, for $\dim W = \dim V - 1$, 
Jacquet and Rallis \cite{MR2767518}
proposed a relative trace formula approach to the Gan--Gross--Prasad conjecture.
In \cite{MR3159075}, Harris formulated a refinement of the conjecture in the style of Ichino--Ikeda \cite{MR2585578}.
The GGP conjecture and its refinement in this case 
has then been shown by the work of
% Jacquet--Rallis \cite{MR2767518}, 
% who proposed a relative trace formula approach to the conjecture, 
Zhang 
\cite{MR3245011, MR3164988}, Beuzart-Plessis--Liu--Zhang--Zhu \cite{MR4298750}, and Beuzart-Plessis--Chaudouard--Zydor
\cite{MR4426741}.
In the skew-Hermitian case, when $W = V$, a relative trace formula approach
was developed by Liu \cite{MR3244725} and Xue \cite{MR3228451,MR3505397}.

In \cite{MR4622393}, Gan, Gross, and Prasad formulated
a twisted variant of the conjecture in the case $W = V$. Let $E/F$ and $K/F$ be two
quadratic extensions of number fields, and let
$V$ be a skew-Hermitian space over $E$.
Instead of considering $\U(V)$ as a subgroup of 
$\U(V)(F \times F) = \U(V) \times \U(V)$ as in the original GGP conjecture, 
they consider $\U(V)$ as a subgroup
of $\U(V)(K)$.
They also formulate a refined version of the twisted conjecture.

In this paper, we prove some cases of the twisted GGP conjecture for $E = K$ when $E/F$ is an unramified extension, under some additional unramifiedness assumptions and
local conditions which we describe below.
We also obtain a refined version of our main theorem under these
assumptions.

\subsection{Main result}
Let $E/F$ be a quadratic extension of
number fields. Let $\sigma$ denote the generator
of $\Gal(E/F)$, whose action we will also denote by
$\ol{\phantom{x}}$.
Let $F^-$ be the set of purely imaginary elements of
$E$, and let $j$ be an element of $F^-$

Let $\eta = \otimes \eta_v$ be the quadratic character
of $\AA_F^\times/F^\times$ associated to $E/F$ by global
class field theory, and let $\mu = \otimes \mu_v$ be a character of
$\AA_E^\times/E^\times$ such that $\mu|_{\AA_F^\times}
= \eta$.
Fix a nontrivial additive charater
$\psi' = \otimes \psi_v'$ of $\AA_F/F$.

Let $(V, \langle -, - \rangle_V)$ be an $n$-dimensional
skew-Hermitian space over $E$, and let 
$\U(V)$ be the attached unitary group. Let 
$\Res_{E/F} V$ be the symplectic space over $F$ of dimension
$2n$, 
with symplectic form $\Tr_{E/F}(\langle -, - \rangle_V)$.
Let $\LL$ be a Lagrangian of $(\Res_{E/F} V)^\vee$.
Let $\omega_{\psi', \mu}$ be the Weil representation
of $\U(V)$, which is realized on $\S(\LL(\AA_F))$. 
For $\phi \in \S(\LL(\AA_F))$, define the theta series
\[
	\Theta_{\psi', \mu}(h, \phi) = \sum_{x \in \LL(F)}
	(\omega_{\psi', \mu}(h)\phi)(x).
\]
%It is a function on $\U(V)(F)\bs \U(V)(\AA_F)$.

Fix an isomorphism $V \cong E^n$, and
consider the embedding $\U(V) \subseteq
\Res_{E/F} \GL_n$. Let $\pi$ be an irreducible cuspidal
automorphic representation of $\GL_n(\AA_E)$.
For $\varphi \in \pi$, $\phi \in \S(\LL(\AA_F))$,
define the period integral
\[
	\P(\varphi, \phi) = \int\limits_{\U(V)(F)\bs \U(V)(\AA_F)}
	\varphi(h) \ol{\Theta_{\psi', \mu}(h, \phi)} \, dh.
\]

We are interested in whether $\P(\varphi, \phi)$ gives
a nonzero element of $\Hom_{\U(V)(\AA_F)}(\pi \otimes \ol{\omega_{\psi',
\mu}}, \CC)$.

\begin{remark}
Jacquet has studied the unitary periods
\[
	\int\limits_{\U(V)(F) \bs \U(V)(\AA_F)} \varphi(h)\, dh.
\]
%without the theta series in our period integral. 
We say that $\pi$ is \emph{distinguished by} $\U(V)$ if there exists
$\varphi \in \pi$ such that the above period integral does not vanish.
In \cite{MR2172953} and
\cite{MR2733072}, Jacquet showed that $\pi$ is distinguished by some
unitary group if and only if it is the base change of a
cuspidal automorphic representation of $\GL_n(\AA_F)$,
and that if $\pi$ is a base change, then it
is in fact distinguished by the quasi-split unitary group.
\end{remark}

\begin{conjecture}[{\cite[Conjecture~2.2]{MR4622393}}]
\label{conj:unrefined}
The period integral $\P(\varphi, \phi)$ is not identically
zero if and only if
\begin{enumerate}[\normalfont(1)]
\item $\Hom_{\U(V)(F_v)}(\pi_v \otimes \ol{\omega_{\psi_v', \mu_v}}, \CC) \neq 0$ for all places $v$ of $F$.
\item $L(\frac12, \pi \times \sig{\pi^\vee} \otimes \mu^{-1}) \neq 0$.
\end{enumerate}
Furthermore, if $L(\frac12, \pi \times \sig{\pi^{\vee}} \otimes \mu^{-1}) \neq 0$, then there exists a unique
skew-Hermitian space $V$ of dimension $n$ over $E$ such that
$\P$ is nonzero.
\end{conjecture}

Our main theorem is the following.

\begin{restatable}{theorem}{mainthm}
\label{thm:maintheorem}
Let $V$ be the split skew-Hermitian space of dimension
$n$ over $E$.
Assume
\begin{enumerate}[\normalfont(1)]
\item
\label{item:unr}
The extension $E/F$ is unramified everywhere.

\item
\label{item:unrpi} 
The representation $\pi_v$
is unramified for all inert places $v$ of $F$.

\item 
\label{item:unrchars}
For every inert place $v$ of $F$, the
character $\mu_v$ is unramified, $\psi_v'$ has
conductor $\O_{F_v}$, and $j \in \O_{E_v}^\times$.

\item Every archimedean place of $F$ and
every place of $F$ lying above $2$ is split in $E$.

\item \label{item:scusp}
There exist two nonarchimedean split places
$v_1$, $v_2$ of $F$ such that $\pi_{v_1}$
and $\pi_{v_2}$ are supercuspidal.

\end{enumerate}
Then the period integral $\P(\varphi, \phi)$ is not identically zero
if and only if 
$L(\frac12, \pi \times \prescript{\sigma}{}{\pi^\vee}
\otimes \mu^{-1}) \neq 0$.
\end{restatable}

\subsection{Refined twisted GGP conjecture}
There is also a refined version of the twisted GGP conjecture,
formulated in \cite{MR4622393}, 
and we can also prove a refined version of Theorem \ref{thm:maintheorem}.

In order to state the refined conjecture, we fix a decomposition
$dh = \prod_v dh_v$ of the Tamagawa measure on $\U(V)(\AA_F)$,
for Haar measures $dh_v$ on $\U(V)(F_v)$, and let
$dg$ be the Tamagawa measure on $\GL_n(\AA_E)$.
%and a decomposition
%$dg = \prod_v dg_v$ of the Tamagawa measure on $\GL_n(\AA_E)$,
%for Haar measures $dg_v$ on $\GL_n(E_v)$.

We define the Petersson inner product 
$\langle -, - \rangle_{\Pet}$ on $\pi$ by
\[
	\langle \varphi, \varphi'\rangle_{\Pet}
	= \int\limits_{\GL_n(E)\bs\GL_n(\AA_E)^1} \varphi(g)\ol{\varphi'(g)}
	\,dg,
\]
where $\GL_n(\AA_E)^1$ consists of the elements 
$g \in \GL_n(\AA_E)$ with $|\det g| = 1$.
Fix invariant inner products 
$\langle -, -\rangle_v$ on $\pi_v$
for each $v$ such that $\langle -, - \rangle_{\Pet}
= \prod_v \langle -, -\rangle_v$.

Finally, fix an isomorphism $\LL \cong F_n$, and for each $v$,
let $dx_v$ be the self-dual measure on $\LL(F_v)$ with
respect to $\psi_v'$. We have the inner product on the local
Weil representation $\omega_{\psi'_v, \mu_v}$ of $\U(V)(F_v)$,
which is realized on $\S(\LL(F_v))$, given by
\[
	\langle \phi_v, \phi_v' \rangle_v = \int\limits_{\LL(F_v)}
	\phi_v(x)\ol{\phi_v'(x)}\, dx_v.
\]

\begin{definition}
For $\varphi_v, \varphi_v' \in \pi_v$ and
$\phi_v, \phi_v' \in \S(\LL(F_v))$, define
\[
	\alpha_v(\varphi_v, \phi_v, \varphi_v', \phi_v') =
	\int\limits_{\U(V)(F_v)} \langle \pi_v(h)\varphi_v, \varphi_v'
	\rangle_v \langle \omega_{\psi'_v, \mu_v}(h_v) \phi_v,
	\phi_v'\rangle_v \, dh.
\]
% For convenience, we also 
% let $\alpha_v(\varphi_v, \phi_v) = 
% \alpha_v(\varphi_v, \phi_v, \varphi_v, \phi_v)$.
For $\pi_v$ tempered, this integral converges absolutely
(see Proposition \ref{prop:alphaconvergence}).
\end{definition}

Define the quotient of $L$-functions 
\[
	\L(s, \pi) = \prod_{i = 1}^n \frac{\zeta_E(s + i - \frac12)}{L(s + i - \frac12, \eta^i)} \cdot
	\frac{L(s, \pi \times \sig{\pi^\vee} \otimes \mu^{-1})}{L(s + \frac12, \pi, \Ad)}.
\]
Note that both the numerator and denominator of $\L(s, \pi)$
have a simple pole at $s = \frac12$, so $\L(s, \pi)$ is holomorphic
at $s = \frac12$.
Let $\L(s, \pi_v)$ denote the corresponding quotient
of local $L$-factors.

\begin{conjecture}
\label{conj:unramified}
Let $v$ be a nonarchimedean place of $F$.
Assume that the extension $E_v/F_v$ is unramified, 
the Hermitian space $V$ is split at $v$, 
the conductor of $\psi_v'$ is $\O_{F_v}$, and $\mu_v$
is unramified. 
Let $\varphi_v$ be the $\GL_n(\O_{E_v})$-fixed vector of 
$\pi_v$ such that
$\langle \varphi_v, \varphi_v \rangle_v = 1$, and let
$\phi_v$ be the $\U(V)(\O_{F_v})$-fixed vector of
$\S(\LL(F_v))$ such that $\langle \phi_v, \phi_v \rangle_v = 1$. Then 
\[
	\alpha_v(\varphi_v, \phi_v, \varphi_v, \phi_v) = \L(\tfrac12, \pi_v) \cdot
	\Vol(\U(V)(\O_{F_v})).
	% \prod_{i=1}^n
	% \frac{\zeta_{E_v}(i)}{L(i, \eta_v^i)}\cdot\frac{L(\tfrac12, \pi_v
	% \times \prescript{\sigma}{}{\pi_v^\vee} \otimes \mu_v^{-1})}{L(1, \pi_v, \Ad)}.
\]
\end{conjecture}
%Conjecture \ref{conj:unramified} is a work in progress by Evan Chen.
We also define the normalized local period integral
\[
	\alpha_v^\natural(\varphi_v, \phi_v, \varphi_v', \phi_v') = \frac{1}{\L(\frac12, \pi_v)}\cdot \alpha_v(\varphi_v, \phi_v, \varphi_v', \phi_v').
\]

\begin{conjecture}[{\cite[Conjecture~2.7]{MR4622393}}]
\label{conj:refined}
Assume that $\pi$ is a tempered cuspidal automorphic
representation of $\GL_n(\AA_E)$. Then for
$\varphi = \otimes_v \varphi_v \in \pi$ and
$\phi = \otimes_v\phi_v \in \S(\LL(\AA_F))$, we have
\[
	|\P(\varphi, \phi)|^2 = \L(\tfrac12, \pi) \prod_v \alpha_v^\natural(\varphi_v, \phi_v, \varphi_v, \phi_v).
\]
\end{conjecture}

We have the following refined version of Theorem
\ref{thm:maintheorem}.

\begin{restatable}{theorem}{refined}
\label{thm:refined}
Let $V$ be the split skew-Hermitian space of dimension $n$
over $E$. Assume that $\pi$ is tempered.
Also assume conditions (1)-(5) of 
Theorem \ref{thm:maintheorem}.
If Conjecture \ref{conj:unramified} is true,
then for $\varphi = \otimes_v \varphi_v \in \pi$ and
$\phi = \otimes_v \phi_v \in \S(\LL(\AA_F))$,
we have
\[
	|\P(\varphi, \phi)|^2 =
	\L(\tfrac12, \pi) \prod_v \alpha_v^\natural(\varphi_v, \phi_v, \varphi_v, \phi_v).
\]
\end{restatable}

\subsection{Outline of proof of Theorem \ref{thm:maintheorem}}
Theorem \ref{thm:maintheorem} is proved by comparing two
relative trace formulas, which we describe below, as 
well as a novel aspect of the trace formulas.

Let 
\[		
	G = \Res_{E/F}\GL_n,\quad H = \U(V).
\]
Let $\C_c^\infty(G(\AA_F))$ be the space of compactly
supported smooth functions on $G(\AA_F)$, and let
$\S(\LL(\AA_F))$ be the space of Schwartz functions on
$\LL(\AA_F)$. 
Let  
$f \in \C_c^\infty(G(\AA_F))$, $\phi_1 \otimes \phi_2 \in
\S(\LL(\AA_F))^{\otimes 2}$. Define
\[
	K_f(g, h) = \sum_{\zeta \in G(F)} f(g^{-1}\zeta h).
\]
To study the period integral $\P$, we
consider the distribution
\[
	\sJ(f, \phi_1 \otimes \phi_2) = \iint\limits_{(H(F)\bs H(\AA_F))^2}
	K_f(g, h) \ol{\Theta_{\psi', \mu}(g, \phi_1)}
	\Theta_{\psi', \mu}(h,
	\phi_2) \, dg\, dh.
\]

On the other hand, let
\begin{gather*}
	G' = \Res_{E/F}(\GL_n \times \GL_n),\\
	H_1 = \Res_{E/F}\GL_n,\quad
	H_2 = \{(g, \ol g) : g \in \Res_{E/F}\GL_n\},
\end{gather*}
with  $H_1$ embedded diagonally into $G'$ 
and $H_2$ embedded componentwise.
Let $f' \in \C_c^\infty(G'(\AA_F))$, $\phi' \in \S(\AA_{E,n})$. (Here, $\AA_{E,n}$ denotes the row vectors of dimension
$n$.) Define
\[
	K_{f'}(g, h) = \sum_{\gamma \in G'(F)} f'(g^{-1}\gamma h).
\]
We also define
\[
	(R_{\mu}(g)\phi')(z) = \mu(\det g)|\det g|^{\frac12}\phi'(zg),
\]
and the theta series
\[
	\Theta_{\mu}(g, \phi') = 
	\sum_{z \in E_n} (R_\mu(g)\phi')(z).
\]
To study the $L$-value $L(\frac12, \pi \times \sig{\pi^\vee} \otimes \mu^{-1})$, we consider the distribution
\[
	\sI(f', \phi') = \int\limits_{H_2(F)\bs H_2(\AA_F)}\int\limits_{H_1(F)\bs H_1(\AA_F)} K_{f'}(g, h)
	\ol{\Theta_{\mu}(g, \phi')}
	\, dg\, dh.
\]
Both of these distributions have a geometric and
spectral expansion. 

The new feature of these relative trace formulas is that
the distributions do not decompose as sums over the orbits
of an $H_1 \times H_2$ group action.
Due to the theta series factor, the geometric decompositions
of $I$ and $J$ involve partial Fourier transforms
\[
	^\dagger \colon \S(\AA_{E,n}) \to \S(\AA_{F,n} \times \AA_F^{-,n})
\]
and 
\[
	^\ddagger \colon \S(\LL(\AA_F))^{\otimes 2} \to \S(V^\vee(\AA_E)),
\]
as in \cite{MR3244725} (for their precise definitions, see
Definitions \ref{def:pftgl} and \ref{def:pftu}).
However, the group $H_1(F) \times H_2(F)$ 
does not act on 
$G'(F) \times F_n \times F^{-,n}$.
The new idea of this paper is to choose
representatives $\gamma$ of 
\[
	H_1(F) \bs G'(F) / H_2(F)
\]
such that the stabilizer of $\gamma$ in $H_1(F) \times H_2(F)$ 
acts on $F_n \times F^{-,n}$.
Similarly, we choose representatives $\zeta$ of 
\[
	H(F) \bs G(F) / H(F)
\]
whose stabilizers in $H(F) \times H(F)$ act on
$V^\vee$.
The orbital integrals obtained in the geometric decomposition
contain a factor involving the partial Fourier transform
of a Weil representation, and do depend on the choice
of representatives.

One of our key local results is the following fundamental lemma.
Let $v$ be a place of $F$. Then we have
\begin{align*}
	&\int_{\GL_n(E_v)} \widetilde f_v'(g^{-1}\gamma \ol g)
	\ol{(R_{\mu_v}(g)\phi_v')^\dagger(x, y)}\, dg \\
	&= \omega([\ol\gamma\gamma, x, y])
	\iint\limits_{(\U(V)(F_v))^2}
	f_v(g^{-1}\zeta h) (\phi_{2,v} \otimes \ol{\omega_{\psi_v',\mu_v}(h^{-1}g)
	\phi_{1,v}})^\ddagger(z h)\, dg\, dh
\end{align*}
for
\[
	\widetilde f_v' = \mathbbm{1}_{\GL_n(\O_{E_v})},
	\quad \phi_v' = \mathbbm{1}_{\O_{E_v, n}},
	\quad f_v = \mathbbm{1}_{\GL_n(\O_{E_v})},
	\quad \phi_{1,v} = \phi_{2,v} = \mathbbm{1}_{\LL(\O_{F_v})},
\]
and $[\gamma, x, y] \in \GL_n(E) \times F_n \times F^{-,n}$
and 
$[\zeta, z] \in \GL_n(E) \times V^\vee$ which match in the
sense of Lemma \ref{lem:matching}.
Here, $\omega([\ol\gamma\gamma, x, y])$ is a transfer factor. For its precise
definition, see Section \ref{sec:transfer};
for the complete statement of the fundamental lemma, 
see Section \ref{sec:FL}.

In the above equation, $\gamma$ and $\zeta$ belong to
our sets of double coset representatives, once we
identify $H_1\bs G'(F)/H_2(F)$ with the $\sigma$-conjugacy
classes of $\GL_n(E)$.
We prove this fundamental lemma by choosing representatives globally so that 
at all inert $v$, it reduces to
the Jacquet--Rallis fundamental lemma \cite{MR2767518}, 
%which has been proven for all odd places in \cite{MR4305382}.
which has been shown by Yun and Gordon
\cite{MR2769216} when the residue characteristic of
$F$ is large, and in full by Beuzart-Plessis \cite{MR4305382} via a purely local proof. A global proof is given
for residue characteristic greater than $n$
in W. Zhang's work \cite{MR4250392} on the arithmetic fundamental lemma, and
in general by Z. Zhang \cite{zhang2021maximal}.

Although we have not shown the existence of smooth transfer in general, transfer at split places of $F$ can 
be made explicit through a direct computation.
Since we assume in Theorem \ref{thm:maintheorem} that for all places $v$ of $F$, either $v$ is split or $v$ is inert and all data is unramified, this allows us to deduce the theorem. 
Assumption \eqref{item:scusp} is used to simplify the
relative trace formula.

In fact, the unramifiedness assumptions of our theorem
are not so easy to remove.
In order to make the aforementioned reduction, $\gamma$ and $\zeta$ are chosen so that the Weil representation factor
can be ignored, and the orbital integrals reduce to the 
orbital integrals in the Jacquet--Rallis relative trace formula.
(The properties we require of $\gamma$ and $\zeta$
are essentially the same as those in \cite[Lemma~8.8]{MR564478}.)
However, it does not appear that in general the orbital
integrals can be simplified in a similar way, 
which is why it appears to be difficult to
obtain the full transfer statement.

\begin{remark}
% Note that the phenomenon described above does not cause as much difficulty
% in the Fourier--Jacobi case, since in that case there is
% a natural choice of representatives of 
% $\U(V) \times \U(V) / \U(V)$, namely the elements of the
% form $(1, \zeta)$,
% and a similar choice can be made on the general linear side.
% Since we do not have such a nice set for $G(F) / H(F)$,
% the unramifiedness assumptions in our theorem are in fact not so easy to remove. We remark that this is the first time such
% a difficulty has occurred in a relative trace formula.
The phenomenon described above does not cause
as much difficulty
in the Fourier--Jacobi case, because in that case there is
a natural section of the quotient map
$\U(V) \times \U(V) \to \U(V) \times \U(V) / \U(V)$.
We do not have such a nice choice for $G \to G/H$.
We remark that this is the first time such
a difficulty has occurred.
\end{remark}

\subsection{Outline of proof of Theorem \ref{thm:refined}}
We also outline the proof of Theorem \ref{thm:refined},
which is also based on the relative trace formulas described 
above.

First, we expect that $\alpha_v$ gives a nonzero element of 
\[
	\Hom_{\U(V)(F_v)}(\pi_v \otimes \ol{\omega_{\psi_v', \mu_v}} \otimes \ol{\pi_v} \otimes \omega_{\psi_v', \mu_v}, \CC)
\]
if and only if 
$\Hom_{\U(V)(F_v)}(\pi_v \otimes \ol{\omega_{\psi_v', \mu_v}}, \CC) \neq 0$.
We have not shown this in general.
However, under the assumptions of Theorem \ref{thm:refined}, we do know that $\alpha_v \neq 0$ for each place $v$ of $F$.
% (At split places $v$, the nonvanishing of $\alpha_v$
% follows from the nonvanishing of the local Rankin--Selberg
% period --- see Section \ref{sec:unitaryperiods}.)
Then since 
$\P \otimes \ol{\P}$ is an element of 
\[
	\Hom_{\U(V)(\AA_F)}
	(\pi \otimes \ol{\omega_{\psi', \mu}} \otimes \ol{\pi} \otimes
	\omega_{\psi', \mu}, \CC),
\]
there exists a constant $C$ such that
\[
	|\P(\varphi, \phi)|^2 = C \times
	\prod_v \alpha_v^\natural(\varphi_v, \phi_v, \varphi_v, \phi_v)
\]
for all $\varphi = \otimes_v \varphi_v$,
$\phi = \otimes_v \phi_v$. 

We consider the global distribution
\[
  \sJ_\pi(f, \phi_1 \otimes \phi_2) =
  \sum_{\varphi}
  \P(\pi(f)\varphi, \phi_1)\ol{\P(\varphi, \phi_2)},
\]
for $f \in \C_c^\infty(G(\AA_F))$ and $\phi_1 \otimes \phi_2
\in \S(\LL(\AA_F))^{\otimes 2}$,
where the sum runs over an orthonormal basis of $\pi$.
We also consider the local distribution
\begin{align*}
	J_{\pi_v}(f_v, \phi_{1,v} \otimes \phi_{2,v})
	= \sum_{\varphi_v}
	\alpha_v(\pi_v(f_v)\varphi_v, \phi_{1,v}, \varphi_v, \phi_{2,v}),
\end{align*}
as well as its normalized version
\[
	J_{\pi_v}^\natural(f_v, \phi_{1,v} \otimes \phi_{2,v})
	= \sum_{\varphi_v}
	\alpha_v^\natural(\pi_v(f_v)\varphi_v, \phi_{1,v}, \varphi_v, \phi_{2,v}),
\]
where the sums above run over an orthonormal basis of $\pi_v$.

\begin{conjecture}
\label{conj:refinedJ}
Assume that $\pi$ is a tempered automorphic representation
of $\GL_n(\AA_E)$.
If $f = \otimes_v f_v$, $\phi_1 = \otimes_v \phi_{1,v}$, and
$\phi_2 = \otimes_v \phi_{2,v}$ are all factorizable, then
\[
	\sJ_{\pi}(f, \phi_1 \otimes \phi_2)
	= \L(\tfrac12, \pi)
	\prod_v \sJ_{\pi_v}^\natural(f_v, \phi_{1,v} \otimes \phi_{2,v}).
\]
\end{conjecture}{}

At least when we know the nonvanishing of $\alpha_v^\natural$,
Conjecture \ref{conj:refinedJ} can be seen to be equivalent
to Conjecture \ref{conj:refined}.

We also have analogous global and local distributions
$\sI_{\Pi}$, $\sI_{\Pi_v}$, and $\sI_{\Pi_v}^\natural$
on the general linear group, where $\Pi \coloneqq \pi \boxtimes \sig{\pi^\vee}$ is a representation of $G'(\AA_F)$.
The global distribution $\sI_{\Pi}$ 
decomposes as a product of the normalized local distributions
$\sI_{\Pi_v}^\natural$ is a manner similar to Conjecture
\ref{conj:refinedJ}. Such a decomposition follows from known
results on the Rankin--Selberg convolution.

The relative trace formula comparison gives
an identity between the global distributions
$\sJ_{\pi}$ and $\sI_{\Pi}$. Thus, it reduces our problem to
proving a relation between
the local distributions $\sJ_{\pi_v}$ and $\sI_{\Pi_v}$.
Because of the assumptions in our theorem, we only have to
consider split places $v$ and places where all data is
unramified, so the local distribution identity we require
will also follow easily from Conjecture \ref{conj:unramified}
and results already shown in \cite{MR3505397}.

\subsection{Acknowledgments}
The author would like to thank her advisor Wei Zhang for suggesting
the problem, and for his guidance throughout the whole process.
She would also like to thank Hang Xue for going through many
of the details with her. Finally, she would like to thank Wee Teck Gan,
Weixiao Lu, and Zhiyu Zhang for their discussions and comments.
This material is based upon work supported by the National Science Foundation Graduate Research Fellowship under Grant Nos.\
1745302 and 2141064.
Part of this work was supported by the National Science Foundation under Grant No.\ 1440140, while the author was in residence at the Simons Laufer Mathematical Sciences Institute
in Berkeley, California, during the Spring Semester of 2023.

\subsection{Notation and measures}
For the rest of this paper, let $E/F$ be an unramified 
quadratic extension of number fields. 
Let $\sigma$ denote the generator of $\Gal(E/F)$.
We also denote the action of $\sigma$ by $\ol{\phantom{x}}$.
Let $F^-$ be the set of purely imaginary elements of
$E$, and fix an element $j$ of $F^-$.

Let $\eta = \otimes \eta_v$ be the quadratic character
of $\AA_F^\times/F^\times$ associated to $E/F$ by global
class field theory, and let $\mu = \otimes \mu_v$ be a character of
$\AA_E^\times/E^\times$ such that $\mu|_{\AA_F^\times}
= \eta$. 

Let $\psi' = \otimes \psi'_v$ be a nontrivial additive character
of $\AA_F/F$, and extend it to $\psi\colon \AA_E/E \to \CC^\times$ by $\psi(x) = \psi'(\frac12\Tr(x))$,
where $\Tr$ is the trace of $\AA_E/\AA_F$. 

Let $\Mat_{m,n}$ denote the affine group scheme of $m
\times n$ matrices. For any commutative ring $R$, let
$R^n = \Mat_{n, 1}(R)$, $R_n = \Mat_{1, n}(R)$.
Let $\Mat_n = \Mat_{n,n}$.
For $g \in \Mat_{n}(R \otimes_F E)$ for
any $F$-algebra $R$, let $g^\ast = \transp{(\sig{g})}$ denote the conjugate transpose. 

Let $\SHerm_n^\times$ be the subvariety of skew-Hermitian 
matrices in $\Res_{E/F}\GL_n$.
For $\beta \in \SHerm_n^\times(F)$ (resp. $\SHerm_n^\times(F_v)$), let 
\[
	\U_n^\beta(R) \coloneqq \{h \in \GL_n(R \otimes_F E) : h^\ast
	\beta h = \beta \},
\]
for any $F$-algebra (resp. $F_v$-algebra) $R$.

Let $[\SHerm_n^\times(F)]$ denote the set 
of similarity classes in $\SHerm_n^\times(F)$,
where $\beta$ and $\beta'$ are similar 
if $\beta' = a^\ast\beta a$ for some $a \in \GL_n(E)$.

Let $\C_c^\infty(M)$ denote the compactly supported
smooth functions on $M$, for a smooth manifold $M$.
%\todo{But where do we use smoothness?}

For an algebraic variety $X$ over $F_v$, let $\S(X(F_v))$ denote the space of Schwartz functions on $X(F_v)$.
For $F_v$ nonarchimedean, these spaces are the same.

For an algebraic variety $X$ over $F$,
fix an integral model of $X$ over $\O_F$ away from
finitely many places.
Define $\S(X(\AA_F)) =
\bigotimes_v' \S(X(F_v))$ with respect to the 
characteristic functions $\mathbbm{1}_{X(\O_{F_v})}$.

For a reductive group $G$ acting on an affine
variety $X$, we say that $x \in X$ is 
semisimple if its orbit is closed, and 
regular if its stabilizer is of minimal dimension.
%Let $X_{\rss}$ denote the regular semisimple locus.

\subsubsection{Measures on additive groups}
Let $v$ be a place of $F$. For $\phi \in \S(F_v)$, we have
the Fourier transform
\[
	\widehat\phi(y) = \int_{F_v} \phi(x)\psi_v'(xy)\,dx.
\]
Let $dx$ be the self-dual measure on $F_v$, i.e., the 
Haar measure on $F_v$ such that 
$\skew{3.3}\widehat{\widehat{\phi}}(x) = \phi(-x)$. 
Fix the measure on $\Mat_{m,n}(F_v)$ via the identification
$\Mat_{m,n}(F_v) \simeq F_v^{mn}$.
It is also the self-dual measure on $\Mat_{m,n}(F_v)$ with
respect to the Fourier transform
\[
	\widehat\phi(y) = \int\limits_{\Mat_{m,n}(F_v)}\phi(x)\psi_v'(\Tr xy)\, dx.
\]
We similarly define measures on $\Mat_{m,n}(E_v)$, using
the additive character $\psi_v$.
We normalize the absolute values on $F_v$ and $E_v$ so
that 
\[
	d(ax) = |a|_{F_v}\,dx
\] 
and similarly for $E_v$.
We take the product measure $dx = \prod_v dx_v$ on $\Mat_{m,n}(\AA_F)$, 
and similarly for $\Mat_{m,n}(\AA_E)$.

We also fix an isomorphism $\LL \simeq F_n$
of $F$-vector spaces, and for each $v$, we define the
measure on $\LL(F_v)$ to be the pullback of the self-dual
measure $dx$ on $F_{v,n}$, and we take the product measure on
$\LL(\AA_F)$.

\subsubsection{Measures on the general linear group}
Let $v$ be a place of $F$. Let $d^\times x$ be the 
normalized multiplicative measure on $F_v^\times$, given by
\[
	d^\times x = \zeta_{F_v}(1)\frac{dx}{|x|_{F_v}}.
\]
On $\GL_n(F_v)$, define the measure
\[
	dg = \zeta_{F_v}(1)\frac{\prod_{ij}dx_{ij}}{|\det g|_{F_v}^n}, \quad g = (x_{ij}).
\]
When $\psi_v'$ is unramified, the subgroup
$\GL_n(\O_{F_v})$ has measure
$\zeta_{F_v}(2)^{-1}\cdots \zeta_{F_v}(n)^{-1}$ under
$dg$.

Let $B_n$ be the standard upper triangular Borel subgroup
of $\GL_n$, and let $N_n$ be its unipotent radical.
We take the measure
\[
	du = \prod_{1 \le i < j \le n} du_{ij}, \quad
	u = (u_{ij})
\]
on $N_n(F_v)$.

Let 
\[ dg = \zeta_F^\ast(1)^{-1}\prod_v dg_v \] 
be the Tamagawa measure on $\GL_n(\AA_F)$, where
$\zeta_F^\ast(1)$ is the residue of the pole at $s = 1$ 
of $\zeta_F(s)$.
Similarly, we 
take the product measure on $N_n(\AA_F)$.
We also fix the measures on $\GL_n(E_v)$, $\GL_n(\AA_E)$,
$N_n(E_v)$, $N_n(\AA_E)$ in the same way.

\subsubsection{Measures on the unitary group}
Let $V$ be a skew-Hermitian space over $E$, and let
$\U(V)$ be its unitary group. Let $\fu(V)$ be the Lie algebra
of $\U(V)$. 
Let $dX$ be the self-dual measure on $\fu(V)(F_v)$ for the
Fourier transform
\[
	\widehat\phi(Y) = \int_{\fu(V)(F_v)}\phi(X) \psi_v'(\Tr XY)\, dX,
\]
and similarly define the self-dual measure on $\fu(V)(\AA_F)$.

Let $\omega$ be a top invariant differential form on $\U(V)$
whose pullback via the Cayley transform
\[
	X \mapsto (1 + X)(1 - X)^{-1} \colon \fu(V) \to \U(V)
\]
gives rise to the self-dual measure $dX$ on $\fu(V)$.
We take the measure $|\omega|_{F_v}$ on $\U(V)(F_v)$.
When $\psi_v'$ is unramified and $\U(V)$ is unramified at $v$,
the subgroup $\U(V)(\O_{F_v})$ has measure
$L(1, \eta_v)^{-1} \cdots L(n, \eta_v^n)^{-1}$.
Finally, we define the measure on $\U(V)(\AA_F)$ by
\[
	dh = \prod_v |\omega|_{F_v}.
\]

\section{Relative trace formula on the general linear
group}
\subsection{Orbital integrals on the general linear group}

Let 
\[ G' = \Res_{E/F}(\GL_n \times \GL_n), \]
and let
$H_1 = \Res_{E/F}\GL_n$ and 
$H_2 = \{(g, \ol g) : g \in \Res_{E/F}\GL_n\}$ be its subgroups, with 
$H_1$ embedded diagonally into $G'$, 
and $H_2$ embedded componentwise.

\begin{definition}
\label{def:norm}
Define the norm map $\Nm \colon \GL_n(E) \to \GL_n(E)$ 
by $\Nm(\gamma) = \ol\gamma\gamma$.
The group $\GL_n(E)$ acts on itself from the right by
$\sigma$-conjugation 
\[ \gamma.g = g^{-1}\gamma\ol g. \]
By \cite[Lemma~1.1]{MR1007299},
the norm map induces an injection from $\sigma$-conjugacy
classes in $\GL_n(E)$ to conjugacy classes in $\GL_n(F)$.

Finally, an element $\gamma$ of $\GL_n(E)$ is \emph{normal} 
if $\ol\gamma\gamma \in \GL_n(F)$.
\end{definition}

% \begin{definition}
% \label{def:glnormal}
% An element $\gamma$ of $\GL_n(E)$ is \emph{normal} 
% if $\ol\gamma \gamma \in \GL_n(F)$.
% \end{definition}

\begin{lemma}
Every $\sigma$-conjugacy class of $\GL_n(E)$ contains a normal element. 
\end{lemma}
\begin{proof}
This follows from \cite[Lemma~1.1]{MR1007299}, 
which tells us that for every $\gamma \in
\GL_n(E)$, there exists $g \in \GL_n(E)$ such that
$g^{-1}\ol\gamma\gamma g \in \GL_n(F)$. Then $\ol{g}^{-1}
\gamma g$ is a normal element contained in the $\sigma$-conjugacy class
of $\gamma$.
\end{proof}

\begin{lemma}
\label{lem:glnormal}
If $\gamma \in \GL_n(E)$ is normal and
$\ol\gamma\gamma$ is regular semisimple, then 
\[
	T_\gamma \coloneqq \{ g \in \GL_n(E) : g^{-1}\gamma
	\ol g = \gamma \}
\]	
is contained in $\GL_n(F)$. In fact it is equal to the centralizer
of $\ol\gamma\gamma$ in $\GL_n(F)$.
\end{lemma}

\begin{proof}
Note that if $g \in T_\gamma$, then $g \in C_{\GL_n(E)}(\ol\gamma\gamma)$. In addition, $\gamma \in C_{\GL_n(E)}(\ol\gamma
\gamma)$ since $\ol\gamma\gamma \in \GL_n(F)$. 
Since $\ol \gamma\gamma$ is regular semisimple, its centralizer
is a torus, so $\gamma$ commutes with $g$. 
Thus, $g^{-1}\gamma \ol g = \gamma$ implies $g = \ol g$, 
so $g \in C_{\GL_n(F)}(\ol\gamma\gamma)$. Conversely, if $g
\in C_{\GL_n(F)}(\ol\gamma\gamma)$, then $g$ commutes with
$\gamma$, so $g^{-1}\gamma \ol g = g^{-1}\gamma g = \gamma$, so
$g \in T_\gamma$.
\end{proof}

\begin{definition}
\label{def:glorbits}
Fix a set $S$ of normal representatives of the regular
semisimple 
$\sigma$-conjugacy classes in $\GL_n(E)$. Let 
\[
	X = S \times F_n \times F^{-,n}.
\]
Say that an element $[\gamma, x, y] \in X$ is
\emph{regular semisimple} if $[x, y]\in F_n \times F^{-,n}$ is regular semisimple
under the right action of $T_\gamma$ given by
\[
	[x, y].t = [xt, t^{-1}y].	
\]
Two elements $[\gamma, x, y]$ and $[\gamma', x', y']$
of $X$ are \emph{equivalent} if $\gamma = \gamma'$ and
there exists $t \in T_\gamma$ such that
$[x, y].t = [x', y']$.
Let $X_{\rss}$ denote the regular semisimple elements
of $X$, and let $[X_\rss]$ denote the set of regular
semisimple equivalence classes.
\end{definition}

Let $v$ be a place of $F$. The above definitions
can also be made for the local extension $E_v/F_v$.

\begin{definition}
\label{def:weilgl}
Let $R_{\mu_v}$ be the representation
of $\GL_n(E_v)$ on $\S(E_{v,n})$ defined by
\[
	(R_{\mu_v}(g)\phi_v')(z) = \mu_v(\det g)|\det g|^{\frac12}\phi_v'(zg).
\]
\end{definition}

Note that $R_{\mu_v}$ is a unitary representation
when $\mu_v$ is unitary.

\begin{definition}
\label{def:pftgl}
Let $^\dagger$ be the partial Fourier transform
\[
	^\dagger \colon \S(E_{v,n}) \to \S(F_{v,n} \times F_v^{-,n})
\]
given by
\[
	\phi_v'^\dagger(x, y) = |j|_v\int_{F_{v,n}} \phi_v(x + jw)
	\psi_v'(jwy)\, dw.
\]
Recall that the measure is chosen to be the self-dual measure on 
$F_{v,n}$ with respect to $\psi_v'$.
\end{definition}

Note that for $g \in \GL_n(F_v)$, we have
\begin{equation}
\label{eq:pftGLF}
	(R_{\mu_v}(g)\phi_v')^\dagger(x, y)
	= \eta_v(\det g)\phi_v'^\dagger(xg, g^{-1}y).
\end{equation}

\begin{definition}
\label{def:orbintGL}
For $f_v' \in \C_c^\infty(G'(F_v))$, define a function
$\widetilde{f_v'} \in \C_c^\infty(\GL_n(E_v))$ by
\[
	\widetilde{f_v'}(g) = \int_{\GL_n(E_v)} f_v'((g, 1)(h, \ol h)) \, dh.
\]
For 
%$\widetilde{f_v'}\in \C_c^\infty(\GL_n(E_v))$,
$\phi_v' \in \S(E_{v,n})$ and $[\gamma, x, y] \in X$, let
\[
	\Orb_v^{\psi, \mu}(s, [\gamma, x, y], f_v',
	\phi_v') = \int_{\GL_n(E_v)} \widetilde{f_v'}(g^{-1}
	\gamma \ol g)\ol{(R_{\mu_v}(g)\phi_v')^\dagger(x, y) |\det g|^{s-\frac12}} \, dg.
\]
This depends only on the equivalence class of
$[\gamma, x, y]$ in $X$.
\end{definition}

In general, when $s = \frac12$, we will
suppress $s$ from the notation.

\begin{lemma}
\label{lem:glorbintconverge}
For any test function $(f_v', \phi_v')$ 
and any regular semisimple equivalence class 
$[\gamma, x, y]$,
the orbital integral $\Orb_v^{\psi,\mu}([\gamma, x, y], f_v', \phi_v')$
is absolutely convergent.
\end{lemma}
\begin{proof}
Using \eqref{eq:pftGLF}, the orbital integral at 
$s = \frac12$ is equal to
\begin{align*}
\int\limits_{T_\gamma\bs\GL_n(E_v)}
\widetilde{f_v'}(g_0^{-1}\gamma\ol{g_0})
\int_{T_\gamma} \ol{\eta_v(\det t)(R_{\mu_v}(g_0)\phi_v')^\dagger(xt, t^{-1}y)}\,dt\,dg_0.
\end{align*}
Note that for regular semisimple $[\gamma, x, y]$,
the stabilizer of $[x, y]$ in $T_\gamma$ is trivial.
Since the orbit of
$[x, y]$ under the $T_\gamma$-action is closed,
and $(R_{\mu_v}(g_0)\phi_v')^\dagger$
is a Schwartz function on $F_n \times F^{-,n}$,
the function $t \mapsto (R_{\mu_v}(g_0)\phi_v')^\dagger(xt, t^{-1}y)$ is a Schwartz function on $T_\gamma$. Thus the inner integral is absolutely convergent,
and a continuous function of $g_0$.
% \todo{Okay, I don't know what Schwartz function
% means in the archimedean case. Also we need to use 
% some continuity of the Weil representation.}

Similarly, since the orbit of $\gamma$ under 
$\sigma$-conjugation is closed, and $\widetilde{f_v'}$
is compactly supported and smooth on $\GL_n(E_v)$, the
function $g_0 \mapsto \widetilde{f_v'}(g_0^{-1}\gamma\ol{g_0})$ is compactly supported and smooth on 
$T_\gamma \bs \GL_n(E_v)$,
so the outer integral is absolutely convergent.
\end{proof}

\subsection{Global distribution on the general linear group}
We can also define the representation $R_\mu$ of $\GL_n(\AA_E)$
on $\S(\AA_{E,n})$, and the global partial Fourier transform 
$^\dagger\colon \S(\AA_{E,n}) \to \S(\AA_{F,n} \times
\AA_F^{-,n})$
using the same formulas as in Definitions \ref{def:weilgl}
and \ref{def:pftgl}.

\begin{definition}
For $\phi' \in \S(\AA_{E,n})$, define the theta series
\[
	\Theta_{\mu}(s, g, \phi') = |\det g|^{s-\frac12}
	\sum_{z \in E_n} (R_\mu(g)\phi')(z).
\]
%It is a function on $\GL_n(E)\bs\GL_n(\AA_E)$.
\end{definition}

\begin{definition}
For $f' \in \C_c^\infty(G'(\AA_F))$, define the kernel function
\[
	K_{f'}(g, h) = \sum_{\gamma \in G'(F)} f'(g^{-1}\gamma h).
\]
It is a function on $(G'(F)\bs G'(\AA_F))^2$.
\end{definition}

\begin{definition}
Let $f' \in \C_c^\infty(G'(\AA_F))$, $\phi' \in \S(\AA_{E,n})$.
Let
\[
	\sI(s, f', \phi') = \int\limits_{H_2(F)\bs H_2(\AA_F)}\int\limits_{H_1(F)\bs H_1(\AA_F)} K_{f'}(g, h)
	\ol{\Theta_{\mu}(s, g, \phi')}
	\, dg\, dh.
\]
\end{definition}

\begin{definition}
\label{def:glgood}
Say that a test function $(f', \phi')$ is \emph{good}
if it satisfies the following conditions.
\begin{enumerate}[(1)]
\item 
\label{item:glmatrixcoeff}
There exists a nonarchimedean split place $v_1$ of $F$,
such that $f_{v_1}'$ is a truncated matrix 
coefficient of a supercuspidal representation.
This means
\[
	f_{v_1}'(g) = \widehat f_{v_1}'(g) \mathbbm{1}_{G'^\ast(F_{v_1})}(g),
\]
where $\widehat f_{v_1}'$ is a matrix coefficient of a supercuspidal representation of $G'(F_{v_1})$, and
\[
	G'^\ast(F_{v_1}) = \{(g_1, g_2) \in G'(F_{v_1})
	: (\det g_1, \det g_2) \in \O_{E_{v_1}}^\times
	\times \O_{E_{v_1}}^\times \}.
\]

\item 
\label{item:glrsssupp}
There exists a nonarchimedean split place $v_2
\neq v_1$ such that $\widetilde f_{v_2}'$ is supported
on the regular semisimple locus of $\GL_n(E_{v_2})$,
and if $[\gamma, x, y] \in X$ and
\[
	\widetilde f_{v_2}'(g^{-1}\gamma \ol g)\ol{(R_{\mu_{v_2}}(g)\phi_{v_2}')^\dagger(x, y)} \neq 0
\]
for some $g \in \GL_n(E_v)$, then $[\gamma, x, y] \in X_\rss$.

\item For all archimedean places $v$, we have
$E_v = F_v \times F_v$, and $\phi_v'$ is a finite linear
combination of functions of the form 
$\phi'_{1,v} \otimes \phi_{2,v}'$
where $\phi_{1,v}', \phi_{2,v}' \in \S(F_{v,n})$.
% Moreover, $f_v'$ is $K_v$-finite where $K_v$ is a 
% maximal compact subgroup of $G'(F_v)$.
% \todo{Is the first condition of (3) used for split
% smooth transfer? and the second for convergence?}
\end{enumerate}
\end{definition}

\begin{proposition}
\label{prop:GeoGL}
For good test functions $(f', \phi')$, the integral
defining $\sI(f', \phi')$ is absolutely convergent.
Moreover, if
$f' = \prod_v f_v'$ and $\phi' = \prod_v \phi_v'$
are both factorizable, then
\[
	\sI(f', \phi') = 
	\sum_{[\gamma, x, y] \in [X_\rss]} 
	\zeta_E^\ast(1)^{-1} \prod_{v} \Orb_v^{\psi,
	\mu}([\gamma, x, y], f_v', \phi_v').
\]
% where the sum runs over all regular semisimple
% equivalence classes of elements in $X$.
\end{proposition}
\begin{proof}
Let $\widetilde{f'}$ denote the global analog
of the function $\widetilde{f_v'}$ from Definition
\ref{def:orbintGL}.
It is easy to check that
\begin{align*}
	&\int\limits_{\GL_n(E)\bs\GL_n(\AA_E)} 
	\sum_{\gamma \in G'(F)} f'((g^{-1}, g^{-1})\gamma (h, \ol h)) \,dh %\\
	% &= \int\limits_{\GL_n(E)\bs\GL_n(\AA_E)} 
	% \sum_{\gamma \in \GL_n(E)} \sum_{a \in \GL_n(E)} f'((g^{-1}, g^{-1})(\gamma a, \ol a) (h, \ol h)) \,dh \\
	% &= \sum_{\gamma \in \GL_n(E)} \int\limits_{\GL_n(E)\bs \GL_n(\AA_E)} 
	% \sum_{a \in \GL_n(E)} f'((g^{-1}, g^{-1})(\gamma, 1)
	% (ah, \ol{ah})) \, dh \\
	% &= \sum_{\gamma \in \GL_n(E)} 
	% \int_{\GL_n(\AA_E)}
	% f'((g^{-1}, g^{-1})(\gamma, 1)(h, \ol h))\, dh \\
	% &= \sum_{\gamma \in \GL_n(E)} \int\limits_{\GL_n(E)\bs\GL_n(\AA_E)}\sum_{a \in \GL_n(E)} f'((g^{-1}, g^{-1})(\gamma a, \ol a)(h, \ol h))\, dh \\ &
	= \sum_{\gamma \in \GL_n(E)} \widetilde{f'}(g^{-1}
	\gamma \ol g).
\end{align*}
In Definition \ref{def:glorbits}, we fixed a set $S$ 
of normal representatives of the regular semisimple
$\sigma$-conjugacy classes in $\GL_n(E)$.
Using the assumption \eqref{item:glrsssupp} on $(\widetilde f_{v_2}', \phi_{v_2}')$ from Definition \ref{def:glgood}, we have
\begin{align*}
	&\sum_{\gamma \in \GL_n(E)} \widetilde{f'}(g^{-1} \gamma \ol g)
	\ol{\Theta_{\mu}(g, \phi')}\\ 
	&= \sum_{\gamma \in S} \sum_{a \in T_\gamma \bs \GL_n(E)}
	\widetilde{f'}(g^{-1}a^{-1} \gamma \ol a \ol g) 
	\ol{\Theta_{\mu}(ag, \phi')} \\
	&= \sum_{\gamma \in S} \sum_{a \in T_\gamma \bs \GL_n(E)}
	\widetilde{f'}(g^{-1}a^{-1} \gamma \ol a \ol g) 
	\sum_{(x, y) \in F_n
	\times F^{-,n}} \ol{(R_\mu(ag)\phi')^\dagger(x, y)}\\
	&= \sum_{\gamma \in S} \sum_{a \in T_\gamma \bs \GL_n(E)}
	\widetilde{f'}(g^{-1}a^{-1} \gamma \ol a \ol g) \sum_{[x, y] \in (F_n
	\times F^{-,n})/T_\gamma} \sum_{t \in T_{[\gamma, x, y]}\bs T_\gamma} \ol{(R_\mu(tag)\phi')^\dagger(x, y)} \\
	% &= \sum_{\gamma \in S} \sum_{[x, y] \in (F_n \times F^n)/T_\gamma} 
	% \sum_{a \in T_\gamma \bs \GL_n(E)}
	% \widetilde{f'}(g^{-1}a^{-1} \gamma \ol a \ol g)
	% \sum_{t \in \Stab_{[x, y]}\bs T_\gamma} \ol{(R_\mu(tag)\phi')^\dagger(x, y)} \\
	&= \sum_{[\gamma, x, y] \in [X_\rss]} \sum_{a \in \GL_n(E)}
	\widetilde{f'}(g^{-1}a^{-1}\gamma\ol{ag}) 
	\ol{(R_\mu(ag)\phi')^\dagger(x, y)},
\end{align*}
where $T_{[\gamma, x, y]}$ denotes the stabilizer of
$[x, y]$ in $T_\gamma$. In the third line, we have used
Poisson summation.
% , and in the sum,
% $[\gamma, x, y]$ runs over equivalence classes of 
% elements in $X_\rss$.

Combining the above calculations, we have
\begin{align*}
	\sI(f', \phi') &= \int\limits_{\GL_n(E)\bs\GL_n(\AA_E)} \sum_{\gamma
	\in \GL_n(E)} \widetilde{f'}(g^{-1}\gamma \ol g)
	\ol{\Theta_{\mu}(g, \phi')}\, dg \\
	% &= \sum_{[\gamma, x, y] \in X} \int\limits_{\GL_n(E) \bs \GL_n(\AA_E)}
	% \sum_{a \in T_{[\gamma, x, y]}\bs \GL_n(E)} \widetilde{f'}(g^{-1}a^{-1}
	% \gamma \ol{ag}) \ol{(R_\mu(ag)\phi')^\dagger(x, y)}\, dg \\
	&= \sum_{[\gamma, x, y] \in [X_\rss]} \int_{\GL_n(\AA_E)}
	\widetilde{f'}(g^{-1}\gamma \ol g)\ol{(R_\mu(g)\phi')^\dagger
	(x, y)}\, dg.
\end{align*}
% where the sum runs over equivalence classes of
% elements in $X_\rss$.
This finishes the proof of the factorization of 
$\sI(f', \phi')$ into products of local orbital integrals. 

The convergence of the global orbital integral is
proved in the same way as Lemma \ref{lem:glorbintconverge}.
To prove the convergence of the outer sum, note that
since $\widetilde{f'}$ is compactly supported
on $\GL_n(\AA_E)$, there are only finitely many
$\gamma \in S$ such that $\widetilde{f'}(g^{-1}\gamma g) \neq 0$ for some $g \in \GL_n(\AA_E)$.
Thus it suffices to show that for fixed $\gamma$, the
sum over the regular semisimple orbits of
$F_n \times F^n$ under the $T_\gamma$-action
converges absolutely.

As in the proof of \cite[Proposition~3.2.3]{MR3228451},
let $K_\gamma$ be a subset of $\GL_n(\AA_E)$
such that $\GL_n(\AA_E) = T_\gamma(\AA_F) \cdot
K_\gamma$. Then
\[
	\int_{\GL_n(\AA_E)} \widetilde{f'}(g^{-1}\gamma \ol g)\ol{(R_\mu(g)\phi')^\dagger(x, y)}\,dg
	= \int_{T_\gamma(\AA_F)}\ol{\eta(\det t)}\phi_1'(xt, t^{-1}y)\,dt
\]
where
\[
	\phi_1'(x, y) = \int_{K_\gamma} 
	\widetilde{f'}(k^{-1}\gamma \ol k)\ol{(R_\mu(k)\phi')^\dagger(x, y)}\, dk.
\]
Since $\widetilde{f'}$ is compactly supported and
$\gamma$ is regular semisimple, the integrand above vanishes
outside some compact subset of $K_\gamma$,
so $\phi_1'$ is a Schwartz function on
$\AA_{F,n} \times \AA_F^{-,n}$. It is shown in the
proof of \cite[Proposition~3.2.3]{MR3228451}
that the sum
\[
	\sum_{[x, y]}
	\int_{T_\gamma(\AA_F)} \phi_1'(xt, t^{-1}y)\,dt
\]
over regular semisimple orbits of $F_n \times F^{-,n}$
under the action of $T_\gamma(F)$ converges
absolutely, which implies that the absolute 
convergence of the outer sum.
\end{proof}

\section{Relative trace formula on the unitary group}
\subsection{Orbital integrals on the unitary group}
Let $G = \Res_{E/F}\GL_n$, and let $H = \U(V)$ be
its subgroup. 
Let $\LL + \LL^\vee$ be a polarization of the symplectic space
$(\Res_{E/F} V)^\vee$. 
Recall that we have fixed an isomorphism $\LL \simeq F_n$.

\begin{definition}
\label{def:unormal}
For $\beta \in \SHerm_n^\times(F)$, 
we say that an element $\zeta \in \GL_n(E)$ is \emph{normal with respect to} $\beta$ if $\zeta$ commutes with
$\beta^{-1}\zeta^\ast\beta\zeta$.
\end{definition}

\begin{lemma}
\label{lem:unormal}
Suppose $\zeta \in \GL_n(E)$ is normal with respect to $\beta$
and $\beta^{-1}\zeta^\ast\beta\zeta$ is regular
semisimple. Then 
\[
	T_\zeta \coloneqq \{(g, h) \in \U_n^\beta(F)
	\times \U_n^\beta(F) \colon g^{-1}\zeta h = \zeta \}
\]
is contained in the diagonal $\Delta \U_n^\beta(F)$,
so we can identify $T_\zeta$ with the centralizer
of $\zeta$ in $\U_n^\beta(F)$.
\end{lemma}
\begin{proof}
Note that if $g^{-1}\zeta h = \zeta$ for $g, h \in \U_n^\beta(F)$,
then $h \in C_{\GL_n(E)}(\beta^{-1}\zeta^\ast\beta\zeta)$.
Since $\zeta$ is normal with respect to $\beta$, $\zeta$ is also in this centralizer.
Since $\beta^{-1}\zeta^\ast\beta\zeta$ is regular semisimple, its centralizer in $\GL_n(E)$ is a torus. 
Thus $h$ commutes with $\zeta$, so
$g^{-1}h\zeta = g^{-1}\zeta h = \zeta$, so $g = h$.
\end{proof}

% \begin{definition}
% \label{def:uorbits}
% For $\beta \in \SHerm_n^\times(F)$, let 
% \[
% 	Y^\beta = \{[\zeta, z] \in \GL_n(E) \times E_n : \beta^{-1}\zeta^\ast\beta\zeta \in \Nm(\GL_n(E)_{\rss}) \}.
% \]
% The group $\U_n^\beta(F) \times \U_n^\beta(F)$ acts on
% $Y^\beta$ from the right by
% \[
% 	[\zeta, z].(g, h) = [g^{-1}\zeta h, zh].
% \]
% \end{definition}

\begin{definition}
\label{def:uorbits}
% Let $R^\beta$ be a set of normal representatives 
% for the regular semisimple orbits of
% \[
% 	\{\zeta \in \GL_n(E) : \beta^{-1}\zeta^\ast\beta\zeta
% 	\in \Nm(\GL_n(E)) \}
% \]
Consider the right action of 
$\U_n^\beta(F) \times \U_n^\beta(F)$ on $\GL_n(E)$ 
given by 
\[ \zeta.(g, h) = g^{-1}\zeta h. \]
Let $R^\beta$ be a set of normal representatives of the
regular semisimple orbits of this action whose
elements $\zeta$ satisfy $\beta^{-1}\zeta^\ast\beta\zeta \in \Nm(\GL_n(E))$.
Let
\[
	Y^\beta = R^\beta \times E_n.
\]
Say an element $[\zeta, z] \in Y^\beta$ is 
\emph{regular semisimple} if $z \in E_n$ is
regular semisimple under the right multiplication
action of $T_\zeta$ on $E_n$.

Two elements $[\zeta, z]$ and $[\zeta', z']$ of
$Y^\beta$ are \emph{equivalent} if $\zeta = \zeta'$
and there exists $t \in T_\zeta$ such that
$zt = z'$.
Let $Y_\rss^\beta$ denote the regular semisimple elements
of $Y^\beta$, and let $[Y_\rss^\beta]$ denote the
set of regular semisimple equivalence classes.
\end{definition}

Let $v$ be a place of $F$. Again, the above definitions can
be made locally. Let $\omega_{\psi_v', \mu_v}$
be the Weil representation of $\U(V)(F_v)$, which is
realized on $\S(\LL(F_v))$.

\begin{definition}
\label{def:pftu}
Let $^\ddagger$ be the partial Fourier transform
\[
	^\ddagger\colon \S(\LL(F_v))^{\otimes 2}
	\to \S(V^\vee(E_v))
\]
given by
\[
	(\phi_{1,v} \otimes \phi_{2, v})^\ddagger(z)
	= \int_{\LL(F_v)}
	\phi_{1,v}(x + w)\phi_{2,v}(x - w)\psi_v'(\langle w, y\rangle)\, dw,
\]
where $\langle - , - \rangle$ is the symplectic
pairing on $(\Res_{E/F} V)^\vee$, and $z = (x, y)$ with $x \in \LL(F_v)$, $y \in \LL^\vee(F_v)$.
The measure is the pullback of the measure
on $F_{v,n}$ which is self-dual with respect to $\psi_v'$ under the isomorphism
$\LL(F_v) \simeq F_{v,n}$.
\end{definition}

Note that for $h \in \U(V)(F_v)$, we have
\begin{equation}
\label{eq:pftDiag}
	(\omega_{\psi_v', \mu_v}(h)\phi_{1,v}
	\otimes \ol{\omega_{\psi_v', \mu_v}(h)\phi_{2,v}})^\ddagger(z)
	= (\phi_{1,v}\otimes \ol{\phi_{2,v}})^\ddagger(zh).
\end{equation}
Recall that we fixed an isomorphism $V \cong E^n$.
Then the skew-Hermitian form on $V$ determines an element
$\beta \in \SHerm_n^\times(F)$, and the partial Fourier
transform is a map to $\S(E_{v,n})$.

\begin{definition}
\label{def:orbintU}
For $f_v \in \C_c^\infty(G(F_v))$,
$\phi_{1,v} \otimes \phi_{2,v} \in \S(\LL(F_v))^{\otimes 2}$,
and $[\zeta, z] \in Y^\beta$, define
\begin{align*}
	&\Orb_v^{\psi, \mu}([\zeta, z], f_v, \phi_{1, v} \otimes
	\phi_{2, v})\\
	&= \int_{\U(V)(F_v)}\int_{\U(V)(F_v)}
	f_v(g^{-1}\zeta h) (\phi_{2,v} \otimes \ol{\omega_{\psi_v',\mu_v}(h^{-1}g)
	\phi_{1,v}})^\ddagger(z h)\, dg\, dh.
\end{align*}
This depends only on the equivalence class of 
$[\zeta, z]$ in $Y^\beta$.
\end{definition}

\begin{lemma}
For any test function $(f_v, \phi_{1,v} \otimes \phi_{2,v})$
and any regular semisimple equivalence class 
$[\zeta, z]$, the orbital integral 
$\Orb_v^{\psi, \mu}([\zeta, z], f_v, \phi_{1,v}
\otimes \phi_{2,v})$ is absolutely convergent. 
\end{lemma}
\begin{proof}
The orbital integral can be written
\[
	\iint\limits_{T_\zeta \bs (\U(V)(F_v))^2}
	f_v(g_0^{-1}\zeta h_0) \int_{T_\zeta} 
	(\phi_{2,v} \otimes \ol{\omega_{\psi_v', \mu_v}(h_0^{-1}g_0)\phi_{1,v}})^\ddagger(zth_0)\,dt\,dg_0\,dh_0.
\]
As in the proof of Lemma \ref{lem:glorbintconverge},
since the orbit of $z$ under the $T_\zeta$-action
and the orbit of $\zeta$ under the
$\U_n^\beta(F_v) \times \U_n^\beta(F_v)$-action are
closed, and since the function $(\phi_{2,v} \otimes \ol{\omega_{\psi_v', \mu_v}(h_0^{-1}g_0)\phi_{1,v}})^\ddagger$
is a Schwartz function, and $f_v$ is compactly supported
and smooth, the integral is absolutely convergent.
\end{proof}

\subsection{Global distribution on the unitary group}
Let $\omega_{\psi', \mu}$ be the Weil representation
of $\U(V)(\AA_F)$ on $\S(\LL(\AA_F))$. We can also
define a global partial Fourier transform
$^\ddagger \colon \S(\LL(\AA_F))^{\otimes 2}
\to \S(V^\vee(\AA_E))$ using the same formula as
in Definition \ref{def:pftu}.

\begin{definition}
For $f \in \C_c^\infty(G(\AA_F))$, define the kernel
function
\[
	K_f(g, h) = \sum_{\zeta \in G(F)} f(g^{-1}\zeta h).
\]
It is a function on $(G(F)\bs G(\AA_F))^2$.
\end{definition}

\begin{definition}
Let $f \in \C_c^\infty(G(\AA_F))$, $\phi_1 \otimes \phi_2 \in
\S(\LL(\AA_F))^{\otimes 2}$. Let
\[
	\sJ(f, \phi_1 \otimes \phi_2) = \iint\limits_{(H(F)\bs H(\AA_F))^2}
	K_f(g, h) \ol{\Theta_{\psi', \mu}(g, \phi_1)}
	\Theta_{\psi', \mu}(h,
	\phi_2) \, dg\, dh.
\]
\end{definition}

\begin{lemma}[{\cite[Lemma~4.2.1]{MR3228451}}]
\label{lem:productformula}
% For $h \in \U(V)(\AA_F)$, we have
% \begin{equation}
% \label{eq:pftDiag}
% 	(\omega_{\psi', \mu}(h)\phi_{1}
% 	\otimes \ol{\omega_{\psi', \mu}(h)\phi_{2}})^\ddagger(z)
% 	= (\phi_{1}\otimes \ol{\phi_{2}})^\ddagger(zh).
% \end{equation}
For any $a \in \U(V)(F)$, we have
\[
	\ol{\Theta_{\psi', \mu}(g, \phi_1)}
	\Theta_{\psi', \mu}(h, \phi_2) =
	\sum_{z \in V^\vee(E)} (\phi_2 \otimes
	\ol{\omega_{\psi', \mu}(h^{-1}ag)\phi_1})^\ddagger(zh).
\]	
\end{lemma}

\begin{definition}
\label{def:ugood}
Say that a test function $(f, \phi_1 \otimes \phi_2)$
is \emph{good} if it satisfies the following 
conditions.
\begin{enumerate}[(1)]
\item 
\label{item:umatrixcoeff}
There exists a nonarchimedean split place $v_1$
of $F$ such that $f_{v_1}$ is a truncated
matrix coefficient of a supercuspidal representation.
This means
\[
	f_{v_1}(g) = \widehat f_{v_1}(g) \cdot \mathbbm{1}_{G^\ast(F_{v_1})}(g),
\]
% Here since $v_1$ splits, $G(F_{v_1}) \simeq
% \GL_n(F_{v_1}) \times \GL_n(F_{v_1})$. 
where $\widehat f_{v_1}$ is a matrix coefficient of a 
supercuspidal representation of $G(F_{v_1})$, and
\[
	G^\ast(F_{v_1}) = \{g \in
	G(F_{v_1}) : \det g \in \O_{E_{v_1}}^\times \}.
\]
%In particular, $f_{v_1}$ is compactly supported.

\item
\label{item:ursssupp} 
There exists a nonarchimedean split place
$v_2 \neq v_1$ such that $f_{v_2}$ is supported on the
regular semisimple locus of $\GL_n(E_{v_2})$ under the 
$\U(V)(F_{v_2}) \times \U(V)(F_{v_2})$-action, and 
if $[\zeta, z] \in Y^\beta$ and
\[
	f_{v_2}(g^{-1}\zeta h)(\phi_{2, v_2} \otimes
	\ol{\omega_{\psi_{v_2}', \mu_{v_2}}(h^{-1}g)\phi_{1,v_2}})^\ddagger(zh) \neq 0
\]
for some $g, h \in \U(V)(F_{v_2})$, then $[\zeta, z]
\in Y^\beta_\rss$.

\item For all places $v$ of $F$, the function $f_v$
is supported on the subset
\[
	\{ 
	\zeta \in \GL_n(E_v) : 
	\beta^{-1}\zeta^\ast\beta\zeta \in \Nm(\GL_n(E_v)) 
	\}.
\]

% \item For all archimedean places $v$,
% the function $f_v$ is $K_v$-finite where $K_v$
% is the maximal compact subgroup of $G(F_v)$.
\end{enumerate}
%\todo{Where is (4) used?}
\end{definition}

\begin{proposition}
\label{prop:GeoU}
For good test functions $(f, \phi_1 \otimes \phi_2)$,
the integral $\sJ(f, \phi_1 \otimes \phi_2)$
is absolutely convergent. 
Moreover, if $f = \prod_v f_v$, $\phi_1
= \prod_v \phi_{1,v}$, $\phi_2 = \prod_v \phi_{2,v}$
are all factorizable, then
\[
	\sJ(f, \phi_1 \otimes \phi_2)
	= \sum_{[\zeta, z] \in [Y_\rss^\beta]} \prod_v\Orb_v^{\psi, \mu}([\zeta, z], f_v, \phi_{1,v} \otimes \phi_{2,v}).
\]
% where the sum runs over all regular semisimple 
% equivalence classes of elements in
% $Y^\beta$.
\end{proposition}
\begin{proof}
By the assumptions on 
$(f_{v_2}, \phi_{1,v_2} \otimes \phi_{2, v_2})$ and
Lemma \ref{lem:productformula}, we have
\begin{align*}
&\sum_{\zeta \in \GL_n(E)} 
f(g^{-1}\zeta h)
\ol{\Theta_{\psi',\mu}(g, \phi_1)}\Theta_{\psi',\mu}(h, \phi_2)\\
% &= \sum_{\zeta \in R} \sum_{(a, b) \in T_\zeta \bs H(F) \times H(F)}
% f(g^{-1}a^{-1}\zeta bh) 
% \sum_{z \in E_n}
% (\phi_2 \otimes \ol{\omega_{\psi', \mu}(h^{-1}b^{-1}ag)\phi_1})^\ddagger(zbh) \\
% &= \sum_{\zeta \in R^\beta} \sum_{(a, b) \in T_\zeta \bs (H(F))^2}
% f(g^{-1}a^{-1}\zeta bh) \sum_{z \in E_n/T_\zeta}
% \sum_{t \in T_{[\zeta, z]} \bs T_\zeta}
% (\phi_2 \otimes \ol{\omega_{\psi', \mu}(h^{-1}b^{-1}ag)\phi_1})^\ddagger(ztbh) \\
&= \sum_{[\zeta, z] \in [Y_\rss^\beta]} 
\sum_{(a, b) \in H(F) \times H(F)}
f(g^{-1}a^{-1}\zeta bh) (\phi_2 \otimes \ol{\omega_{\psi', \mu}(h^{-1}b^{-1}ag)\phi_1})^\ddagger(zbh).
\end{align*}
% where $[\zeta, z]$ runs over equivalence classes of
% elements in $Y^\beta_\rss$.
Thus,
\begin{align*}
	\sJ(&f, \phi_1 \otimes \phi_2)\\
	&= \sum_{[\zeta, z] \in [Y_\rss^\beta]}
	\int_{\U(V)(\AA_F)}\int_{\U(V)(\AA_F)}
	f(g^{-1}\zeta h) (\phi_2 \otimes \ol{\omega_{\psi', \mu}(h^{-1}g)\phi_1})^\ddagger(zh) \,dg\,dh.
\end{align*}
% where the sum runs over equivalence classes of elements in 
% $Y^\beta_\rss$.
The absolute convergence of the global integral and the
outer sum is shown in the same way
as in the proof of Proposition \ref{prop:GeoGL}.
\end{proof}

\section{Fundamental lemma and smooth transfer}
Let $v$ be a place of $F$.
Throughout this section, all objects are over $F_v$,
and we suppress $v$ from the notation.

\subsection{Matching of orbits}
In this subsection, we prove the matching between the
orbit spaces in the two relative trace formulas.

\begin{definition}
Let 
\[ M_n = \Mat_n \times \Mat_{1, n} \times \Mat_{n, 1}, \]
with right action of $\GL_n$ given by 
\[
  [\xi, x, y].h = [h^{-1}\xi h, xh, h^{-1}y].
\]
An element of $M_n(E)$ is \emph{regular semisimple} if
\begin{enumerate}
\item $\xi$ is regular semisimple as an element of $\Mat_n(E)$;
\item the vectors $\{x, x\xi, \dots, x\xi^{n-1}\}$ span the $E$-vector
space $E_n$;
\item the vectors $\{y, \xi y, \dots, \xi^{n-1}y\}$ span the $E$-vector
space $E^n$.
\end{enumerate}
Let $M_n(E)_{\rss}$ denote the set of regular 
semisimple elements of 
$M_n(E)$. For $[\xi, x, y] \in M_n(E)$, let
\[
  a_i([\xi, x, y]) = \Tr \wedge^i \xi, \quad
  b_i([\xi, x, y]) = x\xi^{i - 1} y
\]
for $1 \le i \le n$. These are invariant under
the $\GL_n(E)$-action on $M_n(E)$.
\end{definition}

% Let $X_\rss$ and $Y_\rss^\beta$ denote the regular 
% semisimple elements of $X$ and $Y^\beta$, respectively,
% and let $[X_\rss]$ and $[Y_\rss^\beta]$ denote the
% sets of equivalence classes.
% These spaces have the following properties.
% First, 
Note that
\[
  X_\rss = \{[\gamma, x, y] \in S
  \times F_n \times F^{-,n} : [\ol\gamma \gamma, x, y]
  \in M_n(E)_\rss \},
\]
and
\[
  Y_\rss^\beta = \{[\zeta, z] \in R^\beta \times E_n : 
  [\beta^{-1}\zeta^\ast\beta\zeta, z, \beta^{-1}z^\ast]\in M_n(E)_\rss\}.
\]

\begin{lemma}[{\cite[Lemma~5.6]{MR3244725}, \cite[Proposition~6.2~and~Theorem~6.1]{rallis2007multiplicity}}]
Two regular semisimple elements 
$[\xi, x, y]$ and $[\xi', x', y']$ of $M$ are
in the same $\GL_n(E)$-orbit if and only if
they have the same invariants.
The $\GL_n(E)$-stabilizer of a regular semisimple
element of $M_n(E)$ is trivial.
% For $[\xi, x, y] \in M_\rss$, the elements $[\xi, x, y]$ and 
% $[\transp{\xi}, \transp{y}, \transp{x}]$ have the same invariants.
\end{lemma}

\begin{lemma}
\label{lem:embedM}
Two elements $[\gamma, x, y]$ and $[\gamma', x', y']$
of $X_\rss$ are equivalent if and only if
$[\ol\gamma\gamma, x, y]$ and $[\ol\gamma'\gamma', x', y']$
are in the same $\GL_n(E)$-orbit of $M_n(E)$.
Two elements $[\zeta, z]$ and $[\zeta', z']$
of $Y_\rss^\beta$ are equivalent if and only
if $[\beta^{-1}\zeta^\ast\beta\zeta, z, \beta^{-1}z^\ast]$ and $[\beta^{-1}\zeta'^\ast\beta\zeta', z',
\beta^{-1}z'^\ast]$ are in the same $\GL_n(E)$-orbit 
of $M_n(E)$.
\end{lemma}
\begin{proof}
The only if direction for both statements is
straightforward.

Suppose that $[\ol\gamma\gamma, x, y].h
= [\ol\gamma'\gamma', x', y']$. Then $\ol\gamma\gamma$
and $\ol\gamma'\gamma'$ are elements of $\GL_n(F)$ which are
$\GL_n(E)$-conjugate, so they are also $\GL_n(F)$-conjugate.
By the injectivity of the norm map, $\gamma$
and $\gamma'$ are $\sigma$-conjugate, so
$\gamma = \gamma'$.
Then it is easy to check that
$[\ol\gamma\gamma, x, y].\ol h = [\ol \gamma\gamma, x',
y']$, so by the triviality of stabilizers of regular
semisimple elements of $M_n(E)$, we have $h \in \GL_n(F)$.
This implies $h \in T_\gamma$, so $[\gamma, x, y]$
and $[\gamma', x', y']$ are equivalent in $X$
since $[x, y].h = [x', y']$.

For the second statement, if
$[\beta^{-1}\zeta^\ast\beta\zeta, z, \beta^{-1}z^\ast].h
  = [\beta^{-1}\zeta'^\ast\beta\zeta', z', \beta^{-1}z'^\ast]$,
then it is easy to check that
$
  [\beta^{-1}\zeta^\ast\beta\zeta, z, \beta^{-1}z^\ast].\beta^{-1}h^{\ast, -1}\beta
  = [\beta^{-1}\zeta'^\ast\beta\zeta', z', \beta^{-1}z'^\ast]
$, 
so by the triviality of stabilizers of regular semisimple elements of $M_n(E)$,
we have $h \in \U_n^\beta(F)$.
Letting $g = \zeta h \zeta'^{-1}$, it is easy to check
that $g \in \U_n^\beta(F)$, so $\zeta = \zeta'$.
Furthermore, since $\zeta$ is normal, we have
$h \in T_\zeta$, so $[\zeta, z]$ and
$[\zeta', z']$ are equivalent since $z' = zh$.
\end{proof}

% \begin{remark}
% \label{rmk:beta}
% If $\beta$ and $\beta'$ are similar, then there
% is a canonical bijection between equivalence classes
% in $Y_\rss^\beta$ and $Y_\rss^{\beta'}$.
% By an argument similar to the above,
% $[\beta^{-1}\zeta^\ast\beta\zeta, z, \beta^{-1}z^\ast]$
% and $[\beta'^{-1}\zeta'\beta'\zeta', z', \beta'^{-1}z'^\ast]$ are in the same $\GL_n(E)$-orbit of $M_n(E)$
% if and only if $\beta$ and $\beta'$ are similar, and $[\zeta, z]
% \in Y_\rss^\beta$ and $[\zeta', z'] \in Y_\rss^{\beta'}$
% correspond under this bijection.
%\end{remark}

We also consider the following orbit spaces for
Lie algebras.

\begin{definition}
\label{def:lieorbits}
Let $\fx = \gl_n \times F_n \times F^{-,n}$, and let
$\fy^\beta = \fh_n^\beta \times E_n$, where
$\fh_n^\beta$ is the subvariety of matrices $A' \in \Res_{E/F}\gl_{n}$ satisfying $A'^\ast = \beta A'\beta^{-1}$.
The groups $\GL_n$ and $\U_n^\beta$ act on
$\fx$ and $\fy^\beta$, respectively, via
\[
  [A, x, y].h = [h^{-1}A h, xh, h^{-1} y],
  \quad [A', z].h = [h^{-1}A' h, zh].
\]
Let $\fx_\rss(F) / \GL_n(F)$ and $\fy_\rss^\beta(F)/ \U_n^\beta(F)$ denote the regular semisimple orbits of 
$\fx(F)$ and $\fy^{\beta}(F)$ under these actions.
% The space $\fx$ embeds into $M_n$ in the obvious way,
% and the space $\fy^\beta$ embeds into $M_n$ via
% $[A', z] \mapsto [A', z, \beta^{-1}z^\ast]$.
\end{definition}

The following lemma follows from \cite[Lemma~3.1]{MR3245011}
after multiplying $A$, $A'$, $y$, and $\beta^{-1}$ 
by an element of $F^-$.

\begin{lemma}[{\cite[Lemma~3.1]{MR3245011}}]
\label{lem:liematch}
There is a bijection
\[
  \fx_\rss(F) / \GL_n(F) \simeq \coprod_{\beta \in [\SHerm_n^\times(F)]} \fy_\rss^\beta(F) / \U_n^\beta(F)
\]
such that $[A, x, y]$ and $[A', z]$ match if and only
if $[A, x, y]$ and $[A', z, \beta^{-1}z^\ast]$ are in the
same $\GL_n(E)$-orbit of $M_n(E)$.
\end{lemma}

\begin{lemma}
\label{lem:hermitiannorms}
Suppose $A' \in \fh_n^\beta(F)$ is regular semisimple and
the norm of some element of $\GL_n(E)$. Then 
$A' = \beta^{-1}\zeta^\ast\beta\zeta$ for some 
$\zeta \in \GL_n(E)$.
\end{lemma}
\begin{proof}
Since $A'$ is the norm of an element of $\GL_n(E)$, there
exists $h \in \GL_n(E)$ and normal $\gamma \in \GL_n(E)$
such that
\[
  h^{-1} A' h = \ol\gamma\gamma \in \GL_n(F).
\]
Since $\ol\gamma\gamma$ is regular semisimple, we have
\[
  \gamma = c_{n-1}(\ol\gamma\gamma)^{n-1} + \cdots + c_0
\]
for some $c_0, \dots, c_{n-1} \in E$.
Then it is easy to check that 
$(h\gamma h^{-1})^\ast = \beta (h\ol\gamma h^{-1}) \beta^{-1}$,
and we can take $\zeta = h\gamma h^{-1}$.
\end{proof}

\begin{lemma}
\label{lem:matching}
There is a bijection 
\[
  [X_\rss] \simeq \coprod_{\beta \in [\SHerm_n^\times(F)]} [Y_\rss^\beta]
\]
such that $[\gamma, x, y]$ and $[\zeta, z]$ match 
if and only if $[\ol\gamma \gamma, x, y]$ and 
$[\beta^{-1}\zeta^\ast \beta \zeta, z, \beta^{-1} z^\ast]$ 
are in the same $\GL_n(E)$-orbit of $M_n(E)$.
% The union is over all isomorphism classes of
% $n$-dimensional Hermitian spaces $V$ over $E$.
\end{lemma}
\begin{proof}
This follows from Lemmas \ref{lem:embedM}, \ref{lem:liematch},
and \ref{lem:hermitiannorms}.
\end{proof}

We write
\[
  [\gamma, x, y] \leftrightarrow [\zeta, z]^\beta
\]
to indicate matching equivalence classes.
Note that the definitions and results above on the matching
of orbits also apply in the global situation.

\subsection{Smooth transfer}
\label{sec:transfer}
For $[\xi, x, y] \in M_n(E)_\rss$, let
\[
  \mathbf T_{[\xi, x, y]} = \det \begin{pmatrix}
  x \\ x\xi \\ \vdots \\ x\xi^{n-1}
  \end{pmatrix},
\]
and define the transfer factor
\[
  \tf([\xi, x, y]) = \mu\left(\mathbf T_{[\xi, x, y]}\right).
  %\cdot (\det \xi)^{-[\frac n2]}\right).
\]

% Also define the spaces
% \[S_n(F) = \{g \in \GL_n(E) : \ol g g = 1\},\]
% \[ \fs_n(F) = \{A \in \Mat_n(E) : \ol A = - A\}, \]
% \[\fu_n^\beta(F) = \{A \in \Mat_n(E) : A^\ast = -\beta A \beta^{-1}\},\]
% for $\beta \in \SHerm_n^\times(F)$. 
% Let $S_n(\O_F)$, $\fs_n(\O_F)$, and $\fu_n^\beta(\O_F)$ denote the
% intersections of $S_n(F)$, $\fs_n(F)$, and $\fu_n^\beta(F)$ with
% $\Mat_n(\O_E)$, respectively. 

We say that test functions $(f', \phi')$ and
$\{(f^\beta, \phi_1^\beta \otimes \phi_2^\beta)\}_{\beta \in [\SHerm_n^\times(F)]}$ are $(S, R^\beta)$-smooth transfer of
each other if for all matching orbits 
\[
  [\gamma, x, y] \leftrightarrow [\zeta, z]^\beta,
\] 
we have
\[
  \Orb^{\psi, \mu}([\gamma, x, y], f', \phi')
  = \tf([\ol\gamma\gamma, x, y]) \Orb^{\psi, \mu}([\zeta, z]^\beta, f^\beta, \phi_1^\beta \otimes \phi_2^\beta).
\]

In general, this notion depends on the choice of representatives. 
To deal with this, we impose the following conditions on $S$ and $R^\beta$.
We will usually omit the $(S, R^\beta)$- prefix when talking
about smooth transfer.

\begin{definition}
\label{def:compatible}
We say that sets of representatives $S$ and
$R^\beta$ are \emph{compatible} if for all $\gamma \in S$,
$\zeta \in R^\beta$, and $h \in \GL_n(E)$ such that
\[
  \beta^{-1}\zeta^\ast\beta\zeta = h^{-1}\ol\gamma\gamma h,
\]
we also have
\[
  \zeta = h^{-1}\gamma h.
\]
\end{definition}

Note that the notion of compatibility also applies in
the global situation.

\begin{definition}
\label{def:intcomb}
Suppose $E/F$ is an unramified quadratic extension of
nonarchimedean local fields.
\begin{enumerate}[(1)]
\item We say that an element $\gamma \in \GL_n(E)$ is
\emph{Kottwitz} if either
\[
  \gamma = c_{n-1}(\ol\gamma\gamma)^{n-1} + \cdots + c_0
\]
for some $c_0, \dots, c_{n-1} \in \O_{E}$, or
there does not exist $g \in \GL_n(E)$ such
that $g^{-1}\ol\gamma\gamma g \in \GL_n(\O_{E})$.

\item Let $\beta \in \SHerm_n^\times(F)$.
We say that an element $\zeta \in \GL_n(E)$ is \emph{Kottwitz
with respect to $\beta$} if either
\[
  \zeta = c_{n-1}(\beta^{-1}\zeta^\ast\beta\zeta)^{n-1} + \cdots + c_0
\]
for some $c_0, \dots, c_{n-1} \in \O_{E}$, or
there does not exist any $h \in \U_n^\beta(F)$ such
that $h^{-1}(\beta^{-1}\zeta^\ast\beta\zeta)h \in
\GL_n(\O_{E})$.
\end{enumerate}
\end{definition}

The following definition will not be used in this
section, but we will need it later for choosing good test
functions.

\begin{definition}
Suppose $E = F \times F$, where $F$ is a nonarchimedean
local field.
Let $\varpi_{F}$ be a uniformizer of $F$.
\begin{enumerate}[(1)]
\item
We say that an element 
$\gamma = (\gamma_1, \gamma_2) \in \GL_n(F) \times \GL_n(F)$
is $k$-\emph{Kottwitz} if 
\[
  \gamma_1 = c_{n-1}(\gamma_2\gamma_1)^{n-1} + \cdots + c_0
\]
for some $c_0 \in 1 + \varpi_{F}^k\O_{F}$ and
$c_1, \dots, c_{n-1} \in \varpi_{F}^k\O_{F}$.

\item Let $\beta \in \SHerm_n^\times(F)$. We say that
an element $\zeta = (\zeta_1, \zeta_2) \in \GL_n(F) \times
\GL_n(F)$ is $k$-\emph{Kottwitz with respect to} $\beta$
if 
\[
  \zeta_1 = c_{n-1}(\beta_1^{-1}\transp\zeta_2\beta_1\zeta_1)^{n-1} + \cdots + c_0
\]
for some $c_0 \in 1 + \varpi_{F}^k\O_{F}$ and
$c_1, \dots, c_{n-1} \in \varpi_{F}^k\O_{F}$.
\end{enumerate}
\end{definition}

For the remainder of this section, we assume that $S$
and $R^\beta$ are compatible, and if $E/F$ is an unramified
field extension, then all elements $\gamma \in S$
are Kottwitz.
In fact, if $\beta^{-1}\zeta^\ast\beta\zeta = h^{-1}\ol\gamma\gamma h$ for some $h \in \GL_n(E)$, and
\[
  \gamma = c_{n-1}(\ol\gamma\gamma)^{n-1} +
  \cdots + c_0
\]
for $c_0, \dots, c_{n-1} \in E$, then
\begin{equation}
\label{eq:zeta}
  \zeta = c_{n-1}(\beta^{-1}\zeta^\ast\beta\zeta)^{n-1}
  + \cdots + c_0.
\end{equation}
Then it is easy to see that if $E/F$ is an unramified
field extension, then every element $\zeta \in R^\beta$
is Kottwitz (with respect to $\beta$) as well.

\begin{lemma}
\label{lem:kottwitz}
Suppose $E/F$ is an unramified field extension.
For $S$ and $R^\beta$ satisfying the above conditions, 
all elements $\gamma \in S$ and $\zeta \in R^\beta$
have the following properties.
\begin{enumerate}[\normalfont(1)]
\item \label{item:glkott}
If $g^{-1}\ol\gamma\gamma g \in \GL_n(\O_{E})$
for $g \in \GL_n(E)$, then $g^{-1}\gamma g
\in \GL_n(\O_{E})$.

\item \label{item:GLnFGLnO}
If $g^{-1}\gamma \ol g \in \GL_n(\O_{E})$ for
$g \in \GL_n(E)$, then
$g \in \GL_n(F)\GL_n(\O_{E})$.

\item \label{item:ukott}
If $h^{-1}(\beta^{-1}\zeta^\ast\beta\zeta)h
\in \GL_n(\O_{E})$ for $h \in \U_n^\beta(F)$, then $h^{-1}\zeta h
\in \GL_n(\O_{E})$.
% \item \label{item:u}
% If $g^{-1} \zeta h \in \GL_n(\O_E)$ for $g, h \in
% \U_n^\beta(F)$, then $g^{-1}h \in \U_n^\beta(\O_F)$.
\end{enumerate}
\end{lemma}
\begin{proof}
For the first property, by definition of Kottwitz, 
if $g^{-1}\ol\gamma\gamma g \in \GL_n(\O_{E})$, then
$\gamma = c_{n-1}(\ol\gamma\gamma)^{n-1} + \cdots + c_0$
for some $c_0, \dots, c_{n-1} \in \O_{E}$.
Then clearly
\[
  g^{-1}\gamma g = c_{n-1}(g^{-1}\ol\gamma\gamma g)^{n-1}
  + \cdots + c_0 \in \Mat_n(\O_{E}).
\]
Furthermore, we have $\Nm(\det(g^{-1}\gamma g)) = \det(g^{-1}\ol\gamma\gamma g) \in \O_{E}^\times$, so in fact $g^{-1}\gamma g
\in \GL_n(\O_{E})$.

The second property follows from 
the first property and the last statement of \cite[Lemma~8.8]{MR564478}.

The third property is true for the same reason as the first,
since by \eqref{eq:zeta}, we have
$\zeta = c_{n-1}(\beta^{-1}\zeta^\ast\beta z)^{n-1}
+ \cdots + c_0$ for the same $c_0, \dots, c_{n-1}$.
\end{proof}

\subsection{Smooth transfer at split places}
In this subsection, we assume that $E = F \times F$.
In this case, all Hermitian
spaces are split, and if $\beta = (-\alpha, \transp{\alpha})$, then
\[
  \U_n^\beta(F) = \{(h, \alpha^{-1}\transp{h^{-1}}\alpha) \in \GL_n(E) = \GL_n(F) \times \GL_n(F)\}.
\]

First we note that the orbital integrals on the 
unitary side do not depend on the choice of Lagrangian.

\begin{lemma}
\label{lem:pftlagrangian}
If $V^\vee = \LL' + \LL'^\vee$
is another polarization of $V^\vee$, the Weil representation
$\omega_{\psi', \mu} \otimes \ol{\omega_{\psi', \mu}}$ is also realized on
$\S(\LL')^{\otimes 2}$. There is an isomorphism of 
representations $\S(\LL)^{\otimes 2} \simeq \S(\LL')^{\otimes 2}$ such that
\[
\begin{tikzcd}
\S(\LL)^{\otimes 2} \arrow{r}{^\ddagger} \arrow[swap]{d}{\simeq} & \S(V^\vee)\\
\S(\LL')^{\otimes 2} \arrow[swap]{ru}{^\ddagger}
\end{tikzcd}
\]
commutes.
\end{lemma}
\begin{proof}
Let $(-V)$ denote the Hermitian space $(V, -\langle -, - \rangle_V)$.
Then $\Res_{E/F}(V + (-V))^\vee$ has two polarizations,
$(\LL + \LL) + (\LL^\vee + \LL^\vee)$ and
$V^\Delta + V^\nabla$, where $V^\Delta$ is the diagonal
embedding of $V^\vee$ in $(V + (-V))^\vee$ and 
$V^\nabla = \{(v, -v) : v \in V\}$.

The Weil representation of the 
metaplectic group $\Mp(V + (-V))$ is realized on both
$\S(\LL)^{\otimes 2}$ and $\S(V^\vee)$.
Denote the action on $\S(\LL)^{\otimes 2}$ by
$\omega_{\psi',\LL}$ and the action on $\S(V^\vee)$
by $\omega_{\psi', V}$. By \cite[Section~I.7]{MR1041060},
the partial Fourier transform is equivariant for
these actions.

There is a natural embedding of $\U(V) \times
\U(V)$ into $\U(V + (-V))$. Let 
\[  i_\mu \colon \U(V + (-V)) \to \Mp(V + (-V)) \]
be the splitting determined by $\mu$.
Then the partial Fourier transform is an isomorphism
of the representations $\omega_{\psi', V} \circ i_\mu$
and $\omega_{\psi', \LL} \circ i_\mu$ of
$\U(V) \times \U(V)$.

By \cite[Proposition~2.2]{MR1327161}, 
$\omega_{\psi', \LL} \circ i_\mu \simeq 
\omega_{\psi', \mu} \otimes \mu \cdot
\ol{\omega_{\psi', \mu}}$ as representations
of $\U(V) \times \U(V)$.
Thus, the composition of the partial Fourier transform
$\S(\LL)^\otimes \to \S(V^\vee)$ with the inverse
partial Fourier transform
$\S(V^\vee) \to \S(\LL')^{\otimes 2}$ will be an isomorphism
of representations $\omega_{\psi', \mu}
\otimes \ol{\omega_{\psi', \mu}}$ of $\U(V) \times \U(V)$.
\end{proof}

With a specific choice of Lagrangian, the Weil 
representation of the unitary group is given by a simple 
formula.

\begin{lemma}
\label{lem:weilrepgln}
% \todo{Check formula. Need to look into how $\mu$
% determines the splitting?}
Let $\beta = (1, -1)$.
Consider the polarization of
$V^\vee \cong F_n \times F_n$ with 
$\LL = F_n \times \{0\}$ and $\LL^\vee =
  \{0\} \times F_n$.
The Weil representation of
$\U(V) = \{ (g, \transp{g^{-1}}) \mid g \in \GL_n(F) \}$
realized on $\S(\LL)$ is given by
\[
  \omega_{\psi', \mu}((g, \transp{g^{-1}}))\phi((x,0))
  = \mu(\det g, 1)|\det g|^{\frac12}\phi((xg, 0)).
\]
\end{lemma}

With $\beta = (1, -1)$ and the polarization as in Lemma \ref{lem:weilrepgln}, we identify both $\LL$ and $\LL^\vee$
with $F_n$ via $(x, 0) \mapsto x$ and $(0, \transp y)
\mapsto \transp y$.
Then we identify $\S(F_n)^{\otimes 2}$ with 
$\S(E_n)$ and a subspace of $\S(\LL)^{\otimes 2}$ 
in the obvious way.

\begin{proposition}
\label{prop:splitmatch}
Suppose $f = f_1 \otimes f_2$, and let
$\transp{f_2}(g) \coloneqq f_2(\transp g)$. Suppose
that
\[
  \widetilde{f'} = \zeta_{E}(1)^{-1}
  (f_1 \otimes \transp{f_2}), \quad \phi' = \phi_1 \otimes \ol{\phi_2}.
\]
Then the test functions $(f', \phi')$ and
$(f, \phi_1 \otimes \phi_2)$ match.
% Then for all matching orbits $[\gamma, x, y] \leftrightarrow
% [\zeta, z]$, we have
% \[
%   \Orb^{\psi,\mu}([\gamma, x, y], f', \phi')
%   = \Orb^{\psi, \mu}([\zeta, z], f, \phi_1
%   \otimes \phi_2).
% \]
\end{proposition}
\begin{proof}
Suppose $[\gamma, x, y] \in X_{\rss}$ and
$[\zeta, z] \in Y_\rss^1$ match. We want to show that
\[
  \Orb^{\psi,\mu}([\gamma, x, y], f', \phi')
  = \Orb^{\psi, \mu}([\zeta, z], f, \phi_1
  \otimes \phi_2).
\]

Suppose $\gamma = (\gamma_1, \gamma_2)$
and $\zeta = (\zeta_1, \zeta_2)$ as elements
of $\GL_n(F) \times \GL_n(F)$.
Then we have
\begin{align*}
  &\Orb^{\psi, \mu}([\gamma, x, y], f', \phi')\\
  % &= \int\limits_{\GL_n(E)} \widetilde{f'}(g^{-1}\gamma \ol g)\ol{(R_\mu(g)\phi')^\dagger(x, y)}\, dg \\
  % &= \iint\limits_{\GL_n(F) \times \GL_n(F)} \widetilde{f'}((g_1^{-1},
  % g_2^{-1})(\gamma_1, \gamma_2)(g_2, g_1))
  % \ol{(R_\mu(g_1, g_2)\phi')^\dagger(x, y)}\, dg_1\, dg_2\\
  &=\zeta_E(1)^{-1}\iint\limits_{\GL_n(F) \times \GL_n(F)}
  f((g_1^{-1}, \transp{g_1})(\gamma_1, \transp{\gamma_2})
  (g_2, \transp{g_2}^{-1})) \times \\
  &\hspace{21em}\ol{(R_\mu(g_1, g_2)\phi')^\dagger(x, y)}\, dg_1\, dg_2.
\end{align*}
On the other hand, under the identification $\U_n^\beta(F) \simeq
\GL_n(F)$, the Haar measure we have fixed on $\U_n^\beta(F)$ is 
$\zeta_F(1)^{-1}$ times the Haar measure on $\GL_n(F)$. Thus,
\begin{align*}
&\Orb^{\psi, \mu}([\zeta, z], f, \phi_1 \otimes \phi_2)\\
% &= \iint\limits_{\U_n^\beta(F) \times \U_n^\beta(F)}
% f(g^{-1}\zeta h)(\phi_2 \otimes \ol{\omega_{\psi, \mu}(h^{-1}g)\phi_1})^\ddagger(zh) \,dg \,dh\\
&= \zeta_E(1)^{-1}\iint\limits_{\GL_n(F) \times \GL_n(F)}
f((g_1^{-1}, \transp{g_1})(\zeta_1, \zeta_2)
(g_2, \transp{g_2}^{-1})) \times \\
&\hspace{9em}(\phi_2 \otimes \ol{\omega_{\psi',\mu}(g_2^{-1}g_1, \transp{(g_1^{-1}g_2)})\phi_1})^\ddagger(z(g_2, \transp{g_2^{-1}}))\,dg_1\,dg_2.
\end{align*}
Suppose $j = (j_0, -j_0)$ for $j_0 \in F$, and
$y = (y_0, -y_0)$ for $y_0 \in F^n$.
Note that
\begin{align*}
&\ol{(R_\mu(g_1, g_2)\phi')^\dagger(x, y)} \\
&= |j_0|
\int_{F_n} \ol{(R_\mu(g_1, g_2)\phi')((x, x) + (j_0w, -j_0w))}
\psi(-j_0wy_0)\, dw \\
% &= \int_{F_n} \mu(\det(g, h))|\det(gh)|^{\frac12}\phi'((x+j_0w)g, (x-j_0w)h)\psi(j_0^2wy)\, dw \\
&= \int_{F_n} \ol{\mu(\det(g_1, g_2))}|\det(g_1g_2)|^{\frac12}\ol{\phi_1((x-w)g_1)}\phi_2((x+w)g_2)\psi(wy_0)\,dw \\
&=\int_{\LL} 
(\omega_{\psi',\mu}(g_2, \transp{g_2^{-1}})\phi_2)(x + w)
\ol{(\omega_{\psi',\mu}(g_1, \transp{g_1^{-1}})\phi_1)(x - w)}
\psi(\langle w, \transp y_0)\rangle)\, dw \\
% &= (\omega_{\psi',\mu}(g_2, \transp{g_2^{-1}})\phi_2 \otimes \ol{\omega_{\psi',\mu}(g_1, \transp{g_1^{-1}})\phi_1})^\ddagger((x, \transp y)) \\
&= (\phi_2 \otimes \ol{\omega_{\psi', \mu}(g_2^{-1}g_1, \transp{(g_1^{-1}g_2)})
\phi_1})^\ddagger((x, \transp y_0)(g_2, \transp{g_2^{-1}})),
\end{align*}
where in the last step we have used \eqref{eq:pftDiag}.
Thus, we see that
\[
  \Orb^{\psi, \mu}([(\gamma_1, \gamma_2), x, y], f', \phi')
  = \Orb^{\psi, \mu}([(\gamma_1, \transp\gamma_2),
  (x, \transp y_0)], f, \phi_1 \otimes \phi_2).
\]

To prove the proposition, it suffices
to show that there exists $h \in \U_n^\beta(F)$ such that
\begin{equation}
\label{eqn:h}
  [h^{-1}(\gamma_1, \transp\gamma_2)h, (x, \transp y_0)h]
  = [(\zeta_1, \zeta_2), (z_1, z_2)].
\end{equation}
Since $[\gamma, x, y]$ and $[\zeta, z]$ match,
there exists $(h_1, h_2) \in
\GL_n(F) \times \GL_n(F)$ such that 
\[
  [(\gamma_1\gamma_2, \gamma_1\gamma_2), (x, x), (y_0, -y_0)].(h_1, h_2)
  = [(\zeta_1\transp{\zeta_2}, \zeta_2\transp{\zeta_1}),
  (z_1, z_2), (\transp{z_2}, \transp{(-z_1)})].
\]
Since $S$ and $R^\beta$ are compatible, we also have
\[
  (h_1^{-1}, h_2^{-1})(\gamma_1, \gamma_2)(h_1, h_2)
  = (\zeta_1, \zeta_2).
\]
% This implies
% \[
%   h_1^{-1}\gamma_1\gamma_2h_1 = \zeta_1\transp{\zeta_2},\quad
%   xh_1 = z_1,\quad h_1^{-1}y = \transp{z_2},\quad
%   h_1^{-1}\gamma_1h_1 = \zeta_1.
% \]
% This also implies
% \[
%   h_1^{-1}\gamma_2h_1 = \transp\zeta_2,\quad
%   \transp{y}\transp{h_1^{-1}} = z_2,
% \]
% so 
By a straightforward calculation, we can check that
$h \coloneqq (h_1, \transp h_1^{-1})$ satisfies \eqref{eqn:h}.
\end{proof}

From now on, when way say that test functions 
$(f', \phi')$ and $(f, \phi_1 \otimes \phi_2)$ match
at a split place $v$, we mean that they satisfy
the conditions of Proposition \ref{prop:splitmatch}.
From the calculation in the proof of Proposition
\ref{prop:splitmatch}, it can be seen
that if $(f', \phi')$ and $(f, \phi_1 \otimes \phi_2)$
match in this sense, then $(f', \phi')$ satisfies
condition \eqref{item:glrsssupp} of Definition \ref{def:glgood} at the place $v$, if
and only if $(f, \phi_1 \otimes \phi_2)$ satisfies
condition \eqref{item:ursssupp} of Definition \ref{def:ugood}
at the place $v$.
Furthermore, if $f'$ is a truncated matrix coefficient
of a supercuspidal representation, then so is $f$, and
if $f$ is a truncated matrix coefficient of a supercuspidal
representation, then $f'$ can be chosen to be one also.

\subsection{Fundamental lemma}
\label{sec:FL}
In this subsection,
we assume that $E/F$ is an unramified field extension.
In this case, there are two isomorphism classes of
$n$-dimensional Hermitian spaces over $E$.
Let $V^+$ denote the split Hermitian space, and let
$V^-$ denote the other one. Let $\beta^+$
and $\beta^-$ be Hermitian matrices corresponding to
them, and let $\U_n^\pm(F) = \U_n^{\beta^\pm}(F)$.
% $\fu_n^\pm(F) = \fu_n^{\beta^\pm}(F)$,
% and $\fu_n^\pm(\O_F) = \fu_n^{\beta^\pm}(\O_F)$.

Assume that $\O_{E,n}$ is a self-dual lattice in
$V^+$.
Also suppose $\LL^+ + \LL^{+,\vee}$ is a polarization of 
$(\Res_{E/F} V^+)^\vee$ such that
\[ 
  (\mathbbm{1}_{\LL^+(\O_F)} \otimes
  \mathbbm{1}_{\LL^+(\O_F)})^\ddagger =
  \mathbbm{1}_{\O_{E,n}},
\]
where $\LL^+(\O_F)$ and the partial Fourier transform
are defined via an isomorphism $\LL^+ \simeq F_n$.
%$= \LL^+ \cap \O_{E,n}$ and
%$\LL^{+,\vee}(\O_F) = \LL^{+,\vee} \cap \O_{E,n}$.

In what follows, we take $(f', \phi')$ and $(f, \phi_1 \otimes
\phi_2)$ to be the \emph{unramified test functions}
\[
  (f', \phi') = \left(
	\frac{1}{\Vol(\GL_n(\O_{E}))^2}{\mathbbm{1}_{\GL_n(\O_{E}) \times \GL_n(\O_E)}} 
	, 
	\mathbbm{1}_{\O_{E,n}}\right)
\]
and
\[
	(f, \phi_{1} \otimes \phi_{2}) = \left(
	\frac{1}{\Vol(\U_n^+(\O_{F}))^2}{\mathbbm{1}_{\GL_n(\O_{E})}}, 
	\mathbbm{1}_{\LL^+(\O_{F})} \otimes \mathbbm{1}_{\LL^+(\O_{F})}\right).
\]
% Here, $\mathbbm{1}_{\LL^+(\O_F)} \otimes
% \mathbbm{1}_{\LL^+(\O_F)}$ is an element of
% $\S(\LL^+)^{\otimes 2}$ such that
% \[ (\mathbbm{1}_{\LL^+(\O_F)} \otimes
% \mathbbm{1}_{\LL^+(\O_F)})^\ddagger =
% \mathbbm{1}_{\O_{E,n}}.\]
For the Hermitian space $V^-$, we call
\[
  (f^{\beta^-}, \phi_1^{\beta^-} \otimes \phi_2^{\beta^-})
  = (0, 0)
\]
the unramified test function.
% Note that all elements $\gamma \in S$ and $\zeta
% \in R^\beta$
% satisfy the following properties, since they were
% chosen to be Kottwitz. If $g^{-1}\ol\gamma\gamma g \in \GL_n(\O_{E})$
% for $g \in \GL_n(E)$, then $g^{-1}\gamma g
% \in \GL_n(\O_{E})$.
% If $h^{-1}(\beta^{-1}\zeta^\ast\beta\zeta)h
% \in \GL_n(\O_{E})$, then $h^{-1}\zeta h
% \in \GL_n(\O_{E})$.

We recall the following version of the Jacquet--Rallis
fundamental lemma for Lie algebras. 
It is shown separately in \cite{MR2769216} and \cite{MR4250392}
for large residue characteristic, and in general in
\cite{MR4305382} and \cite{zhang2021maximal}.
% in \cite{MR2769216}, \cite{MR4305382}, \cite{MR4250392}, and \cite{zhang2021maximal}.
% It is shown in
% \cite{MR2769216} when the residue characteristic of
% $F$ is large, and in full in \cite{MR4305382} by a purely local
% proof. A global proof is given in \cite{MR4250392} for
% residue characteristic greater than $n$, and
% in general in \cite{zhang2021maximal}.

Let $\fh_{n+1}^+$ denote the subvariety of matrices $A' \in \Res_{E/F}\gl_{n+1}$ satisfying $A'^\ast = \beta^+ A'\beta^{+,-1}$.
Note that $\GL_n$ (resp. $\U_n^+$)
embeds into $\GL_{n+1}$ (resp. $\U_{n+1}^+$) via
\[
  g \mapsto \begin{bmatrix}
  g & \\ & 1
  \end{bmatrix},
\] 
and thus acts by conjugation via this embedding on $\gl_{n+1}$
(resp. $\fh_{n+1}^+$).
An element $A \in \gl_{n+1}(F)$ (resp. $A' \in \fh_{n+1}^+(F)$)
is called \emph{relatively regular semisimple} if 
its stabilizer in $\GL_n$ (resp. $\U_n^+$) is trivial
and its orbit is closed.
Two relatively regular semisimple elements 
$A \in \gl_{n+1}(F)$ and $A' \in \fh_{n+1}^+(F)$ 
\emph{match} if they are
conjugate by an element of $\GL_n(E)$.
For relatively regular semisimple elements $A \in \gl_{n+1}(F)$,
the transfer factor is defined by
\[
  \tf(A) \coloneqq \mu\left(\det \begin{pmatrix}
  e_{n+1}^\ast \\ e_{n+1}^\ast A \\ \vdots \\ e_{n+1}^\ast A^n
  \end{pmatrix} \right),
\]
where $e_{n+1}^\ast = (0, \dots, 0, 1) \in F_{n+1}$.

\begin{theorem}[{\cite[Theorem~1]{MR4305382}, \cite[Theorem~4.1]{zhang2021maximal}}]
\label{thm:J-R}
For matching relatively regular semisimple elements 
$A \in \gl_{n+1}(F)$ and $A' \in \fh_{n+1}^+(F)$,
% are strongly regular semisimple
% in the sense of \cite[Definition~2.2.1]{MR2769216},
% and match in the sense of 
% \cite[Definition~2.5.1]{MR2769216}, 
we have
\begin{align*}
  \frac{1}{\Vol(\GL_n(\O_F))}&\int_{\GL_n(F)} \mathbbm{1}_{\gl_{n+1}(\O_F)}(g^{-1}Ag)\eta(\det g)\, dg \\
  &= \tf(A)\frac{1}{\Vol(\U_n^+(\O_F))}
  \int_{\U_n^+(F)}\mathbbm{1}_{\fh_{n+1}^+(\O_F)}(h^{-1}A' h)\, dh.
\end{align*}
Furthermore, if $A$ does not match any $A'$, then
the left hand side of the above equation is zero.
\end{theorem}

With our assumptions on the orbit representatives, our
orbital integrals in fact reduce to the orbital integrals
in the above theorem.

\begin{lemma}
\label{lem:gltoJR}
For $[\gamma, x, y] \in X_\rss$, we have
\begin{align*}
  &\Orb^{\psi, \mu}(s, [\gamma, x, y], f', \phi')\\
  &= \frac{\mathbbm{1}_{\O_F^\times}(\det \ol\gamma\gamma)}{\Vol(\GL_n(\O_F))} 
  \int_{\GL_n(F)} \mathbbm{1}_{\gl_{n+1}(\O_F)}(g^{-1}Ag)\eta(\det g)\ol{|\det g|^{s - \frac12}}\,dg,
\end{align*}
for any 
\[  
  A = \begin{bmatrix}
  \ol\gamma\gamma & jy \\ x & d
  \end{bmatrix} \in \gl_{n+1}(F)
\]
with $d \in \O_F$.
\end{lemma}
\begin{proof}
Since $\gamma$ is Kottwitz, by Lemma \ref{lem:kottwitz},
we have
\begin{align*}
  &\Orb^{\psi, \mu}(s, [\gamma, x, y], f', \phi')\\
  &= \int_{\GL_n(E)} \widetilde f'(g^{-1}\gamma \ol g)
  \ol{(R_\mu(g)\phi')^\dagger(x, y)|\det g|^{s - \frac12}}\, dg \\
  &=\frac{1}{\Vol(\GL_n(\O_E))}\int\limits_{\GL_n(F)\GL_n(\O_E)} \mathbbm{1}_{\GL_n(\O_E)}(g^{-1}\gamma \ol g)
  \ol{(R_\mu(g)\phi')^\dagger(x, y)|\det g|^{s - \frac12}} \, dg \\
  &=\frac{1}{\Vol(\GL_n(\O_F))}\int_{\GL_n(F)} \mathbbm{1}_{\GL_n(\O_E)}(g^{-1}\gamma g) \times \\
  &\hspace{15em}
  \mathbbm{1}_{\O_{F,n}}(xg)\mathbbm{1}_{\O_F^{-,n}}(g^{-1}y) 
  \eta(\det g) \ol{|\det g|^{s - \frac12}}\,dg\\
  &= \frac{1}{\Vol(\GL_n(\O_F))}\int_{\GL_n(F)} \mathbbm{1}_{\GL_n(\O_F)}(g^{-1} \ol\gamma \gamma g) \times \\
  &\hspace{15em}
  \mathbbm{1}_{\O_{F,n}}(xg)\mathbbm{1}_{\O_F^{-,n}}(g^{-1}y) 
  \eta(\det g)\ol{|\det g|^{s - \frac12}}\,dg \\
  % &= \mathbbm{1}_{\O_F^\times}(\det \ol\gamma\gamma)
  % \int_{\GL_n(F)} \mathbbm{1}_{\gl_n(\O_F)}(g^{-1}
  % \ol\gamma \gamma g) \mathbbm{1}_{\O_{F,n}}(xg)\mathbbm{1}_{\O_F^n}(g^{-1}jy) \eta(\det g)\,dg \\
  &= \frac{\mathbbm{1}_{\O_F^\times}(\det \ol\gamma\gamma)}{\Vol(\GL_n(\O_F))}
  \int_{\GL_n(F)} \mathbbm{1}_{\gl_{n+1}(\O_F)}(g^{-1}Ag)\eta(\det g)\ol{|\det g|^{s - \frac12}}\, dg. 
\end{align*}
In the above calculation, $\O_F^-$ denotes $\O_E \cap F^-$.
\end{proof}

\begin{lemma}
\label{lem:utoJR}
For $[\zeta, z]^{\beta^+} \in Y_\rss^{\beta^+}$, we have
\begin{align*}
  \Orb^{\psi, \mu}([\zeta, z]^{\beta^+}, f, \phi_1 \otimes
  \phi_2) = \frac{\mathbbm{1}_{\O_F^\times}(\det \zeta^\ast\zeta)}{\Vol(\U_n^+(\O_F))}
  \int_{\U_n^+(F)} \mathbbm{1}_{\fh_{n+1}^+(\O_F)}(h^{-1}A'h)\, dh,
\end{align*}
for any
\[
  A' = \begin{bmatrix}
  (\beta^+)^{-1}\zeta^\ast \beta^+ \zeta & z^\ast \\ z & d
  \end{bmatrix} \in \fh_{n+1}^+(F)
\]
with $d \in \O_F$.
\end{lemma}
\begin{proof}
Note that we can take $\beta^+ = j$.
Since $\zeta$ is Kottwitz with respect to $\beta^+$, 
by Lemma \ref{lem:kottwitz}, we have
\begin{align*}
&\Orb^{\psi, \mu}([\zeta, z]^{\beta^+}, f, \phi_1 \otimes
  \phi_2) \\
  &=\iint\limits_{\U_n^+(F)\times \U_n^+(F)} f(g^{-1}\zeta h)
  (\phi_2 \otimes \ol{\omega_{\psi',\mu}(h^{-1}g)\phi_1})^{\ddagger}(zh)\,
  dg\, dh \\
  &= \frac{1}{\Vol(\U_n^+(\O_F))^2}\iint\limits_{\U_n^+(F)\times \U_n^+(F)}\mathbbm{1}_{\GL_n(\O_E)}(g^{-1}h
  (h^{-1}\zeta h)) \times \\
  &\hspace{19em}(\phi_2 \otimes \ol{\omega_{\psi,\mu}(h^{-1}g)\phi_1})^\ddagger(zh)\,dg\,dh \\
  &= \frac{1}{\Vol(\U_n^+(\O_F))}\int_{\U_n^+(F)}\mathbbm{1}_{\GL_n(\O_E)}(h^{-1}\zeta^\ast\zeta h)
  (\phi_2 \otimes \phi_1)^\ddagger (zh) \,dh \\
  &= \frac{1}{\Vol(\U_n^+(\O_F))}\int_{\U_n^+(F)} \mathbbm{1}_{\GL_n(\O_E)}(h^{-1}\zeta^\ast
  \zeta h) \mathbbm{1}_{\O_{E,n}}(zh)\, dh \\
  &= \frac{\mathbbm{1}_{\O_F^\times}(\det\zeta^\ast\zeta)}{\Vol(\U_n^+(\O_F))}
  \int_{\U_n^+(F)} \mathbbm{1}_{\fh_{n+1}^+(\O_F)}(h^{-1}A'h)\, dh.
\end{align*}
Note that the fourth line follows from the Lemma \ref{lem:kottwitz} because if $h^{-1}\zeta^\ast\zeta h \in \GL_n(\O_E)$, then $h^{-1}\zeta h \in \GL_n(\O_E)$,
so $g^{-1}h(h^{-1}\zeta h) \in \GL_n(\O_E)$ is equivalent to
$g^{-1}h \in \U_n^+(\O_F)$. 
\end{proof}

With these lemmas, we deduce the following results about the
unramified test functions from the Jacquet--Rallis fundamental
lemma for Lie algebras.

\begin{proposition}
\label{prop:betaminusFL}
Suppose $[\gamma, x, y]$ is regular semisimple and matches with 
$[\zeta, z]^{\beta^-}$. Then
\[
  \Orb^{\psi, \mu}([\gamma, x, y], f', \phi') = 0.
\]
\end{proposition}
\begin{proof}
Choose $d \in \O_F$ so that 
\[  A = \begin{bmatrix}
  \ol\gamma\gamma & jy \\ x & d
  \end{bmatrix} \in \GL_{n+1}(F)
\]
is relatively regular semisimple.
Such a $d$ exists by \cite[Lemma~5.17]{MR3244725},
after multiplying by a purely imaginary element.
The vanishing of the orbital integral then follows from
the second statement of Theorem \ref{thm:J-R}.
\end{proof}

\begin{theorem}[Fundamental lemma]
\label{thm:FL}
Suppose $[\gamma, x, y]$ is regular semisimple and
matches with $[\zeta, z]^{\beta^+}$. 
Then
\[
  \Orb^{\psi, \mu}([\gamma, x, y], f', \phi') =
  \tf([\ol\gamma\gamma, x, y]) 
  \Orb^{\psi, \mu}([\zeta, z]^{\beta^+}, f, \phi_1 \otimes
  \phi_2).
\]
\end{theorem}
\begin{proof}
As before, let $\beta^+ = j$.
Let $d$ be as in the proof of Proposition \ref{prop:betaminusFL}, and let $A$ and $A'$ be as in Lemmas \ref{lem:gltoJR}
and \ref{lem:utoJR}, respectively.
Then $A$ is relatively regular semisimple and matches with
$A'$, so $A'$ is also relatively regular semisimple.
Furthermore, we have 
$\tf(A) = \tf([\ol\gamma\gamma, x, y])$.
Since $[\gamma, x, y]$ and $[\zeta, z]^{\beta^+}$ match,
we also have $\det(\ol\gamma\gamma) = 
\det(\zeta^\ast\zeta)$. Thus by Lemmas \ref{lem:gltoJR},
\ref{lem:utoJR}, and Theorem \ref{thm:J-R}, we have
\begin{align*}
  \Orb^{\psi, \mu}([\gamma, x, y], f', \phi')
  &= \frac{\mathbbm{1}_{\O_F^\times}(\det \ol\gamma\gamma)}{\Vol(\GL_n(\O_F))}
  \int_{\GL_n(F)} \mathbbm{1}_{\gl_{n+1}(\O_F)}(g^{-1}Ag)\eta(\det g)\,dg \\
  &= \tf(A)\frac{\mathbbm{1}_{\O_F^\times}(\det \zeta^\ast\zeta)}{\Vol(\U_n^+(\O_F))}
  \int_{\U_n^+(F)} \mathbbm{1}_{\fh_{n+1}^+(\O_F)}(h^{-1}A'h)\, dh \\
  &= \tf([\ol\gamma\gamma, x, y])\Orb^{\psi, \mu}([\zeta, z]^{\beta^+}, f, \phi_1 \otimes
  \phi_2),
\end{align*}
as desired.  
\end{proof}

Finally, we have the following proposition, which tells us
that the unramified test functions on unitary side satisfy
the third condition of Definition \ref{def:ugood}.

\begin{proposition}
\label{prop:unrgood}
The function $\mathbbm{1}_{\GL_n(\O_E)}$ is supported
on the set
\[
  \{\zeta \in \GL_n(E) : (\beta^+)^{-1}\zeta^\ast\beta^+\zeta \in \Nm(\GL_n(E))\}.
\]
\end{proposition}
\begin{proof}
Again take $\beta^+ = j$.
Suppose that $\zeta \in \GL_n(\O_E)$, so 
$\zeta^\ast\zeta \in \GL_n(\O_E)$.
Let $p_1(X) | p_2(X) | \cdots
| p_r(X)$ be the elementary divisors of $\zeta^\ast\zeta$. 
Since $\zeta^\ast\zeta$ is conjugate to $\transp(\zeta^\ast\zeta)
= \ol{\zeta^\ast\zeta}$, we have $p_i(X) \in F[X]$, so $\zeta^\ast\zeta$
is conjugate to some element $\delta \in \GL_n(F)$.
% Suppose $\zeta^\ast\zeta = h^{-1}\delta h$.
Then by Lemma \ref{lem:kottwitz8-8}, $\delta =
\ol\gamma\gamma$ for some $\gamma$, which implies that
$\zeta^\ast\zeta$ is of this form as well.
\end{proof}

\section{Comparison of relative trace formulas}
\subsection{Choice of representatives}
\label{sec:rep}
In this subsection, we show that it is possible to choose
compatible sets of \emph{global} representatives $S$ and $R^\beta$, so that every element of $S$ is Kottwitz at
every inert place $v$ of $F$. 

\begin{lemma}
\label{lem:compatible}
For every set of representatives $S$, there exists a
set of representatives $R^\beta$ which is compatible with $S$.
\end{lemma}
\begin{proof}
We can choose the set $R^\beta$ as follows.
For every $\zeta' \in \GL_n(E)$ such that
$\beta^{-1}\zeta'^\ast\beta\zeta'$ is a regular semisimple
element of $\Nm(\GL_n(E))$, there exists unique
$\gamma \in S$ such that 
$\beta^{-1}\zeta'^\ast\beta\zeta' = h^{-1}\ol\gamma\gamma h$ for some $h \in \GL_n(E)$. 
Let $\zeta \coloneqq h^{-1}\gamma h$.
Then if \[ \gamma = c_{n-1}(\ol\gamma\gamma)^{n-1} + \cdots + c_0,\]
we have
\begin{align*}
    \beta^{-1}\zeta^\ast\beta  = \ol{c_{n-1}}(h^{-1}\ol\gamma\gamma h)^{n-1} + \cdots + \ol{c_0} = h^{-1}\ol\gamma h,
\end{align*}
so $\beta^{-1}\zeta^\ast\beta \zeta =
\beta^{-1}\zeta'^\ast\beta\zeta'$.
Taking $\zeta$ to 
be the normal representative of the orbit
$\U_n^\beta(F)\zeta'\U_n^\beta(F)$, we obtain a set
of representatives $R^\beta$ compatible with $S$.
\end{proof}

Thus the claim follows from the following proposition, which
is the main result of this subsection.

\begin{proposition}
\label{prop:globalkottwitz}
Let $k \ge 0$ be an integer, and let $v_2$ be
a nonarchimedean split place of $F$.
Suppose $\delta \in \GL_n(F)$ is a regular semisimple
element which is the norm of some element of $\GL_n(E)$. Then there exists
$\gamma \in \GL_n(E)$ with $\ol\gamma\gamma = \delta$,
such that $\gamma$ is Kottwitz at all inert places $v$
of $F$, and $k$-Kottwitz at $v_2$. %not lying over $2$.
\end{proposition}

We first recall some results from \cite{MR564478}.

\begin{lemma}[{\cite[Lemma~8.6]{MR564478}}]
\label{lem:kotcoho}
Let $v$ be an inert place of $F$.
Let $C$ be a commutative 
$\O_{F_v}$-algebra which is free of finite rank as an $\O_{F_v}$-module.
Let $C_{E_v} = C \otimes_{\O_{F_v}} \O_{E_v}$. 
Then \[ \hat H^0(\Gal(E_v/F_v), C_{E_v}^\times) = 0. \] 
\end{lemma}

\begin{lemma}[{\cite[Lemma~8.8]{MR564478}}]
\label{lem:kottwitz8-8}
Let $v$ be an inert place of $F$.
Let $\delta \in \GL_n(F_v)$ and suppose there
exists $g \in \GL_n(E_v)$ such that $g^{-1}\delta g
\in \GL_n(\O_{E_v})$. 
Then there exists a Kottwitz element
$\gamma \in \GL_n(E_v)$ such that $\ol\gamma\gamma = \delta$.
\end{lemma}
\begin{proof}
Let $\Mat_n(E_v)$ be the $E_v$-algebra of $n \times n$ matrices.
Let 
\[
  C_{E_v} = \{ c_{n-1}\delta^{n-1} + \cdots + c_0 \mid
  c_0, \dots, c_{n-1} \in \O_{E_v} \}.
\]
Since $g^{-1}\delta g \in \GL_n(\O_{E_v})$ for some
$g \in \GL_n(E_v)$, the characteristic polynomial of
$\delta$ has coefficients in $\O_{F_v}$ and
constant term in $\O_{F_v}^\times$. 
Thus, $C_{E_v}$ is a commutative $\O_{E_v}$-subalgebra of $\Mat_n(E_v)$, which is free of finite rank,
and $\delta \in C_{E_v}^\times$.

Since $\delta \in \GL_n(F)$, we see that $C_{E_v}$
is stable under the action of $\Gal(E_v/F_v)$ on $\Mat_n(E_v)$.
Let $C = C_{E_v}^{\Gal(E_v/F_v)}$ be the $\O_{F_v}$-subalgebra
of $C_{E_v}$ of elements fixed by the Galois group. 
Since $C_{E_v}$ is a $\Gal(E_v/F_v)$-invariant lattice
in the vector subspace of $\Mat_n(E_v)$ which it generates,
$C \otimes_{\O_{F_v}}\O_{E_v} \simrightarrow
C_{E_v}$ is an isomorphism of $\O_{E_v}$-algebras compatible
with the $\Gal(E_v/F_v)$ action.

Thus by Lemma \ref{lem:kotcoho}, $\hat H^0(\Gal(E_v/F_v),
C_{E_v}^\times) = 0$, which means that every element
of $(C_{E_v}^\times)^{\Gal(E_v/F_v)} = C^\times$
is the norm of some element of $C_{E_v}^\times$.
We have $\delta \in C^\times$, so there exists $\gamma \in
C_{E_v}^\times$ such that $\ol\gamma\gamma = \delta$,
so $\gamma$ is Kottwitz.
\end{proof}

\begin{corollary}
\label{cor:kottwitzexist}
Let $v$ be a nonarchimedean place of $F$.
Let $\delta \in
\GL_n(F_v)$ and suppose that $\delta$ is the norm of
some element of $\GL_n(E_v)$. 
\begin{enumerate}[\normalfont(1)]
\item If $v$ is inert, there exists a Kottwitz element 
$\gamma \in \GL_n(E_v)$ such that $\ol\gamma\gamma = \delta$.
\item If $v$ is split and $k \ge 0$ is an integer,
there exists a $k$-Kottwitz element $\gamma \in \GL_n(E_v)$ 
such that $\ol\gamma\gamma = \delta$.
\end{enumerate}
\end{corollary}
\begin{proof}
The first statement follows from Lemma 
\ref{lem:kottwitz8-8}.
For the second statement, it is easy to see that 
$\gamma = (1, \delta)$ is $k$-Kottwitz for any $k \ge 0$.
\end{proof}

We will also use the following lemma in the
proof of Proposition \ref{prop:globalkottwitz}.

\begin{lemma}
\label{lem:replacesplit}
Let $L/E/F$ be extensions of number fields, with $E/F$
a quadratic extension. Let $\fm$ be an ideal of $\O_L$,
and let $\fA_L$ be an ideal of
$\O_L$ coprime to $\fm$. 
Let $\Sigma$ be a set of places of $F$ containing
only finitely many places which are split in $E$.
Then there exists a prime $\fP_L$ of $\O_L$, coprime to $\fm$, in the same
ray class modulo $\fm$ as $\fA_L$, such that $\fP_F = \fP_L \cap \O_F$ is not contained in $\Sigma$.
\end{lemma}

\begin{proof}
It suffices to note that there exist infinitely many
primes $\fP_L$ of $\O_L$ in the same ray class
modulo $\fm$ as $\fA_L$, which have degree $1$ over $F$
(meaning $|\O_L/\fP_L| = |\O_F/\fP_F|$).

This follows from the fact that the set of prime ideals
of $\O_L$ in the same ray class modulo $\fm$ as $\fA_L$
has positive Dirichlet density.
%(this is a fact that has no relation to $F$).
However, the set $T$ of primes $\fP_L$ of $\O_L$ with 
$|\O_L/\fP_L| \ge |\O_F/\fP_F|^2$ has 
Dirichlet density $0$, due to the fact that
\[
	\sum_{\fP_L \in T} \frac{1}{|\O_L/\fP_L|}
	\le \sum_{\fP_L \in T} \frac{1}{|\O_F/\fP_F|^2}
	\le n\cdot \sum_{\fP_F \subseteq \O_F} \frac{1}{|\O_F/\fP_F|^2}
	< \infty.
\]

This implies the lemma because then there exists 
such a degree $1$ prime $\fP_L$ which does not lie over 
one of the finitely many split places in $\Sigma$, 
and hence not over any prime in $\Sigma$ 
(as a degree $1$ prime would only lie over split places of $F$).
\end{proof}

\begin{proof}[Proof of Proposition \ref{prop:globalkottwitz}]
Let $p(X) \in F[X]$ be the characteristic polynomial
of $\delta$, which is also its minimal polynomial,
and suppose $p(X)$ factors as 
$p(X) = p_1(X)\cdots p_r(X)$
for monic irreducible polynomials $p_i(X) \in F[X]$.
%Let $n_i$ be the degree of $p_i(X)$.
For $1 \le i \le r$, let $K_i = F[X]/(p_i(X))$, and 
let $K = K_1 \times \cdots \times K_r$.
Let $L_i = E\otimes_F K_i$, and let
$L = L_1 \times \cdots \times L_r$.
%Let $\Gal(L_i/K_i) = \langle \sigma \rangle$.
%$\Gal(L/E) = \langle \tau \rangle$,
%\item Let $\alpha, \prescript{\tau}{}{\alpha}$ be the roots of $p(X)$ in $K$.
Let $\alpha_i$ be the image of $X$ in $K_i$, and
let $\alpha = (\alpha_1, \dots, \alpha_r) \in K$.

Let $\Sigma$ be the set of inert places $v$ of $F$
such that there exists $g \in \GL_n(E_v)$ such
that $g^{-1}\delta g \in \GL_n(\O_{E_v})$. 
Note that $\alpha \in \O_{K_v}^\times$ for all
$v \in \Sigma$.
For each inert place $v$ of $F$, say that an element $x_v \in L_v$ is \emph{Kottwitz} if $\ol{x_v}x_v = \alpha$ and either $v
\notin \Sigma$, or
\[ x_v = c_{n-1}\alpha^{n-1}
+ \cdots + c_0 \]
for some $c_0, \dots, c_{n-1} \in \O_{E_v}$.
Note that if $v \in \Sigma$ and $x_v$ is
Kottwitz, then $x_v \in \O_{L_v}^\times$.

In addition, for $v = v_2$, say that an element 
$x_{v} = (x_{v, 1}, x_{v, 2}) \in L_{v} 
\simeq K_{v} \times K_{v}$ is
$k$-\emph{Kottwitz} if $\ol{x_{v}}x_{v} = \alpha$, and 
\[
  x_{v, 1} = c_{n-1}\alpha^{n-1} + \cdots + c_0
\]
for some $c_0 \in 1 + \varpi_{F_v}^k\O_{F_v}$ and
$c_1, \dots, c_{n-1} \in \varpi_{F_v}^k\O_{F_v}$.

To prove the proposition, it suffices to show that there exists 
$x \in L$ which is Kottwitz at $v$ for all $v \in \Sigma$,
and $k$-Kottwitz at $v_2$.

\begin{claim}
\label{cl:compactsupp}
Let $v \in \Sigma \cup \{v_2\}$ be a place of $F$.
\begin{enumerate}[\normalfont(1)]
\item 
If $v \in \Sigma$, let $m_v$ be the smallest nonnegative integer satisfying
\[ 
  m_v \ge \tfrac12\ord_v(\disc p(X)). 
\]
%Note that $m_v = 0$ for almost all inert $v$.
If $x_v \in L_v$ is any Kottwitz
element, and $x_v' \in L_v$ satisfies 
$\ol{x_v'}x_v' = \alpha$ and
\[
  x_v' - x_v \in \varpi_{F_v}^{m_v}\O_{L_v},
\] 
then $x_v'$ is Kottwitz as well.

\item 
If $v = v_2$, let $m_v$ be the smallest nonnegative integer
satisfying 
\[ 
  m_v \ge \tfrac12\ord_v(\disc p(X)) + k. 
\]
If $x_v = (x_{v,1}, x_{v,2}) \in L_v$ is any $k$-Kottwitz element,
and $x_v' = (x_{v,1}', x_{v,2}') \in L_v$ satisfies
$\ol{x_v'}x_v' = \alpha$ and
\[
  x_{v,1}' - x_{v,1} \in \varpi_{F_v}^{m_v}\O_{K_v},
\]
then $x_v'$ is $k$-Kottwitz as well.
\end{enumerate}
\end{claim}

\begin{proof}
Suppose $p(X)$ factors as $p(X) = q_1(X) \cdots q_s(X)$
for monic irreducible polynomials $q_i(X) \in F_v[X]$.
Let $K_{i,v}' = F_v[X]/(q_i(X))$, so
$K_v = K_{1,v}' \times \cdots \times K_{s, v}'$,
and let $L_{i,v}' = E_v \otimes_{F_v} K_{i,v}'$, so
$L_v = L_{i,v}' \times \cdots \times L_{s, v}'$.
Let $\alpha_{i}'$ denote the image of $X$ in
$K_{i,v}'$. Then 
$\alpha = (\alpha_1', \dots, \alpha_s') \in K_v$.
Fix an algebraic closure $\ol{F_v}$ of $F_v$.
Let $\tau_{i,1}, \dots, \tau_{i, \ell_i}$ denote
the different embeddings $K_{i,v}' \hookrightarrow \ol{F_v}$.
Each $\tau_{ij}$ extends to an embedding
$L_{i,v}' \hookrightarrow E_v \otimes_{F_v} \ol{F_v}$.

Elements $x_v \in L_v$ are of the form $(x_{1,v}, \dots, x_{s,v})$
for $x_{i,v} \in L_{i,v}'$.
If $x_{v} = c_{n-1}\alpha^{n-1} + \cdots + c_0$
for $c_0, \dots, c_{n-1} \in E_v$, then
\[
  c_{n-1}\tau_{ij}(\alpha_i')^{n-1}
  + \cdots + c_0 = \tau_{ij}(x_{i,v})
\]
for $1 \le i \le s$, $1 \le j \le \ell_i$, where
both sides are viewed as elements of $E_v \otimes_{F_v}
\ol{F_v}$.
Then
\[
  \begin{pmatrix}
  c_{n-1} \\ \vdots \\ c_{n-\ell_1}\\ \vdots \\ c_0
  \end{pmatrix} = \begin{pmatrix}
  \tau_{1,1}(\alpha_1')^{n-1} & \cdots & 1 \\
  % \sigma_{1,2}(\alpha_1)^{n-1} & \cdots & 1 \\
  \vdots & \ddots & \vdots \\
  \tau_{1,\ell_1}(\alpha_1')^{n-1} & \cdots & 1 \\
  \vdots & \ddots & \vdots \\
  \tau_{s, \ell_s}(\alpha_s')^{n-1} & \cdots & 1
  \end{pmatrix}^{-1}
  \begin{pmatrix}
  \tau_{1,1}(x_{1,v}) \\
  \vdots \\ \tau_{1, \ell_1}(x_{1,v}) \\\vdots \\ \tau_{s, \ell_s}(x_{\ell_s, v})
  \end{pmatrix}.
\]
Now, the matrix has determinant equal to a square root of 
$\disc(p(X))$, which is an element of $F$. Thus the claim
follows from the above formula.
\end{proof}

\begin{claim}
\label{cl:choosex0}
There exists $x_0 \in L$ such that $\ol{x_0}x_0 = \alpha$, and $x_0 \in \O_{L_v}^\times$ for all $v \in \Sigma$,
and $x_{0,1} \in \O_{K_v}^\times$ for $v = v_2$.
% \[
%   \alpha \in \O_{K_v}^\times \implies x_0 \in \O_{L_v}^\times,
% \]
% for places $v$ of $F$ inert in $E$. 
\end{claim}
\begin{proof}
It suffices to show that for each $1 \le i \le r$,
there exists $x_{0,i} \in L_i$ such that
$\ol{x_{0,i}}x_{0,i} = \alpha_i$, and
$x_{0,i} \in \O_{L_{i,v}}^\times$ for all $v \in \Sigma$,
and $x_{0,i,1} \in \O_{K_i,v}^\times$ for $v = v_2$.
If $L_i \simeq K_i \times K_i$, we can just
take $x_{0,i} = (1, \alpha_i)$. Thus below we assume
that $L_i$ is a field.

Let $u_2$ and $\ol{u_2}$ denote the places of 
$E$ lying above $v_2$, with $u_2$ corresponding to the first
component of $E_{v_2}$
under the isomorphism $E_{v_2} \simeq F_{v_2} \times F_{v_2}$.

The fact that $\delta$ is a norm of an element of $\GL_n(E)$
implies that there exists
$x_{0,i}' \in L_i$ such that $\ol{x_{0,i}'}x_{0,i}' = \alpha_i$.
Let
\[
  \fA = \prod_{w \nmid \ol{u_2}} \mathfrak P_w^{\max\{\ord_{w}(x_{0,i}'), 0\}} 
  \prod_{w \mid \ol{u_2}} \mathfrak P_w^{\max\{-\ord_{\ol w}(x_{0,i}'), 0\}},
\]
where $w$ runs over places of $L_i$, and
$\fP_w$ denotes the prime ideal of $\O_{L_i}$ corresponding
to $w$.
By Lemma \ref{lem:replacesplit} with $\fm = 1$, there
exists a prime ideal $\fP$ of $\O_{L_i}$ which is in the same
ideal class as $\fA$, such that $\fp = \fP \cap \O_F$
is not in $\Sigma \cup \{v_2\}$.
Let $a \in L_i$ be such that $a\fP = \fA$, and
set 
\[ x_{0,i} = x_{0,i}'\cdot(\ol a/a).\]

We claim that $x_{0,i}$ has
the desired property. 
Let $v \in \Sigma \cup \{v_2\}$, 
and let $w$ be a place of $L_i$ lying above $v$, which
does not lie above $\ol{u_2}$. 
We want to show that $x_{0,i} \in \O_{L_{i,w}}^\times$.
% If $w = \ol w$, then $v \in \Sigma$, so $\alpha \in \O_{K_{i,v}}^\times$,
% and $\ol{x_{0,i}}x_{0,i} \in \O_{K_{i,v}}^\times$ implies
% $x_{0,i} \in \O_{L_{i,w}}^\times$.
% So suppose that $w \neq \ol w$, and
% let $\fp_v$, $\fP_{w}$, $\ol{\fP_{w}}$ be the corresponding
% prime ideals. 

Note that
\[
  \ord_{\ol w}(\fA) = \max\{-\ord_w(x_{0,i}'), 0\}.
\]
Indeed, this is true by construction for $v = v_2$, and
for $v \in \Sigma$, it follows from the fact that
$\ol{x_{0,i}'}x_{0,i}' \in \O_{K_{i,v}}^\times$, which implies
$\ord_{w}(x_{0,i}') + \ord_{\ol w}(x_{0,i}') = 0$.
% If $v \in \Sigma$, then $\ol{x_{0,i}'}x_{0,i}' \in \O_{K_{i,v}}^\times$, so
% \[ \ord_{w}(x_{0,i}') + \ord_{\ol w}(x_{0,i}') = 0, \]
% which implies $\ord_{\ol w}(\fA) = \max\{-\ord_w(x_{0,i}'), 0\}$.
% By construction this is also true if $v = v_2$.
Recall also that $a^{-1}\fA = \fP$, where
$\fP$ is a prime ideal of $\O_{L_i}$ which does not lie over
$v$, and thus is not equal to $\fP_{w}$ or
$\ol{\fP_{w}}$. Thus, $\ord_{w}(a) = \ord_{w}(\fA)$,
and $\ord_{w}(\ol a) = \ord_{\ol{w}}(a) = \ord_{\ol{w}}(\fA)$.
Thus, 
\begin{align*}
  \ord_{w}(x_{0,i})
  &= \ord_{w}(x_{0,i}') - \ord_{w}(\fA) + \ord_{\ol w}(\fA)
  \\
  &= \ord_{w}(x_{0,i}') - \max\{ \ord_{w}(x_{0,i}'), 0 \}
  + \max\{-\ord_{w}(x_{0,i}'), 0\} 
  %\\
  % &= \ord_{w}(x_{0,i}') - \max\{ \ord_{w}(x_{0,i}'), 0 \}
  % - \min\{\ord_{w}(x_{0,i}'), 0\} 
  \\&
  = 0.\qedhere
\end{align*}
\end{proof}

Let $x_0 \in L$ satisfy the conditions of Claim
\ref{cl:choosex0}.
For each $v \in \Sigma$, let $x_v$ be a 
Kottwitz element of $L_v$. This is possible by
Lemma \ref{cor:kottwitzexist}. 
Furthermore, let $x_{v_2} = (1, \alpha)$, 
which is a $k$-Kottwitz element of $L_{v_2}$.
For each $v \in \Sigma \cup \{v_2\}$, let 
$a_v \in \O_{L_v}^\times$ be such that $x_v/x_0 = a_v/\ol{a_v}$. This is possible
because $L_v/K_v$ is an unramified extension,
and $x_v/x_0 \in \O_{L_v}^\times$ for all 
$v \in \Sigma \cup \{v_2\}$.

Now we show that we can choose a suitable $a \in L^\times$ so that $x = x_0\cdot(a/\ol a)$ is 
$k_v$-Kottwitz at all $v \in \Sigma'$.

\begin{claim}
\label{cl:choosea}
Let $m_v$ be as in Claim \ref{cl:compactsupp}.
There exists $a \in L^\times$ such that 
$a \in \O_{L_v}^\times$ and
$a - a_v \in \varpi_{F_v}^{m_v}\O_{L_v}$
for all $v \in \Sigma \cup \{v_2\}$.
\end{claim}
\begin{proof}
As $L$ is a product of fields containing $E$, and
we can choose each component of $a$ independently, 
for the proof of this claim only, 
we assume $L$ is a field.

Note that $m_v = 0$ for all but finitely many places
$v \in \Sigma \cup \{v_2\}$.
By weak approximation, there exists $a' \in L$ such that
$a' - a_v \in \varpi_{F_v}^{m_v}\O_{L_v}$
for all $v \in \Sigma \cup \{v_2\}$ such that $m_v > 0$.
Let 
\[
  \fm = \prod_{v \in \Sigma \cup \{v_2\}} \prod_{w\mid v}\fP_w^{\ord_{w}(\varpi_{F_v}^{m_v})}
\]
where $w$ runs over all primes of $L$ lying over $v$.
Since $a_v \in \O_{L_v}^\times$ for $v \in \Sigma \cup \{v_2\}$, we have that $(a')$ and $\fm$ are coprime.
By Lemma \ref{lem:replacesplit} with the modulus
$\fm$ and $\fA = (a')$, there exists a
prime ideal $\fQ \subseteq \O_L$ such that $\fQ$
and $(a')$ are equivalent in the ray class group modulo
$\mathfrak m$, and $\fq = \fQ \cap F$ is not
contained in $\Sigma \cup \{v_2\}$.
Let $b \in L_{\fm, 1}$ be such that
$(ba') = \fQ$, and let $a = ba'$. 

We claim that $a$ satisfies the properties required
in the claim. 
Indeed, for all $v \in \Sigma \cup \{v_2\}$, we have
$a \in \O_{L_v}^\times$, since $(a) = \fQ$ lies over a prime of $F$ which is not in $\Sigma \cup \{v_2\}$.
Furthermore, if $m_v = 0$, then $a, a_v \in \O_{L_v}^\times$
implies $a - a_v \in \O_{L_v}$, and
if $m_v > 0$, then
\[
  a - a_v = a'(b - 1) + (a' - a_v) \in \varpi_{F_v}^{m_v}\O_{L_v},
\]
since by construction $b - 1 \in \varpi_{F_v}^{m_v}\O_{L_v}$
and $a' - a_v \in \varpi_{F_v}^{m_v}\O_{L_v}$.
\end{proof}

Finally, with $a$ as above, let $x = x_0\cdot(a/\ol a)$.
We claim that $x$ is Kottwitz at all inert places of $F$
and $k$-Kottwitz at $v_2$. 
If $v$ is inert and not in $\Sigma$, this is automatic, 
so assume $v \in \Sigma \cup \{v_2\}$.
Recall that this implies $a_v \in \O_{L_v}^\times$.
Since $a, a_v \in \O_{L_v}^\times$
and $a - a_v \in \varpi_{F_v}^{m_v}\O_{L_v}$,
we have
$a/\ol a - a_v/\ol{a_v} \in \varpi_{F_v}^{m_v}\O_{L_v}$.
If $v \in \Sigma$, recall that $x_0 \in \O_{L_v}^\times$ as well, so
\[
  x - x_v 
  = x_0(a/\ol a - a_v/\ol{a_v})
  \in \varpi_{F_v}^{m_v}\O_{L_v}.
\]
If $v = v_2$, recall that $x_{0,1} \in \O_{K_v}^\times$,
so
\[
  x_1 - x_{v, 1} = x_{0,1}(a_1/a_2 - a_{v,1}/a_{v,2}) 
  \in \varpi_{F_v}^{m_v}\O_{K_v}.
\]
In either case, by Claim \ref{cl:compactsupp},
since $x_v$ is Kottwitz (resp. $k$-Kottwitz), 
$x$ is Kottwitz (resp. $k$-Kottwitz) at $v$ as well.
\end{proof}

From now on, we require that $S$ and $R^\beta$ satisfy
the conditions stated at the beginning of this subsection,
i.e., they are compatible, and every element of $S$ is
Kottwitz at \emph{every} inert place $v$ of $F$.
This is possible by Lemma \ref{lem:compatible}
and Proposition \ref{prop:globalkottwitz}. 

\subsection{Spectral side of relative trace formulas}
In this subsection, let $Z$ denote the center of $G$,
let $Z'$ denote the center of $G'$, and let $Z_2$
denote the center of $H_2$.
For any function $f$ on $G$ and any element $a \in Z$,
define $f^a(g) = f(ga)$, and similarly for 
functions on $G'$.

We say that global test functions $(f', \phi')$ and
$\{(f^\beta, \phi_1^\beta \otimes \phi_2^\beta)\}_{\beta \in [\SHerm_n^\times(F)]}$
are smooth transfer of each other if they factorizable
and are smooth
transfer of each other at each place $v$, and
at each place $v$, the test function $(f_v^\beta,
\phi_{1,v}^\beta \otimes \phi_{2,v}^\beta)$ depends 
only on the similarity class of $\beta$ in
$[\SHerm_n^\times(F_v)]$.

For such a pair of \emph{good} matching test functions, 
let $\Sigma$ be a finite set of nonarchimedean places of 
$F$ containing $\{v_1, v_2\}$, 
such that for all $v \notin \Sigma$,
either $v$ is split, or we are in the unramified
situation and the test functions are unramified.
Here, $v_1$ and $v_2$ are nonarchimedean split places of 
$F$ such that conditions \eqref{item:glmatrixcoeff} and \eqref{item:glrsssupp}
of Definition \ref{def:glgood} are satisfied by $(f', \phi')$.

Let $\iota \colon Z' \to Z$ be the map
$(a_1, a_2) \mapsto a_1\ol{a_2}^{-1}$.
% Note that $\iota$ induces an isomorphism
% \[
%   Z'(F)Z_2(\AA_F)\bs Z'(\AA_F) \simrightarrow
%   Z(F) \bs Z(\AA_F).
% \]
For each $v \in \Sigma$, let
$U_v$ be an open compact subgroup of $Z(F_v)$
such that $f_v^{a} = f_v$ for all $a$ such that $a \in U_v$,
and $f_v'^a = f_v'$ for all $\iota(a) \in U_v$.
Let
\[
  Z(\AA_F, \Sigma) = \{a \in Z(\AA_F) :
  a_v \in U_v \text{ for all } v \in \Sigma \}.
\]
Let $Z(F, \Sigma) = Z(\AA_F, \Sigma) \cap Z(F)$.
Let $Z'(\AA_F, \Sigma)$ (resp. $Z'(F, \Sigma)$) 
be the preimage of $Z(\AA_F, \Sigma)$ (resp.
$Z(F, \Sigma)$) under $\iota$.
Note that $\iota$ induces an isomorphism
\[
  Z'(F, \Sigma)Z_2(\AA_F)\bs Z'(\AA_F, \Sigma) \simrightarrow
  Z(F,\Sigma) \bs Z(\AA_F, \Sigma).
\]

\begin{lemma}
\label{lem:distrequal}
Suppose $(f', \phi')$ and $\{(f^\beta, \phi_1^\beta \otimes \phi_2^\beta)\}_{\beta \in [\SHerm_n^\times(F)]}$ are good matching
test functions.
Then for all $a \in Z'(\AA_F, \Sigma)$, we have
\[
  \sI(f'^a, \phi') = \sum_{\beta \in [\SHerm_n^\times(F)]}
  \sJ^\beta((f^\beta)^{\iota(a)}, \phi_1^\beta \otimes \phi_2^\beta).
\]
\end{lemma}
\begin{proof}
% Let $\Sigma$ be
% a finite set of places of $F$ containing
% $v_1$, $v_2$, such that for all $v \notin \Sigma$,
% either $v$ is split, or we are in the unramified
% situation and the test functions are unramified.
% For each $v \in \Sigma$, let
% $U_v$ be an open compact subgroup of $Z(F_v) \simeq E_v^\times$
% such that $f_v^{a} = f_v$ and
% $(\widetilde{f'})^a = \widetilde{f'}$ 
% for all $a \in U_v$.
% Let
% \[
%   \AA_E^{\times, \Sigma} = \{a \in \AA_E^\times :
%   a_v \in U_v \text{ for all } v \in \Sigma \}. 
% \]
% By weak approximation, $Z(\AA_F) \simeq \AA_E^\times = E^\times \cdot
% \AA_E^{\times, \Sigma}$.
Note that $\iota(a) \in Z(\AA_F) \simeq \AA_E^\times$.
By Lemma \ref{lem:replacesplit},
there exists some element 
$\widetilde{a} \in E^\times$ such that
% \todo{The proof of this is really similar to the
% proof of Proposition \ref{prop:globalkottwitz}, I 
% don't really want to write it again\dots}
$\widetilde{a} \in U_v$
for all $v \in \Sigma$, and
$\widetilde{a} \in \iota(a_v)\O_{E_v}^\times$
for all inert places $v$ of $F$ not in $\Sigma$.

Now, if $S$ is a set of normal representatives
of the regular semisimple
$\sigma$-conjugacy classes of $\GL_n(E)$,
then
\[
  S^{a} \coloneqq \{ \gamma \widetilde{a}^{-1} : \gamma \in S \}
\]
is also such a set. Let $X^a$ denote the corresponding
orbit space given by Definition \ref{def:glorbits}.
Similarly, 
\[
  R^{\beta, a} \coloneqq \{\zeta \widetilde{a}^{-1} : \zeta \in R^\beta\},
\]
has a corresponding orbit space $Y^{\beta, a}$
given by Definition \ref{def:uorbits}.
It is easy to see that $S^a$ and $R^{\beta, a}$ still
satisfy the first condition described in
Section \ref{sec:rep}.

Note that 
$\widetilde{f'^a} = (\widetilde{f'})^{\iota(a)}$.
Then
\begin{align*}
  \Orb_v^{\psi, \mu}([\gamma\widetilde{a}^{-1}, x, y],
  f_v'^a, \phi_v') &=
  \int_{\GL_n(E_v)} \widetilde{f_v'}(g^{-1}\gamma\ol g\widetilde{a}^{-1}\iota(a_v))\ol{(R_{\mu_v}(g)\phi_v')^\dagger(x, y)}\, dg \\
  % &= \int_{\GL_n(E_v)} \widetilde{f_v'}(g^{-1}\gamma\ol g)\ol{(R_{\mu_v}(g)\phi_v')^\dagger(x, y)}\, dg \\
  &= \Orb_v^{\psi, \mu}([\gamma, x, y], f_v', \phi_v')
\end{align*}
for all $v \in \Sigma$ and all inert $v \notin \Sigma$.
Similarly,
\begin{align*}
  \Orb_v^{\psi, \mu}([\zeta\widetilde{a}^{-1}, z]^{\beta}, (f_v^{\beta})^{\iota(a)}, \phi_{1,v}^\beta \otimes \phi_{2,v}^\beta)
  = \Orb_v^{\psi, \mu}([\zeta, z]^\beta, f_v^\beta, \phi_{1,v}^\beta
  \otimes \phi_{2,v}^\beta)
  % &= \iint\limits_{\U(V)(F_v) \times \U(V)(F_v)}
  % f_v(g^{-1}\zeta h\tilde z^{-1}\iota(z_v))(\phi_{2, v}\otimes \ol{\omega_{\psi_v', \mu_v}(h^{-1}g)\phi_{1,v}})^\ddagger(zh)\,dg\,dh  
\end{align*}
for all $v \in \Sigma$ and all inert $v \notin \Sigma$.
Furthermore, if $v$ is split, it is easy to check
that $f_v'^a$ and $(f_v^\beta)^{\iota(a)}$ are still related
as in Proposition \ref{prop:splitmatch}.
Thus, by Propositions \ref{prop:splitmatch},
\ref{prop:betaminusFL}, and Theorem \ref{thm:FL},
we have
\[
  \Orb_{v}^{\psi, \mu}([\gamma\widetilde{a}^{-1}, x, y], f_v'^a, \phi_v') =
  \Orb_{v}^{\psi, \mu}([\zeta\widetilde{a}^{-1}, z]^\beta, (f_v^\beta)^{\iota(a)}, \phi_{1,v}^\beta 
  \otimes \phi_{2,v}^\beta)
\]
for all $v$ and all matching regular semisimple
equivalence classes
$[\gamma\widetilde{a}^{-1}, x, y] \in X^a$
and $[\zeta\widetilde{a}^{-1}, z] \in Y^{\beta, a}$.

In addition, since $v_1, v_2 \in \Sigma$,
the test functions $(f'^a, \phi')$ and $((f^\beta)^{\iota(a)}, \phi_1^\beta \otimes \phi_2^\beta)$ are still good when
we use the new sets
of representatives $S^a$ and $R^{\beta, a}$.
The desired equality then follows from Propositions 
\ref{prop:GeoGL} and \ref{prop:GeoU}.
\end{proof}

% Unitary side:
% Set of places $\Sigma$?
% Subset of center that we integrate over?

\begin{definition}
For any irreducible cuspidal automorphic representation
$\pi$ of $G(\AA_F)$ and test function 
$(f, \phi_1 \otimes \phi_2)$, define
\[
  \sJ_\pi(f, \phi_1 \otimes \phi_2) =
  \sum_{\varphi}
  \frac{\P(\pi(f)\varphi, \phi_1)\ol{\P(\varphi, \phi_2)}}{\langle \varphi, \varphi \rangle_{\Pet}},
\]
where the sum runs over an orthogonal basis of $\pi$.
\end{definition}

\begin{definition}
Similarly, for any irreducible cuspidal automorphic
representation $\Pi$ of $G'(\AA_F)$ and test
function $(f', \phi')$, define
\[
  \sI_\Pi(s, f', \phi') = \sum_{\varphi}
  \frac{\lambda(s, \Pi(f')\varphi, \mu, \phi')\ol{\P'(\varphi)}}{\langle \varphi, \varphi \rangle_{\Pet}},
\]
where the sum runs over an orthogonal basis of $\Pi$.
Here,
\[
  \lambda(s, \varphi, \mu, \phi') \coloneqq \int\limits_{H_1(F) \bs H_1(\AA_F)} \varphi(g) \ol{\Theta_\mu(s, g, \phi')}\, dg
\]
is the Rankin--Selberg integral, and
\[
  \P'(\varphi) \coloneqq \int\limits_{H_2(F)Z_2(\AA_F)\bs H_2(\AA_F)} \varphi(h)\,dh.
\]
\end{definition}

\begin{proposition}
\label{prop:uspectral}
Let $\chi$ be a character of $Z(F)\bs Z(\AA_F)$.
If the test function $(f, \phi_1 \otimes \phi_2)$
is good, then
\[
  \int\limits_{Z(F, \Sigma)\bs Z(\AA_F, \Sigma)} \sJ(f^z, \phi_1 \otimes \phi_2)\chi(z)\,dz =
  \sum_{\pi} \sJ_\pi(f, \phi_1 \otimes \phi_2),
\]
where the sum runs over all irreducible
cuspidal automorphic representations of $G(\AA_F)$
with central character $\chi$.
Moreover, the right hand side is absolutely convergent.
\end{proposition}
\begin{proof}
Note that by weak approximation, $Z(\AA_F) =
Z(F)\cdot Z(\AA_F, \Sigma)$, so in fact 
$Z(F, \Sigma)\bs Z(\AA_F, \Sigma) \simeq Z(F)\bs Z(\AA_F)$.
The left hand side is equal to
\begin{align*}
  &\int\limits_{Z(F)\bs Z(\AA_F)}\sJ(f^z, \phi_1
  \otimes \phi_2)\chi(z)\, dz \\
  &= \int\limits_{Z(F)\bs Z(\AA_F)}\iint\limits_{(H(F) \bs H(\AA_F))^2} K_{f^z}(g, h)
  \ol{\Theta_{\psi', \mu}(g, \phi_1)}\Theta_{\psi', \mu}(h, \phi_2)\, dg\, dh\, dz \\
  &= \iint\limits_{(H(F)\bs H(\AA_F))^2}
  \left(\int\limits_{Z(F)\bs Z(\AA_F)} K_f(g, hz)\chi(z)\, dz\right) \ol{\Theta_{\psi', \mu}(g, \phi_1)}\Theta_{\psi', \mu}(h, \phi_2)\, dg\, dh.
\end{align*}
Now, note that
\[
  (R(f)\varphi)(g) = \int_{G(\AA_F)} f(h)\varphi(gh)\, dh
\]
is an operator on $L^2(G(F)\bs G(\AA_F), \chi)$.
Furthermore, for $\varphi \in L^2(G(F)\bs G(\AA_F), \chi)$, we have
\begin{align*}
  (R(f)\varphi)(g) 
  %&= \int_{G(\AA_F)} f(h)\varphi(gh)\, dh\\ 
  % &= \int_{G(\AA_F)} f(g^{-1}h)\varphi(h)\, dh\\
  &= \int\limits_{G(F)\bs G(\AA_F)} K_f(g, h)\varphi(h)\,dh \\
  &= \int\limits_{G(F)Z(\AA_F)\bs G(\AA_F)}
  \left( \int\limits_{Z(F)\bs Z(\AA_F)} K_f(g, hz)\chi(z)\,dz \right) \varphi(h)\, dh.
\end{align*}
By \cite[Lemma~5.5.4]{MR3228451}, if $\pi$ is
an irreducible admissible representation of $G(\AA_F)$ with
central character $\chi$ such
that $R(f)\varphi \neq 0$ for some $\varphi \in \pi$,
then $\pi_{v_1}$ is supercupsidal, and hence $\pi$
is cuspidal. (In fact, if $f_{v_1}$ is a truncated
matrix coefficient of a supercuspidal representation
$\pi_0$ of $G(F_v)$, we must have $\pi_{v_1} 
\simeq \pi_0 \otimes \chi_{v_1}\chi_0^{-1}$,
where $\chi_0$ is the central character of $\pi_0$.)
Thus, we have the spectral expansion
\[
  \int\limits_{Z(F)\bs Z(\AA_F)} K_f(g, hz)\chi(z)\,dz = \sum_{\pi} \sum_{\varphi}
  (\pi(f)\varphi)(g) \ol{\varphi(h)},
\]
where $\pi$ runs over all irreducible cuspidal
automorphic representations of $G(\AA_F)$ with
central character $\chi$, and $\varphi$ runs over an
orthonormal basis of $\pi$.

The spectral expansion in the proposition statement then
follows. The absolute convergence of the sum on the right hand
side follows from \cite[Proposition~A.1.2]{MR4332778}. 
\end{proof}

\begin{proposition}
\label{prop:glspectral}
Let $\chi'$ be a character of $Z'(F)\bs Z'(\AA_F)$
which is trivial on $Z_2(\AA_F)$. If the test
function $(f', \phi')$ is good, then
\[
  \int\limits_{Z'(F, \Sigma)Z_2(\AA_F)\bs Z'(\AA_F, \Sigma)}
  \sI(f'^{z}, \phi')\chi'(z)\,dz =
  \sum_{\Pi} \sI_\Pi(f', \phi'),
\]
where the sum runs over all irreducible cuspidal
automorphic representations of $G'(\AA_F)$ with
central character $\chi$. Moreover, the right hand side
is absolutely convergent.
\end{proposition}
\begin{proof}
The proof is the same as the proof of Proposition \ref{prop:uspectral}.
\end{proof}

\subsection{Relative trace identity}
We obtain the following relative trace
identity.
% Let $(f', \phi')$ and 
% $\{(f^V, \phi_1^V \otimes \phi_2^V) \}_{V \in [\SHerm_n^\times(F)]}$ be good matching test functions, which
% are all factorizable.
% Suppose that at the place $v_1$, the test function
% $f_{v_1}^V$ is a truncated matrix coefficient of
% a supercupsidal representation $\pi_0$ of $G(F_v)$.
% Then $f'_{v_1}$ is a truncated matrix coefficient
% of $\Pi_0 \coloneqq \pi_0 \boxtimes \prescript{\sigma}{}{\pi_0^\vee}$.
% Let $\chi$ be a character of $Z(F) \bs Z(\AA_F)$
% whose component at $v_1$ is the central character of
% $\pi_0$. Then by Lemma \ref{lem:distrequal} and
% Propositions \ref{prop:uspectral}
% and \ref{prop:glspectral}, we have
% \begin{equation}
%   \sum_{\Pi} \sI_\Pi(f', \phi') = 
%   \sum_{V \in [\SHerm_n^\times(F)]} \sum_{\pi} 
%   \sJ_\pi^V(f^V, \phi_1^V \otimes \phi_2^V).
% \end{equation}
% On the left hand side, the sum runs over all
% irreducible cuspidal automorphic representations $\Pi$
% of $G'(\AA_F)$ such that $\Pi_v \simeq \Pi_0$ and the
% central character of $\Pi$ is 
% $\chi \otimes \prescript{\sigma}{}{\chi^{-1}}$.
% On the right hand side, the inner sum runs over
% all irreducible cuspidal automorphic representations
% $\pi$ of $G(\AA_F)$ such that $\pi_{v_1} \simeq \pi_0$
% and the central character of $\pi$ is $\chi$.

\begin{proposition}
\label{prop:identity}
Let $\pi$ be an irreducible cuspidal automorphic
representation of $\GL_n(\AA_E)$ such that $\pi_{v_1}$
is supercuspidal. 
Let $\Pi = \pi \boxtimes \prescript{\sigma}{}{\pi^\vee}$ be a representation of $G'(\AA_F)$.
For good matching test functions $(f', \phi')$ and
$\{(f^\beta, \phi_1^\beta \otimes \phi_2^\beta)\}_{\beta \in
[\SHerm_n^\times(F)]}$, 
we have
\[
  \sI_{\Pi}(f', \phi') =
  \zeta_E^\ast(1)^{-1}\sum_{\beta \in [\SHerm_n^\times(F)]}
  \sJ_{\pi}^\beta(f^\beta, \phi_1^\beta \otimes \phi_2^\beta).
\]
\end{proposition}
\begin{proof}
Let $\chi$ be the central character of $\pi$.
By Lemma \ref{lem:distrequal} and Propositions
\ref{prop:uspectral} and \ref{prop:glspectral}, we 
have
\begin{equation}
\label{eq:spectral}
  \sum_{\Pi} \sI_\Pi(f', \phi') = 
  \zeta_E^\ast(1)^{-1}\sum_{\beta \in [\SHerm_n^\times(F)]} \sum_{\pi} 
  \sJ_\pi^\beta(f^\beta, \phi_1^\beta \otimes \phi_2^\beta),
\end{equation}
where the sum on the left runs over irreducible
cupsidal automorphic representations $\Pi$
of $G'(\AA_F)$ with central character
$\chi \otimes \prescript{\sigma}{}{\chi^{-1}}$
and the inner sum on the right runs over irreducible
cupsidal automorphic representations of $G(\AA_F)$
with central character $\chi$.

Now, let $\Sigma$ denote the set of nonarchimedean
split places $v$ of $F$ not equal to $v_1, v_2$,
such that
\begin{gather*}
  \quad (f_v',  \phi_v') = \left(\frac{1}{\Vol(K_v')}\mathbbm{1}_{K_v'}, \mathbbm{1}_{\O_{E_v, n}}\right), \\
  (f_v^\beta, \phi_{1,v}^\beta \otimes \phi_{2,v}^\beta) = \left(\frac{1}{\Vol(\U_n^\beta(\O_{F_v}))^2}\mathbbm{1}_{K_v}, \mathbbm{1}_{\LL(\O_{F_v})} \otimes \mathbbm{1}_{\LL(\O_{F_v})}\right)
\end{gather*}
for all $\beta \in [\SHerm_n^\times(F)]$,
where $K_v' = \GL_n(\O_{E_v}) \times \GL_n(\O_{E_v})$ 
and $K_v = \GL_n(\O_{E_v})$.
Note that if $\pi_v$ is not unramified for some $v \in \Sigma$, then both sides of the desired equality 
are zero. Thus we assume that $\pi$ is unramified
at all $v \in \Sigma$, and we let
\begin{align*}
  \Lambda_{\Pi_\Sigma} \colon \H(G_{\Sigma}', K_\Sigma') \coloneqq \bigotimes_{v \in \Sigma}'\H(G'(F_v), K_v') \to \CC \\
  \lambda_{\pi_\Sigma} \colon \H(G_\Sigma, K_\Sigma)
  \coloneqq \bigotimes_{v \in \Sigma}'\H(G(F_v), K_v) \to \CC
\end{align*}
be the characters of the split spherical Hecke
algebras of $\Pi_\Sigma = \otimes_{v \in \Sigma}' \Pi_v$ and 
$\pi_\Sigma = \otimes_{v \in \Sigma}' \pi_v$.

Recall that we have a smooth transfer map 
$\tau\colon \H(G_\Sigma', K_\Sigma') \to \H(G_\Sigma, K_\Sigma)$
given by Proposition \ref{prop:splitmatch}, such that
\[
  \tau(\alpha_1 \otimes \alpha_2) = 
  \alpha_1 \ast \prescript{\sigma}{}{\alpha_2^\vee},
\]
where $\alpha_1, \alpha_2 \in \H(G_\Sigma, K_\Sigma)$
and $\prescript{\sigma}{}{\alpha_2^\vee}(g)
\coloneqq \alpha_2(\ol{g}^{-1})$. Note that
\[
  \lambda_{\pi_\Sigma}(\tau(\alpha)) =
  \Lambda_{\pi_\Sigma \boxtimes \prescript{\sigma}{}{\pi_\Sigma^\vee}}(\alpha).
\]
Now, for each $\alpha \in
\H(G_\Sigma', K_\Sigma')$,
we have that 
$(f' \ast \alpha, \phi')$ and
$\{(f^{\beta} \ast \tau(\alpha), \phi_1^\beta \otimes \phi_2^\beta)\}_{\beta \in [\SHerm_n^\times(F)]}$ 
are also good matching test functions.
% $(f'^\Sigma\otimes \alpha, \phi')$ and
% $\{(f^{\beta, \Sigma} \otimes \tau(\alpha), \phi_1^\beta \otimes \phi_2^\beta)\}_{\beta \in [\SHerm_n^\times(F)]}$ are good matching test functions, 
% where $f'^\Sigma \coloneqq \prod_{v \notin \Sigma}f_v'$ and
% $f^{\beta, \Sigma} \coloneqq \prod_{v \notin \Sigma} f_v^\beta$.
Note that
\[
  \sI_\Pi(f' \ast \alpha, \phi') =
\Lambda_{\Pi_{\Sigma}}(\alpha)\sI_\Pi(f', \phi'),
\]
and 
\[
  \sJ_\pi^\beta(f^{\beta} \ast \tau(\alpha), \phi_1^\beta \otimes \phi_2^\beta) = 
  \lambda_{\pi_{\Sigma}}(\tau(\alpha))
  \sJ_\pi^\beta(f^\beta, \phi_1^\beta \otimes \phi_2^\beta).
\]
Thus, by \eqref{eq:spectral}, we have
\[
  \sum_{\Pi} \Lambda_{\Pi_\Sigma}(\alpha)\sI_{\Pi}(f', \phi')
  = \zeta_E^\ast(1)^{-1}\sum_{\beta \in [\SHerm_n^\times(F)]}\sum_{\pi} \lambda_{\pi_\Sigma}(\tau(\alpha))
  \sJ_\pi^\beta(f^\beta, \phi_1^\beta \otimes \phi_2^\beta).
\]
Now, we use the fact that the characters of the
split spherical Hecke algebra for different 
representations are linearly independent.
In addition, by \cite[Corollary~B]{MR3267122}, if two irreducible cuspidal automorphic
representations of $\GL_n(\AA_E)$ are isomorphic
at almost all places of $E$ lying over split places
of $F$, then they are isomorphic.
These two facts together imply the proposition.
\end{proof}

\section{Local distributions}
\subsection{Periods on the general linear group}
In this subsection, we recall some facts about the 
decomposition of the global periods
$\lambda$ and $\P'$ into products of local invariant linear forms.

Let $\pi$ be a cuspidal automorphic representation of $\GL_n(\AA_E)$. We extend $\psi$ to a character of $N_n(\AA_E)$ by
\[
	\psi(u) = \psi\left(\sum_{i=1}^{n-1} u_{i,i+1} \right),
	\quad u = (u_{ij}).
\]
Similar conventions apply to other unipotent groups such
as $N_n(E_v)$, $N_n(\AA_F)$, $N_n(F_v)$.
Let $\C^\infty(N_n(\AA_E)\bs \GL_n(\AA_E), \psi)$ denote the
space of smooth functions $f$ on $\GL_n(\AA_E)$ such that
\[
	f(ug) = \psi(u)f(g), \quad u \in N_n(\AA_E), g \in \GL_n(\AA_E).
\]
Similarly, we define its local counterpart
$\C^\infty(N_n(E_v) \bs \GL_n(E_v), \psi_v)$ for any place $v$
of $F$.

For $\varphi \in \pi$, its Whittaker--Fourier coefficient is defined by
\[
	W_\varphi(g) = \int\limits_{N_n(E) \bs N_n(\AA_E)}
	\varphi(ug) \ol{\psi(u)} \, du.
\]
The Whittaker model of $\pi$, denoted $\W(\pi, \psi)$, is 
the image of the
equivariant embedding $\varphi \mapsto W_\varphi \colon \pi \hookrightarrow
\C_c^\infty(N_n(\AA_E) \bs \GL_n(\AA_E), \psi)$.

Similarly, we have the local Whittaker model
$\W(\pi_v, \psi_v)$, which is a subspace of $\C_c^\infty(N_n(E_v) \bs \GL_n(E_v), \psi_v)$.

\begin{definition}
For $W_{1,v}, W_{2,v} \in W(\pi_v, \psi_v)$, let
\[
	\langle W_{1,v}, W_{2,v} \rangle_v
	= \int\limits_{N_{n-1}(E_v) \bs \GL_{n-1}(E_v)}
	W_{1,v}\left( \begin{pmatrix}
	g & \\ & 1
	\end{pmatrix}\right)\ol{W_{2,v}\left( \begin{pmatrix}
	g & \\ & 1
	\end{pmatrix}\right)}\,dg.
\]
This integral converges absolutely since $\pi_v$ is generic unitary.
\end{definition}

If $\pi$ and $\psi$ are unramified at $v$, and 
$W_v$ is the $\GL_n(\O_{E_v})$-fixed vector normalized
so that $W_v(1) = 1$, we have
\[
	\langle W_v, W_v \rangle_v = \Vol(\GL_n(\O_{E_v}))
	L(1, \pi_v \times \pi_v^\vee).
\]
Thus, we define a normalized inner product
\[
	\langle W_{1,v}, W_{2,v}\rangle_v^\natural
	= \frac{\langle W_{1,v}, W_{2,v}\rangle_v}{L(1, \pi_v \times \pi_v^\vee)}.
\]

We have the following decomposition of the Petersson inner
product.

\begin{proposition}[{\cite[Proposition~3.1]{MR3164988}}]
\label{prop:petdecomp}
Let $\varphi_1, \varphi_2 \in \pi$. Assume 
$W_{\varphi_1} = \otimes_v W_{1,v}$ and
$W_{\varphi_2} = \otimes_v W_{2,v}$. Then
\[
	\langle \varphi_1, \varphi_2 \rangle_{\Pet}
	= \zeta_E^\ast(1)^{-1} \Res_{s=1} L(s, \pi \times \pi^\vee)
	\prod_v \langle W_{1,v}, W_{2,v} \rangle_v^\natural.
\]
% where $\AA_E^1$ is the subgroup of $\AA_E^\times$ consisting of
% adeles of absolute value one.
\end{proposition}
Note that the constant here is different from the constant in
\cite[Proposition~3.1]{MR3164988} due to different
normalizations of the Petersson inner product.

% Under the measure in \cite{MR3164988}, 
%$\Vol(E^\times \bs \AA_E^1) = \zeta_E^\ast(1)$.
% Under our measure, it is $1$.

Now, let $\pi_1$, $\pi_2$ be cuspidal automorphic representations
of $\GL_n(\AA_E)$. Write $\Pi = \pi_1 \boxtimes \pi_2$.
For $\varphi = \varphi_1 \otimes \varphi_2 \in \Pi$ and
$\phi' \in \S(\AA_{E,n})$,
recall that we defined the global Rankin--Selberg integral
\[
	\lambda(s, \varphi, \mu, \phi') = \int\limits_{\GL_n(E) \bs \GL_n(\AA_E)} \varphi_1(g)\varphi_2(g) \ol{\Theta_\mu(s, g, \phi')}\, dg.
\]
When $s = \frac12$, we write
$\lambda(\varphi, \mu, \phi') = \lambda(\frac12, \varphi, \mu, \phi')$.

\begin{definition}
For $W_{1,v} \in \W(\pi_{1,v}, \psi_v)$,
$W_{2,v} \in \W(\pi_{2,v}, \ol\psi_v)$, and $\phi_v'
\in \S(E_{v,n})$,
define the local Rankin--Selberg period
\begin{align*}
	\lambda_v(s, &W_{1,v}, W_{2,v}, \mu_v, \phi_v')\\
	&= \int\limits_{N_n(E) \bs \GL_n(E)} W_{1,v}(g)W_{2,v}(g)
	\ol{\mu_{v}(\det g)\phi_v'(e_n g)} |\det g|^{s}\, dg,
\end{align*}
where $e_n = (0, \dots, 0, 1) \in E_n$.
This integral is absolutely convergent for $\Re(s)$ large enough
and has a meromorphic continuation to the whole complex plane.
\end{definition}

Also define the normalized version
\[
	\lambda_v^\natural(s, W_{1,v}, W_{2,v}, \mu_v, \phi_v')
	= \frac{\lambda_v(s, W_{1,v}, W_{2,v}, \mu_v, \phi_v')}{L(s, \pi_{1,v} \times \pi_{2,v} \otimes \mu_v^{-1})},
\]
which is homolorphic for all $s$.
Let 
\[
	\lambda_v^\natural(W_{1,v}, W_{2,v}, \mu_v, \phi_v')
	= \lambda_v^\natural(\tfrac12, W_{1,v}, W_{2,v}, \mu_v, \phi_v').
\]
If $\Pi$, $\psi$, and $\mu$ are all
unramified at $v$, and $W_{1,v}$ and $W_{2,v}$ are both $\GL_n(\O_{E_v})$-fixed
and normalized so that $W_{1,v}(1) = W_{2,v}(1) = 1$, 
we have
\[
	\lambda_v^\natural(W_{1,v}, W_{2,v}, \mu_v, \phi_v')
	= \Vol(\GL_n(\O_{E_v})).
\]
If $\Pi_v$ is tempered, $\lambda_v(s, W_{1,v}, W_{2,v},
\mu_v, \phi_v')$ is holomorphic at $s = \frac12$, and we
can define 
$\lambda_v(W_{1,v}, W_{2,v}, \mu_v, \phi_v')
= \lambda_v(\frac12, W_{1,v}, W_{2,v}, \mu_v, \phi_v')$
directly.

\begin{proposition}[\cite{MR701565}]
\label{prop:rs}
Let $\varphi = \varphi_1 \otimes \varphi_2 \in \Pi$ and
$\phi' \in \S(\AA_{E,n})$.
Then we have the decomposition of the Rankin--Selberg period
\[
	\lambda(\varphi, \mu, \phi')
	= \zeta_E^\ast(1)^{-1}L(\tfrac12, \pi_1 \times \pi_2 \otimes \mu^{-1})
	\prod_v \lambda_v^\natural(W_{1,v}, W_{2,v}, \mu_v, \phi_v'),
\]
where 
$W_{\varphi_1} = \otimes_v W_{1,v}
 \in \W(\pi_1, \psi)$ and $W_{\varphi_2} = \otimes_v W_{2,v} \in
\W(\pi_2, \ol{\psi})$.
\end{proposition}

\subsection{Periods on the unitary group}
\label{sec:unitaryperiods}
For each place $v$ of $F$, the local conjecture
states that 
\[
	\dim\Hom_{\U(V)(F_v)}(\pi_v \otimes \ol{\omega_{\psi_v', \mu_v}}, \CC) \le 1.
\]
At least under the assumptions of Theorem \ref{thm:maintheorem}, the
local conjectue is true.
Indeed, at split places, it holds by \cite[Theorem~B]{MR2999291}
and \cite[Theorem~C]{MR2874638}.
At inert places, since $\pi_v$ is unramified, 
it holds by \cite[Corollary~4.8]{MR4622393}.
Thus, for the rest of this subsection, we assume we are in a
situation in which the local conjecture holds.

For each $v$, let $\ell_v$ be a generator of the space
$\Hom_{\U(V)(F_v)}(\pi_v \otimes \ol{\omega_{\psi_v', \mu_v}}, \CC)$. 
normalized so that when all data
is unramified, including $\varphi_v$ and $\phi_v$,
we have $\ell_v(\varphi_v, \phi_v) = 1$.
Then there exists a constant $C$ such that
\begin{equation}
\label{eq:factorPell}
	\P(\varphi, \phi) =
	C \times \prod_v \ell_v(\varphi_v, \phi_v)
\end{equation}
for all $\varphi = \otimes_v \varphi_v$,
$\phi = \otimes_v \phi_v$. 

\begin{proposition}
\label{prop:alphaconvergence}
Let $v$ be a place of $F$, and assume that $\pi_v$ is tempered.
The integral defining
\[
	\alpha_v(\varphi_v, \phi_v, \varphi_v', \phi_v') =
	\int\limits_{\U(V)(F_v)} \langle \pi_v(h)\varphi_v, \varphi_v'
	\rangle_v \langle \omega_{\psi'_v, \mu_v}(h_v) \phi_v,
	\phi_v'\rangle_v \, dh
\]
converges absolutely, for all $\varphi_v, \varphi_v' \in \pi_v$ and
$\phi_v, \phi_v' \in \S(\LL(F_v))$.
\end{proposition}
\begin{proof}
We make use of the estimates from \cite[Appendix~B]{MR3505397}.

We view $\U(V)$ as an algebraic group over $F_v$,
and we write $|\phantom{x}|_v \coloneqq |\phantom{x}|_{E_v}$.
Let $r$ be the split rank of $\U(V)$, and suppose that
\[
	\beta = \begin{pmatrix}
	& & & & & & 1 \\
	& & & & & \iddots & \\
	& & & & 1 & & \\
	& & & \beta' & & & \\
	& & 1 & & & & \\
	& \iddots & & & & & \\
	1 & & & & & &
	\end{pmatrix}
\]
where there are $r$ ones on each half of the diagonal.
Let
\[
	A_0 \coloneqq \left\{ \diag(a_1,
	\dots, a_r, 1, \dots, 1, \sig{a_r}^{-1},
	\dots, \sig{a_1}^{-1}) \right\}
\]
be a maximal split torus of $\U(V)$, where $\diag(a_1, \dots, a_n)$
denotes the diagonal matrix with entries $a_1, \dots, a_n$.
Let
\[
	A_0^+ \coloneqq \left\{ \diag(a_1,
	\dots, a_r, 1, \dots, 1, \sig{a_r}^{-1},
	\dots, \sig{a_1}^{-1}) :
	|a_1|_v \le \cdots
	\le |a_r|_v \le 1\right\}.
\]
Let $P_0$ be the minimal parabolic subgroup of $\U(V)$ whose Levi
is the centralizer of $A_0$. Let $K = \U(V)(\O_{F_v})$.

Let $B_n$ be the upper triangular matrices in
$\GL_n(E)$.
For $g = (a_{ij}) \in \GL_n(E_v)$ with
$g^{-1} = (b_{ij})$, let
\[
	\|g\| = \max\{ |a_{ij}|_v, |b_{ij}|_v : 1 \le i, j \le n \},
	\quad \sigma(g) = \log\|g\|.
\]
Let $\Xi$ be the Harish-Chandra function
for $\GL_n(E)$.
The functions $\Xi$, $\sigma$, and $\|\phantom{x}\|$ are
$\GL_n(\O_{E_v})$-bi-invariant functions on $\GL_n(E_v)$.

The modulus characters of $P_0$ and $B_n$ are
given by
\[
	\delta_{P_0}(m) = \delta_{B_n}(m)^{1/2}
	= |a_1|_v^{n-1} \cdots |a_r|_v^{n-2r+1},
\]
for $m = \diag(a_1, \dots, a_r, 1, \dots,
1, a_r^{\tau, -1}, \dots, a_1^{\tau, -1})$.

Now, we have the following bounds from \cite[Appendix~B]{MR3505397}.
We have a Cartan decomposition $\U(V)(F_v) = KA_0^+K$.
For any $f \in L^1(\U(V)(F_v))$ we have
\[
	\int_{\U(V)(F_v)}f(h)\,dh = \int_{A_0^+}\nu(m)
	\int_{K \times K}f(k_1mk_2)\,dk_1 \,dk_2\, dm,
\]
where $\nu(m)$ is a positive function on $A_0^+$ such that
there is a constant $A > 0$ with
\[
	A^{-1}\delta_{P_0}(m)^{-1} \le \nu(m) \le A\delta_{P_0}(m)^{-1}.
\]
There exist constants $C_1$, $C_2$ and a positive real number
$d$ such that
\[
	C_1\delta_{B_n}(m)^{1/2}
	\le \Xi(m) \le C_2\delta_{B_n}(m)^{1/2}
	(1 + \sigma(m))^d.
\]
There exist constants $A$ and $B$ such that
\[
	|\langle \pi_v(h)\varphi_v, \varphi_v' \rangle_v|
	\le A\Xi(g)(1 + \sigma(g))^B
\]
for all $g \in \GL_n(E)$.
Finally, there is a constant $A$ such that
\[
	%\int_{K \times K} 
	|\langle \omega_{\psi_v', \mu_v}(k_1ak_2) \phi, \phi' \rangle|
	\le A|a_1 \cdots a_r|_v^{1/2},
\]
for all $k_1, k_2 \in K$ and $m \in A_0^+$ as above.

Thus, we have, up to a positive constant,
\begin{align*}
&\int\limits_{\U(V)(F_v)} |\langle \pi_v(h)\varphi_v, \varphi_v'
	\rangle_v|
	|\langle\omega_{\psi'_v, \mu_v}(h_v) \phi_v,
	\phi_v'\rangle_v| \, dh \\
&= \int_{A_0^+}\nu(m)\int_{K \times K}
|\langle \pi_v(k_1mk_2)\varphi_v, \varphi_v'\rangle_v|
|\langle \omega_{\psi_v', \mu_v}(k_1mk_2)\phi_v, \phi_v'\rangle|
\,dk_1\,dk_2\,dm\\
&\le \int_{A_0^+} \delta_{P_0}(m)^{-1}
\int_{K \times K}\Xi(m)(1+\sigma(m))^B
|a_1\cdots a_0|_v^{1/2}
\,dk_1\,dk_2\,dm \\
&\le \int_{A_0^+} (1+\sigma(m))^{B+d}
|a_1 \cdots a_r|_v^{1/2}\, dm,
\end{align*}
where  $m = \diag(a_1, \dots, a_r, 1, \dots,
1, a_r^{\tau, -1}, \dots, a_1^{\tau, -1})$ as before.
A straightforward calculation gives $\sigma(m) = - \log |a_1|_v$,
so we have to show that
\[
	\int_{A_0^+} (1 - \log|a_1|_v)^{B+d}|a_1 \cdots a_r|_v^{1/2}\,dm
\]
converges.
In fact, this is bounded by
\begin{align*}
\int\limits_{|a_1|_v \le 1} \cdots
\int\limits_{|a_r|_v \le 1} (1 - \log|a_1|_v)^{B+d}|a_1 \cdots
a_r|_v^{1/2} \,da_1 \cdots \,da_r,
\end{align*}
where the $d a_i$ are the multiplicative Haar measure
on $E_v^\times$.
This is a product of the integrals
\[
	\int_{|a_1|_v \le 1} (1-\log|a_1|_v)^{B+d}|a_1|_v^{1/2}\,
	da_1, \quad
	\int_{|a_i|_v \le 1} |a_i|_v^{1/2} \,da_i,
\]
for $2 \le i \le r$,
which are easily seen to be convergent.
\end{proof}

We now recall one result from \cite{MR3505397}.
Let $v$ be a split place of $F$, so $E_v = F_v \times F_v$.
Suppose $\pi_v = \pi_1 \boxtimes \pi_2$ for representations
$\pi_1$ and $\pi_2$ of $\GL_n(F_v)$, and $\mu_v = (\mu_1, \mu_1^{-1})$.

% Note that the representation $\pi_2'$ given by
% $\pi_2'(g) \coloneqq \pi_2(\transp g^{-1})$ is 
% isomorphic to $\ol{\pi_2}$.
% Then for $\varphi_1 \in \pi_1$, $\varphi_2 \in \pi_2$,
% and $\phi_{1,v}, \phi_{2,v} \in \S(F_{v,n})$ we have
% \begin{align*}
% &\alpha_v(\varphi_1 \otimes \varphi_2, \phi_{1,v},
% \varphi_1', \varphi_2', \phi_{2,v})\\
% &=\zeta_{F_v}(1)^{-1}
% \int_{\GL_n(F_v)}
% \langle \pi_1(g)\varphi_1 ,\varphi_1'\rangle
% \langle \pi_2'(g)\varphi_2,\varphi_2'\rangle
% \langle R_{\mu_1}(g)\phi_{1,v}, \phi_{2,v}\rangle\,dg.
% \end{align*}

\begin{definition}
\label{def:alphaWhittaker}
For $W_1, W_1' \in \W(\pi_1, \psi_v')$ and 
$\ol{W_2}, \ol{W_2'} \in \W(\ol{\pi_2}, \ol{\psi_v'})$, define
\begin{align*}
	&\alpha(W_1, \ol{W_2}, \phi_{1,v}, W_1', \ol{W_2'}, \phi_{2,v}) \\
	&= \int_{\GL_n(F_v)}
	\langle \pi_1(g)W_1, W_1' \rangle_v
	\langle \ol{\pi_2}(g)\ol{W_2}, \ol{W_2'}\rangle_v
	\langle R_{\mu_1}(g) \phi_{1,v}, \phi_{2,v}\rangle_v \, dg
\end{align*}
\end{definition}

\begin{proposition}[{\cite[Proposition~6.2.2]{MR3505397}}]
\label{prop:alphalambda}
For all $W_1, W_1' \in \W(\pi_1, \psi_v')$ and 
$\ol{W_2}, \ol{W_2'} \in \W(\ol{\pi_2}, \ol{\psi_v'})$,
we have
\[
	\alpha(W_1, \ol{W_2}, \phi_{1,v}, W_1', \ol{W_2'}, \phi_{2,v})
	= \zeta_{F_v}(1)\lambda_v(W_1, \ol{W_2}, \mu_1, \phi_{1,v})
	\ol{\lambda_v(W_1', \ol{W_2'}, \mu_1, \phi_{2,v})}.
\]
\end{proposition}

Note that the extra factor of $\zeta_{F_v}(1)$ comes from the different
normalization of measures on $\GL_n(F_v)$.

As noted in the introduction, if we assume Conjecture \ref{conj:unrefined},
% then since 
% $	\P \otimes \ol{\P} \in \Hom_{\U(V)(\AA_F)}
% 	(\pi \otimes \ol{\omega_{\psi', \mu}} \otimes \ol{\pi} \otimes
% 	\omega_{\psi', \mu}, \CC)$
% and
% $\alpha_v^\natural$ is a nonzero element of $\Hom_{\U(V)(F_v)}(\pi_v \otimes \ol{\omega_{\psi_v', \mu_v}} \otimes \ol{\pi_v} \otimes \omega_{\psi_v', \mu_v}, \CC)$ for each $v$, 
then there exists a constant $C$ such that
\begin{equation}
\label{eq:factorP}
	|\P(\varphi, \phi)|^2 = C \times
	\prod_v \alpha_v^\natural(\varphi_v, \phi_v, \varphi_v, \phi_v)
\end{equation}
for all $\varphi = \otimes_v \varphi_v$,
$\phi = \otimes_v \phi_v$. 
(At split places $v$, the nonvanishing of $\alpha_v^\natural$
follows from the nonvanishing of $\lambda_v$.)

\subsection{Local distributions}
Let $\pi$ be a cuspidal automorphic representation of 
$\GL_n(\AA_E)$. 
Let $v$ be a place of $F$.

\begin{definition}
Let $f_v \in \C_c^\infty(G(F_v))$, $\phi_{1,v} \otimes 
\phi_{2,v} \in \S(\LL(F_v))^{\otimes 2}$.
Define the local distribution
\[
	\sJ_{\pi_v}(f_v, \phi_{1,v} \otimes \phi_{2,v})
	= \sum_{\varphi_v} 
	\frac{\ell_v(\pi_v(f_v)\varphi_v, \phi_{1,v})
	\ol{\ell_v(\varphi_v, \phi_{2,v})}}
	{\langle \varphi_v, \varphi_v\rangle_v},
\]
where the sum runs over an orthogonal basis of $\pi_v$.
If $\pi_v$ is tempered, in fact define
\begin{align*}
	J_{\pi_v}(f_v, \phi_{1,v} \otimes \phi_{2,v})
	= \sum_{\varphi_v}
	\frac{\alpha_v(\pi_v(f_v)\varphi_v, \phi_{1,v}, \varphi_v, \phi_{2,v})}{\langle \varphi_v, \varphi_v \rangle_v},
\end{align*}
and also define the normalized version
\[
	J_{\pi_v}^\natural(f_v, \phi_{1,v} \otimes \phi_{2,v})
	= \sum_{\varphi_v}
	\frac{\alpha_v^\natural(\pi_v(f_v)\varphi_v, \phi_{1,v}, \varphi_v, \phi_{2,v})}{\langle \varphi_v, \varphi_v \rangle_v}.
\]
Again, the sums run over an orthogonal basis of $\pi_v$.
\end{definition}

Let $\Pi = \pi \boxtimes \sig{\pi^\vee}$ be a cuspidal 
automorphic representation of $G'(\AA_E)$.
Note that
\[
	\W(\sig{\pi_v^\vee}, \ol\psi_v) = 
	\{ \sig{\ol{W_v}}(g) \coloneqq \ol{W_v(\sig{g})} :
	W_v \in \W(\pi_v, \psi_v) \}.
\]

\begin{definition}
Let $f_v' \in \C_c^\infty(G'(F_v))$,
$\phi_v' \in \S(E_{v,n})$. Define the local distribution
\begin{align*}
	I_{\Pi_v}(s, f_v', \phi_v') 
	% &=
	% \sum_{W_{1,v}, W_{2,v}}
	% \frac{\lambda_v(\pi_v(f_{1,v}')W_{1,v}, \sig{\pi_v^\vee}(f_{2,v}')W_{2,v}, \mu_v, \phi_v')
	% \langle W_{1,v}, \sig{W_{2,v}^\vee} \rangle_v}
	% {\langle W_{1,v}, W_{1,v}\rangle_v \langle W_{2,v}, W_{2,v}\rangle_v} \\
	&= \sum_{W_v}
	\frac{\lambda_v(s, \pi_v(f_{1,v}')W_v,
	\sig{\pi_v^\vee}(f_{2,v}') \sig{\ol{W_v}}, \mu_v, \phi_v') }{\langle W_v, W_v \rangle_v},
\end{align*}
where
% $W_{1,v}$ runs over an orthogonal basis of
% $\W(\pi_v, \psi_v)$ and $W_{2,v}$ runs over an
% orthogonal basis of $\W(\sig{\pi_v^\vee}, \ol\psi_v)$, while 
the sum runs over an orthogonal basis
of $\W(\pi_v, \psi_v)$.
If $\sI_{\Pi_v}(s, f_v', \phi_v')$ is holomorphic at $s = \frac12$,
we write $\sI_{\Pi_v}(f_v', \phi_v') = \sI_{\Pi_v}(\frac12,
f_v', \phi_v')$.

Also define the normalized version
\[
	I_{\Pi_v}^\natural(f_v', \phi_v')
	= \sum_{W_v} \frac{\lambda_v^\natural(\pi_v(f_{1,v}')W_{v}, \sig{\pi_v^\vee}(f_{2,v}')\sig{\ol{W_v}}, \mu_v, \phi_v')}
	{\langle W_v, W_v \rangle_v^\natural},
\]
where the sum runs over an orthogonal basis of
$\W(\pi_v, \psi_v)$.
\end{definition}

\begin{proposition}
\label{prop:Idecomp}
As before, let $\Pi = \pi \boxtimes \sig{\pi^\vee}$.
Suppose $f' = \otimes_v f_v'$ and $\phi' = \otimes_v \phi_v'$
are factorizable.
Then
\[
 	\sI_\Pi(f', \phi') = \frac{L(\tfrac12, \pi \times \sig{\pi^\vee} \otimes \mu^{-1})}{\Res_{s=1} L(s, \pi \times \pi^\vee)} 
	\prod_v \sI_{\Pi_v}^\natural(f_v', \phi_v').
\]
\end{proposition}
\begin{proof}
Note that for $\varphi_1, \varphi_2 \in \pi$, we have
$\P'(\varphi_1 \otimes \sig{\varphi_2^\vee}) = \langle \varphi_1, \varphi_2 \rangle_{\Pet}$, where
$\sig{\varphi_2^\vee}(g) \coloneqq \ol{\varphi_2(\sig{g})}$.
Thus, we have
\begin{align*}
	\sI_\Pi(f', \phi') 
	&= \sum_{\varphi_1, \varphi_2} 
	\frac{\lambda(\Pi(f')(\varphi_1 \otimes \sig{\varphi_2^\vee}), \mu, \phi')\langle \varphi_1, \varphi_2 \rangle_{\Pet}}
	{\langle \varphi_1, \varphi_1 \rangle_{\Pet}
	\langle \varphi_2, \varphi_2 \rangle_{\Pet}}\\
	&= \sum_{\varphi} 
	\frac{\lambda(\Pi(f')(\varphi \otimes \sig{\varphi^\vee}), \mu, \phi')}
	{\langle \varphi, \varphi \rangle_{\Pet}},
\end{align*}
where $\varphi_1, \varphi_2$ and
$\varphi$ run over an orthogonal basis of $\pi$.
Then by Propositions \ref{prop:rs} and \ref{prop:petdecomp},
we have
\begin{align*}
	\sI_{\Pi}(f', \phi')
	% &= \sum_{\varphi_1, \varphi_2
	% \in \OB(\pi)} Z(\tfrac12, \Pi(f')(\varphi_1 \otimes 
	% \sig{\varphi_2^\vee}), \mu, \phi')
	% \langle \varphi_1, \varphi_2 \rangle_{\Pet} \\
	% &= \sum_{\varphi \in \OB(\pi)}
	% Z(\tfrac12, \Pi(f')(\varphi \otimes \sig{\varphi^\vee}), 
	% \mu, \phi') \\
	&= \sum_{\varphi}
	\frac{\lambda(\Pi(f')(\varphi \otimes \sig{\varphi^\vee}), 
	\mu, \phi')}{\langle \varphi, \varphi \rangle_{\Pet}} \\
	&= \frac{L(\tfrac12, \pi \times \sig{\pi^\vee} \otimes \mu^{-1})}{\Res_{s=1} L(s, \pi \times \pi^\vee)} \sum_{W} 
	\prod_v \frac{\lambda_v^\natural(\pi_v(f_{1,v}')W_{v}, \sig{\pi_v^\vee}(f_{2,v}')\sig{\ol{W_v}}, \mu_v, \phi_v')}{\langle W_{v}, W_{v} \rangle_v^\natural}\\
	&= \frac{L(\tfrac12, \pi \times \sig{\pi^\vee} \otimes \mu^{-1})}{\Res_{s=1} L(s, \pi \times \pi^\vee)} 
	\prod_v \sI_{\Pi_v}^\natural(f_v', \phi_v'), 
\end{align*}
where in the second line,
$W$ runs over an orthogonal basis of  $\W(\pi, \psi)$.
\end{proof}

\subsection{Local distribution identity}
At split and unramified places of $F$, we have the following
identity between the local distributions 
$\sI_{\Pi_v}$ and $\sJ_{\pi_v}$.

\begin{proposition}
\label{prop:localdistrid}
Let $v$ be a place of $F$.
Suppose that $\pi_v$ is tempered, and that Conjecture
\ref{conj:unramified} holds.
Assume that one of the following conditions holds.
\begin{enumerate}[\normalfont(1)]
\item The place $v$ is inert,
and $(f_v', \phi_v')$ and $\{(f_v, \phi_{1,v} \otimes \phi_{2,v}), (0, 0)\}$ are the unramified test functions.
\item The place $v$ is split, and
$(f_v', \phi_v')$
and $\{(f_v, \phi_{1,v} \otimes \phi_{2,v})\}$ are
matching test functions. 
\end{enumerate}
Then
\[
	\sJ_{\pi_v}(f_v, \phi_{1,v} \otimes \phi_{2,v})
	= \zeta_{E_v}(1)\sI_{\Pi_v}(f_v', \phi_v').
\]
% \halp{Then
% \begin{align*}
% 	\sJ_{\pi_v}^\natural(f_v, \phi_{1,v} \otimes \phi_{2,v})
% 	&= \frac{\Vol(\GL_n(\O_{E_v}))}{\Vol(\U(V)(\O_{F_v}))}
% 	\cdot \sI_{\Pi_v}^\natural(f_v', \phi_v') \\
% 	&=\zeta_{E_v}(1) \cdot \left( 
% 	\prod_{i=1}^n \frac{\zeta_{E_v}(i)}{L(i, \eta_v^i)}
% 	\right)^{-1} 
% 	\sI_{\Pi_v}^\natural(f_v', \phi_v').
% \end{align*}
% }
\end{proposition}
\begin{proof}[Proof of Proposition~\ref{prop:localdistrid} (1)]
First suppose we are in case (1).
Then 
\begin{gather*}
	(f_v', \phi_v') = \left(
	\frac{1}{\Vol(\GL_n(\O_{E_v}))}{\mathbbm{1}_{\GL_n(\O_{E_v})}\otimes \frac{1}{\Vol(\GL_n(\O_{E_v}))}\mathbbm{1}_{\GL_n(\O_{E_v})}}, 
	\mathbbm{1}_{\O_{E_{v},n}}\right), \\
	(f_v, \phi_{1,v} \otimes \phi_{2,v}) = \left(
	\frac{1}{\Vol(\U_n^+(\O_{F_v}))^2}{\mathbbm{1}_{\GL_n(\O_{E_v})}}, 
	\mathbbm{1}_{\LL^+(\O_{F_v})} \otimes \mathbbm{1}_{\LL^+(\O_{F_v})}\right).
\end{gather*}

Let $W_v$ be the $\GL_n(\O_{E_v})$-fixed vector of
$\W(\pi_v, \psi_v)$ normalized by $W_v(1) = 1$. Then
\begin{align*}
\sI_{\Pi_v}(f_v', \phi_v') &=
\frac{\lambda_v(W_v, \sig{\ol{W_v}}, \mu_v, \phi_v')}{\langle W_v, W_v \rangle_v}
= \frac{L(\frac12, \pi_v \times \sig{\pi_v^\vee} \otimes \mu_v^{-1})}{L(1, \pi_v, \Ad)}. 
\end{align*}

Let $\varphi_v$ be the $\GL_n(\O_{E_v})$-fixed vector of
$\pi_v$ such that $\langle \varphi_v, \varphi_v \rangle_v = 1$.
By Conjecture \ref{conj:unramified},
we have
\begin{align*}
\sJ_{\pi_v}(f_v, \phi_{1,v} \otimes \phi_{2,v})
&= \alpha_v(\pi_v(f_v)\varphi_v, \phi_{1,v}, \varphi_v, \phi_{2,v})\\
&= \frac{\Vol(\GL_n(\O_{E_v}))}{\Vol(\U_n^+(\O_{F_v}))^2}
 \cdot \L(\tfrac12, \pi_v) \cdot \Vol(\U_n^+(\O_{F_v})) \\
&= \zeta_{E_v}(1) \cdot \frac{L(\frac12, \pi_v \times \sig{\pi_v^\vee} \otimes \mu_v^{-1})}{L(1, \pi_v, \Ad)},
\end{align*}
so the proposition follows.
\end{proof}

\begin{proof}[Proof of Proposition~\ref{prop:localdistrid} (2)]
Now suppose we are in case (2).
Suppose $\pi_v = \pi_{1} \boxtimes \pi_{2}$ for
$\pi_{1}, \pi_{2}$ representations of $\GL_n(F_v)$,
and 
\[ f_v' = f_{1,1}' \otimes f_{1,2}' \otimes f_{2,1}'
\otimes f_{2,2}'\]
for $f_{ij}' \in \C_c^\infty(\GL_n(F_v))$.
Then $\Pi_v = \pi_{1} \boxtimes \pi_{2} \boxtimes \ol{\pi_{2}} \boxtimes \ol{\pi_{1}}$, and $\zeta_{E_v}(1)^{-1}f_v = f_1 \otimes f_2$, with
\[ f_1 = f_{1,1}' \ast f_{2,2}'^\vee, \quad
\transp{f_2} = f_{1,2}' \ast f_{2,1}'^\vee,
\] where
$f_{ij}'^\vee(g) \coloneqq f_{ij}'(g^{-1})$.
Furthermore we have $\phi_v' = \phi_{1,v} \otimes \ol{\phi_{2,v}}$. Also, suppose $\mu_v = (\mu_1, \mu_1^{-1})$.

Note that the representation $\pi_2'$ given by
$\pi_2'(g) \coloneqq \pi_2(\transp{g^{-1}})$ is isomorphic
to $\ol{\pi_2}$. Then we have
\begin{align*}
&\zeta_{F_v}(1)^{-1}J_{\pi_v}(f_v, \phi_{1,v} \otimes \phi_{2,v}) \\
&= \sum_{\varphi_1, \varphi_2}
\int_{\GL_n(F_v)} 
\langle \pi_1(g)\pi_1(f_1)\varphi_1, \varphi_1\rangle_v
\langle \pi_2'(g) \pi_2'(\transp f_2^\vee) \varphi_2, \varphi_2\rangle_v
\langle R_{\mu_1}(g) \phi_{1,v}, \phi_{2,v} \rangle_v \, dg \\
&= \sum_{W_1, W_2}
\alpha(\pi_1(f_1) W_1, \ol{\pi_2}(\transp f_2^\vee) \ol{W_2},
\phi_{1,v}, W_1, \ol{W_2}, \phi_{2,v}),
\end{align*}
where $\varphi_1$ and $\varphi_2$ run over an orthonormal
basis of $\pi_1$ and $\pi_2$, respectively,
$W_1$ runs over an orthonormal basis of
$\W(\pi_1, \psi_v')$, $W_2$ runs over an orthonormal basis
of $\W(\pi_2, \psi'_v)$, and
$\alpha$ is as in Definition \ref{def:alphaWhittaker}.

Thus, by Proposition \ref{prop:alphalambda}, we have
\begin{align*}
&\zeta_{E_v}(1)^{-1}J_{\pi_v}(f_v, \phi_{1,v} \otimes \phi_{2,v}) \\
&= \sum_{W_1, W_2}
\lambda_v(\pi_1(f_{11}' \ast f_{22}'^\vee)W_1, 
\ol{\pi_2}(f_{21'} \ast f_{12}'^\vee)\ol{W_2}, \mu_1, \phi_{1,v}) \ol{\lambda_v(W_1, \ol{W_2}, \mu_1, \phi_{2,v})} \\
&= \sum_{W_2, W_2}
\lambda_v(\pi_1(f_{11}')W_1, \ol{\pi_2}(f_{21}')\ol{W_2}, \mu_1, \phi_{1,v})
\lambda_v(\ol{\pi_1}(f_{22}')\ol{W_1}, \pi_2(f_{12}')W_2, \mu_1^{-1}, \ol{\phi_{2,v}}).
\end{align*}

By directly expanding the definition of $I_{\Pi_v}(f_v', \phi_v')$,
we find that this is also equal to $I_{\Pi_v}(f_v', \phi_v')$.
\end{proof}

\subsection{Nonvanishing results}
In this subsection, we collect some lemmas which will be used
to choose good test
functions such that the local distributions are nonzero.

We first recall the following theorem from \cite{MR3228451}.
\begin{theorem}[{\cite[Theorem~7.1.1]{MR3228451}}]
\label{thm:good}
Let $v$ be a nonarchimedean split place of $F$
such that $\pi_v$ is supercuspidal. 
Assume that the distribution $\sJ_{\pi_v}$ is
nonzero. 
Let $U$ be a Zariski open dense subset of
$\GL_n(E_v) \times E_{v,n}$.
Then there exists a test function 
$(f_v, \phi_v)$, with
$\phi_v = \sum_{j = 1}^r \phi_{1,v}^{(j)} \otimes \phi_{2,v}^{(j)}$, such that $\sJ_{\pi_v}(f_v, \phi_v)
\neq 0$, and the function
\[
	[g, z] \mapsto f_v(g)\left(
	\sum_{j = 1}^r \phi_{2,v}^{(j)} \otimes 
	\ol{\omega_{\psi_v', \mu_v}(g_1, \transp{g_1^{-1}})\phi_{1,v}^{(j)}}
	\right)^\ddagger(z)
\]
is supported on $U$, where 
$g = (g_1, g_2) \in \GL_n(E_v)$.
\end{theorem}

\begin{lemma}
\label{lem:good}
With notation and assumptions as in 
Theorem \ref{thm:good}, there exists a test
function $(f_v, \phi_v)$, 
with $\phi_v = \sum_{j = 1}^r \phi_{1,v}^{(j)} \otimes \phi_{2,v}^{(j)}$,
and a choice of 
global representatives $R^\beta$, such that 
$\sJ_{\pi_v}(f_v, \phi_v)$ is nonzero,
$f_v$ is supported on the regular semisimple locus
of $\GL_n(E_v)$ under the $\U(V)(F_v) \times \U(V)(F_v)$-action, and if $[\zeta, z] \in Y^\beta$ and
\[
	f_{v}(g^{-1}\zeta h)
	\left(\sum_{j=1}^r \phi_{2, v}^{(j)} \otimes
	\ol{\omega_{\psi_{v}, \mu_{v}}(h^{-1}g)\phi_{1,v}^{(j)}}\right)^\ddagger(zh) \neq 0
\]
for some $g, h \in \U(V)(F_{v})$, then $[\zeta, z]
\in Y^\beta_\rss$.
\end{lemma}
\begin{proof}
Let $U$ be the regular semisimple locus of 
$\GL_n(E_v) \times E_{v,n}$, i.e., the set of elements
$[g, z] \in \GL_n(E_v) \times E_{v,n}$ such that
\[
	[\beta^{-1}g^\ast\beta g, z, \beta^{-1}z^\ast]
	\in M_n(E_v)_\rss.
\]
Let $(f_v, \phi_v)$ be the test function given by
Theorem \ref{thm:good} with this $U$, and suppose that 
$\phi_v = \sum_{j = 1}^r \phi_{1,v}^{(j)} \otimes \phi_{2,v}^{(j)}$.
We will find a suitable set of representatives $R^\beta$
so that this test function has the desired properties.

Since $v$ is split, we can assume $\beta = (1, -1)$.
Let $m \ge 0$ be an integer such that
$\omega_{\psi_{v}', \mu_{v}}(h)\phi_{1,v}^{(j)} = \phi_{1,v}^{(j)}$ for all $1 \le j \le r$ and
$h \in \U(V)(F_{v}) \cap K_m$, where 
\[
	K_m = 1 + \varpi_{F_v}^m\Mat_n(\O_{E_v}).
\]
We can also assume that $f_{v}$ has support
in $K_m (x_1, x_2)$ for some $(x_1, x_2) \in \GL_n(E_v)$.

Let $k \ge m$ be an integer such that
\[
	\varpi_{F_v}^{\lfloor (k-m)/(n-1)\rfloor}(x_1\transp{x_2}) \in \Mat_n(\O_{F_v}).
\] 
By Proposition \ref{prop:globalkottwitz}, we can
choose the representatives $S$ so that every
$\gamma \in S$ is $k$-Kottwitz at $v$.
Then if $S$ and $R^\beta$ are compatible, 
the representatives $\zeta \in R^\beta$
will be $k$-Kottwitz with respect to $\beta$ at $v$.

Then, we claim that for all $1 \le j \le r$, $g, h \in \U(V)(F_v)$, and $\zeta \in R^\beta$ such that $g^{-1}\zeta h
\in K_m(x_1, x_2)$, we have
\[
	\omega_{\psi_{v}', \mu_{v}}(h^{-1} g)\phi_{1,v}^{(j)}
	= 
	\omega_{\psi_{v}', \mu_{v}}(h_1^{-1}\zeta_1g_1, \transp h_1\transp \zeta_1^{-1} \transp g_1^{-1})\phi_{1,v}^{(j)}.
\]
Indeed, $g^{-1}\zeta h \in K_m(x_1, x_2)$ implies that
\[
	g^{-1}\zeta^\ast\zeta g
	\in K_m(x_1\transp{x_2}, x_2\transp{x_1})K_m.
\]
Since $\zeta$ is $k$-Kottwitz at $v$, we have
\[
	\zeta_1 = c_{n-1}(\transp\zeta_2\zeta_1)^{n-1} + \cdots + c_0,
\]
for some $c_0 \in 1 + \varpi_{F_v}^{k}\O_{F_v}$ and
$c_1, \dots, c_{n-1} \in \varpi_{F_v}^{k}\O_{F_v}$. This,
along with the assumptions on $k$, imply that 
\[
	g_1^{-1}\zeta_1 g_1 \in 1 + \varpi_{F_v}^m\Mat_n(\O_{F_v}),
\]
so
\[
	h^{-1} g
	\in (h_1^{-1}\zeta_1 g_1, \transp h_1 \transp \zeta_1^{-1}\transp g_1^{-1})K_m,
\]
from which the claim follows.

Thus, we see that if $[\zeta, z] \in Y^\beta$ and
\[
	f_v(g^{-1}\zeta h)
	\left(\sum_{j=1}^r \phi_{2,v}^{(j)} \otimes
	\ol{\omega_{\psi_v', \mu_v}(h^{-1}g)\phi_{1,v}^{(j)}}\right)^\ddagger(zh) 
	\neq 0
\]
for some $g, h \in \U(V)(F_v)$, then the above expression is
equal to
\[
	f_v(g^{-1}\zeta h)
	\left(\sum_{j=1}^r \phi_{2,v}^{(j)} \otimes
	\ol{\omega_{\psi_v', \mu_v}(h_1^{-1}\zeta_1 g_1, \transp h_1\transp \zeta_1^{-1} \transp g_1^{-1})\phi_{1,v}^{(j)}}\right)^\ddagger(zh),
\]
so by assumption on $(f_v, \phi_v)$, we have
$[g^{-1}\zeta h, zh]$ is a regular semisimple element of
$\GL_n(E_v) \times E_{v, n}$, so $[\zeta, z]$ is a regular
semisimple element of $Y^\beta$.
\end{proof}

\begin{lemma}
\label{lem:goodv1}
Let $v$ be a nonarchimedean %split 
place of $F$.
%such that $\pi_v$ is supercuspidal.
Assume that $\ell_v \neq 0$ for some $\varphi_{v} \in \pi_v$. 
Then there exists a test function $(f_v, \phi_v)$
such that $\sJ_{\pi_v}(f_v, \phi_v) \neq 0$, and
$f_v$ is a truncated matrix coefficient of $\pi_v$.
\end{lemma}
\begin{proof}
Let $\varphi_{v} \in \pi_v$ be orthogonal to the subspace
of $\pi_v$ consisting of vectors $e$ such
that $\ell_v(e, \phi_{0,v}) = 0$ for all
$\phi_{0, v} \in \S(\LL(F_v))$.
Let
\[
	f_{0}(g) = \langle \pi_v(g^{-1})\varphi_v, \varphi_v \rangle_v \mathbbm{1}_{G^\ast(F_v)}(g).
\]
Note that $\langle \pi_v(f_0)\varphi_v, \varphi_v \rangle \neq 0$, so there exists $\phi_{0,v} \in \S(\LL(F_v))$ such
that 
\[ 
	\ell_v(\pi_v(f_0)\varphi_v, \phi_{0,v}) \neq 0. 
\]
Also note that $f_0^\star = f_0$, 
where $f_0^\star(g) \coloneqq \ol{f_0(g^{-1})}$.
Let $f_v = f_0 \ast f_0^\star$, and let
$\phi_v = \phi_{0,v} \otimes \phi_{0, v}$.
Then $f_v$ is also a truncated matrix coefficient, and
\begin{align*}
	\sJ_{\pi_v}(f_v, \phi_v)
	&= \sum_{e}
	\ell_v(\pi_v(f_0 \ast f_0^\star) e, \phi_{0,v})\ol{\ell_v(e, \phi_{0,v})}\\
	&= \sum_{e}
	|\ell_v(\pi_v(f_0)e, \phi_{0,v})|^2
	\neq 0,
\end{align*}
where the sums are over an orthonormal basis of $\pi_v$.
\end{proof}

\begin{lemma}
\label{lem:goodarch}
Let $v$ be an archimedean place of $F$.
Suppose that 
\[\lambda_v(W_{1,v}, W_{2,v}, \mu_v,
\widetilde\phi_v') \neq 0\]
for some 
$W_{1,v} \in \W(\pi_v, \psi_v)$, $W_{2,v}
\in \W(\sig{\pi_v^\vee}, \ol{\psi_v'})$,
and $\widetilde \phi_v' \in
\S(E_{v,n})$. Then there
exists $\phi_v' \in \S(E_{v,n})$ which
is a finite linear combination of functions
of the form $\phi_{1,v}' \otimes \phi_{2,v}'$
where $\phi_{1,v}', \phi_{2,v}' \in \S(F_{v,n})$,
such that $\lambda_v(W_{1,v}, W_{2,v}, \mu_v, \phi_v') \neq 0$.
\end{lemma}
\begin{proof}
This follows from the fact that $\lambda_v(W_{1,v}, W_{2,v}, \mu_v, \widetilde\phi_v')$
is a continuous functional of $\widetilde \phi_v'$
with respect to the Fr\'echet space topology on
$\S(E_{v, n})$,
by \cite[Theorem~2.3]{MR2533003}.
\end{proof}

\section{Proof of main results}
\subsection{Proof of Theorem \ref{thm:maintheorem}}
In this subsection, we let
\[
	\sJ_{\pi_v}(f_v, \phi_{1,v} \otimes \phi_{2,v})
	= \sum_{\varphi_v} 
	\frac{\ell_v(\pi_v(f_v)\varphi_v, \phi_{1,v})
	\ol{\ell_v(\varphi_v, \phi_{2,v})}}
	{\langle \varphi_v, \varphi_v \rangle_v}.
\]
Note that
\[
	\sJ_{\pi}(f, \phi_1 \otimes \phi_2) = |C|^2
	\prod_v \sJ_{\pi_v}(f_v, \phi_{1,v}\otimes \phi_{2,v})
\]
where $C$ is the constant in \eqref{eq:factorPell}.

\mainthm*
\begin{proof}
First we fix the Lagrangian $\LL$.
Since $V$ is split, it has a basis
$e_1, \dots, e_n$ such that the corresponding 
skew-Hermitian matrix is 
\[
j\begin{pmatrix}
& & 1 \\ & \iddots & \\ 1 & &
\end{pmatrix},
\] 
and we can take $\LL$ to be
the direct sum of the $E$-vector space generated by $e_1,
\dots, e_{\lfloor n/2 \rfloor}$ and the $F$-vector space
generated by $e_{\lceil n/2 \rceil}$.
This makes the
assumptions at the beginning of Section \ref{sec:FL} true
at all inert places.

Now, suppose that $\P(\varphi, \widetilde\phi_0) \neq 0$
for some factorizable $\varphi \in \pi$
and $\widetilde\phi_0 \in \S(\LL(\AA_F))$. 
% Since $\P$ is an element of $\Hom_{\U(V)(\AA_F)}(\pi \otimes \ol{\omega_{\psi', \mu}}, \CC)$, 
% it is clear that 
% $\Hom_{\U(V)(F_v)}(\pi_v \otimes \ol{\omega_{\psi_v', \mu_v}}, \CC) \neq 0$ for all $v$. 
% In fact, since $\pi_v$ is unramified at all inert $v$, we conclude that $V$ is the split Hermitian
% space of dimension $n$ over $E$. 
% \todo{is this even true and why?}
By a theorem of Jacquet--Piatetski-Shapiro--Shalika
\cite{MR701565}, 
$L(\frac12, \pi \times \sig{\pi^\vee} \otimes \mu^{-1}) \neq 0$ if and only if there exists $\varphi \in \Pi$ and $\phi' \in \S(\AA_{E,n})$ such that
$\lambda(\varphi, \mu, \phi') \neq 0$.
Thus, in order to show that 
$L(\frac12, \pi \times \sig{\pi^\vee} \otimes \mu^{-1}) \neq 0$, it suffices to show there exists a
test function $(f', \phi')$ such that 
$\sI_{\Pi}(f', \phi') \neq 0$.

% As in \cite[Section~6]{MR3228451}, we say that
% a test function $(f, \phi)$ is of positive type if
% it is a finite sum of test functions of the form
% \[ (f_0 \ast f_0^\star, \phi_0 \otimes \phi_0), \]
% for some $f_0 \in C_c^\infty(G(\AA_F))$,
% $\phi_0 \in \S(\LL(\AA_F))$. 
% Here, $f_0^\star(g) \coloneqq \ol{f_0(g^{-1})}$.
% We have $\sJ_\pi(f, \phi) \ge 0$
% for test functions of positive type.
% There is an analogous local statement and definition.

Since $\P(\varphi, \widetilde\phi_0) \neq 0$, there exists
a factorizable test function 
$(\widetilde f, \widetilde\phi)$ of the form 
$(f_0 \ast f_0^\star, \widetilde\phi_0 \otimes \widetilde\phi_0)$ 
such that $\sJ_\pi(\widetilde f, \widetilde \phi) \neq 0$.
% Indeed, we can take $\widetilde\phi = \widetilde \phi_0
% \otimes \widetilde \phi_0$,
% and $\widetilde f = \widetilde f_0 \ast \widetilde f_0^{\star}$ for any $\widetilde f_0 \in \C_c^\infty(G(\AA_F))$ such that $\pi(\widetilde f_0) \neq 0$.
We modify the test function $(\widetilde f, \widetilde \phi)$ in the
same way as in \cite[Section~6.2]{MR3228451}.

Since $\pi_v$ is unramified at all inert $v$,
we can assume that $(\widetilde f_v, \widetilde \phi_v)$ 
is the unramified test function such that
$\sJ_{\pi_v}(\widetilde f_v, \widetilde \phi_v) \neq 0$,
at all inert places $v$.

At the place $v_1$, let $(f_{v_1}, \phi_{v_1})$ be the 
test function given by Lemma \ref{lem:goodv1}.
At the place $v_2$, let $(f_{v_2}, \phi_{v_2})$ be
the test function given by Theorem \ref{thm:good},
and let $S$ and $R^\beta$ be the sets of representatives given by Lemma \ref{lem:good}.
We define
\[
	f = \prod_{v \neq v_1, v_2} \widetilde f_v \times
	f_{v_1} \times f_{v_2},
\]
and 
\[
	\phi = \prod_{v \neq v_1, v_2} \widetilde \phi_{v} \times \phi_{v_1} \times \phi_{v_2}.
	% \left( \sum_{j} \phi_{1, v_2}^{(j)} \otimes
	% \phi_{2, v_2}^{(j)} \right).
\]

Then $(f, \phi)$ is a good test function, and
$\sJ_{\pi}(f, \phi) \neq 0$. 
Consider the test functions $\{(f^\beta, \phi^\beta)\}_{\beta \in \SHerm_n^\times(F)}$, where 
$(f^{\beta'}, \phi^{\beta'}) = 0$ if $\beta'
\not\sim \beta$ and $(f^\beta, \phi^\beta) = (f, \phi)$.
Since $V$ is unramified at all inert places
and $(f_v, \phi_v)$ is the unramified test function
at all inert places, by Propositions \ref{prop:splitmatch},
\ref{prop:betaminusFL}, and Theorem \ref{thm:FL},
there exists a smooth transfer $(f', \phi')$ of
these test functions.
Then $(f', \phi')$ is also good, so by Proposition \ref{prop:identity},
we have
\[
	\sI_{\Pi}(f', \phi') = \sJ_{\pi}(f, \phi) \neq 0,
\]
as desired.

In the other direction, suppose $L(\frac12, \pi
\times \sig{\pi^\vee} \otimes \mu^{-1}) \neq 0$. 
We will show that the period integral $\P$ is nonzero.

Since $L(\frac12, \pi \times \sig{\pi^\vee} \otimes \mu^{-1})
\neq 0$, there exists
$\varphi' \in \Pi$, $\widetilde\phi' \in \S(\AA_{E,n})$
such that
$
	\lambda(\varphi', \mu, \widetilde\phi') \neq 0
$.
Furthermore, because $\Pi = \pi \boxtimes \sig{\pi^\vee}$, there exists $\varphi \in \Pi$ such that 
$\P'(\varphi) \neq 0$.
% Let $\widetilde f'$ be such that $\Pi(\widetilde f')\varphi = \varphi'$,
% and $\Pi(\widetilde f')\varphi'' = 0$ for all $\varphi''$
% orthogonal to $\varphi$. 
Thus, there exists factorizable $\widetilde f'$ such that
$\sI_{\Pi}(\widetilde f', \widetilde \phi') 
	%= Z(\varphi', \mu, \widetilde\phi')\P'(\varphi) 
	\neq 0$.

For inert places $v$, we can take $\varphi_v$, $\varphi_v'$, and 
$\widetilde\phi_v'$ to be the spherical elements, and thus we can assume the test function 
$(\widetilde f_v', \widetilde\phi_v')$ is the unramified test function for all inert $v$.

At each archimedean place $v$ of $F$,
let $\phi_v'$ be the test function given by
Lemma \ref{lem:goodarch}.

By Proposition \ref{prop:localdistrid},
at split places $v$ of $F$, we have
\[
	\sI_{\Pi_v}(f'_v, \phi_v') = c\sJ_{\pi_v}(f_v, \phi_v),
\]
for some nonzero constant $c$ and $(f_v', \phi_v')$ and $\{(f_v, \phi_v)\}$ which
are smooth transfer of each other.
In particular, this implies that $\sJ_{\pi_v}$
is nonzero for each $v$.
Let $(f_{v_1}, \phi_{v_1})$ be the test function
given by Lemma \ref{lem:goodv1} such
that $\sJ_{\pi_v}(f_{v_1}, \phi_{v_1}) \neq 0$ and
$f_{v_1}$ is a truncated matrix coefficient.
Let $(f_{v_1}', \phi_{v_1}')$ be the smooth
transfer of $\{(f_{v_1}, \phi_{v_1})\}$.
Similarly, let $(f_{v_2}, \phi_{v_2})$ be the
test function given by Theorem \ref{thm:good}, and
let $S$ and $R^\beta$ be the sets of representatives given
by Lemma \ref{lem:good}.
Let $(f_{v_2}', \phi_{v_2}')$ be the smooth
transfer of $\{(f_{v_2}, \phi_{v_2})\}$.

We modify $(\widetilde f', \widetilde \phi')$
at the archimedean places and $v_1$, $v_2$
in the same way as before, to obtain a good
test function $(f', \phi')$ which is unramified
at all inert places $v$ and satisfies
\[
	\sI_{\Pi}(f', \phi') \neq 0.
\]
Then if $\{(f^\beta, \phi^\beta)\}_{\beta \in
[\SHerm_n^\times(F)]}$ is the smooth transfer
of $(f', \phi')$, by Proposition \ref{prop:identity},
we have 
\[
	\sI_{\Pi}(f', \phi') = \sJ_{\pi}(f^{\beta},\phi^\beta) \neq 0,
\]
for $\beta$ corresponding to the split Hermitian
space $V$. This shows that the period integral
$\P$ for $V$ is not identically zero.
\end{proof}

\subsection{Proof of Theorem \ref{thm:refined}}
In this subsection, we assume that $\pi$ is tempered everywhere,
and let
\begin{align*}
	J_{\pi_v}(f_v, \phi_{1,v} \otimes \phi_{2,v})
	= \sum_{\varphi_v}
	\frac{\alpha_v(\pi_v(f_v)\varphi_v, \phi_{1,v}, \varphi_v, \phi_{2,v})}{\langle \varphi_v, \varphi_v \rangle_v}.
\end{align*}

\refined*
\begin{proof}
Fix the Lagrangian $\LL$ as in the proof of Theorem \ref{thm:maintheorem}.

Choose matching test functions $(f', \phi')$ and
$\{(f^\beta, \phi_1^\beta \otimes \phi_2^\beta)\}_{\beta \in [\SHerm_n^\times(F)]}$ and sets of representatives $S$ and $R^\beta$
which satisfy the following properties.
\begin{enumerate}[(1)]
\item At all inert places $v$ of $F$ and almost all split
places, $(f_v', \phi_v')$ and $\{(f_v^\beta, \phi_{1,v}^\beta \otimes \phi_{2,v}^\beta) \}_{\beta \in [\SHerm_n^\times(F_v)]}$
are as in the fundamental lemma.

\item At all other split places $v$, we have $\sJ_{\pi_v}(f_v, \phi_{1,v} \otimes \phi_{2,v}) \neq 0$, and $f_{v_1}$ is a truncated matrix coefficient. 
This is possible by Lemma \ref{lem:goodv1}.

\item The test function
$(f_{v_2}, \phi_{1,v_2} \otimes \phi_{2, v_2})$ and the
sets of global representatives $S$ and $R^\beta$ are given
by Lemma \ref{lem:good}.

\item At all archimedean places $v$ of $F$, the test function
$\phi_v'$ is a finite linear combination of functions of
the form $\phi_{1,v}' \otimes \phi_{2,v}'$, and
$\sI_{\Pi_v}(f_v', \phi_v') \neq 0$. This is possible by
Lemma \ref{lem:goodarch}.
\end{enumerate}
Then $(f', \phi')$ and
$\{(f^\beta, \phi_1^\beta \otimes \phi_2^\beta)\}_{\beta \in [\SHerm_n^\times(F)]}$ are good matching global test funtions.

Note that Proposition \ref{prop:localdistrid} is equivalent
to
\begin{align*}
	\sJ_{\pi_v}^\natural(f_v, \phi_{1,v} \otimes \phi_{2,v})
	% &= \frac{\Vol(\GL_n(\O_{E_v}))}{\Vol(\U(V)(\O_{F_v}))}
	% \cdot \sI_{\Pi_v}^\natural(f_v', \phi_v') \\
	&=\zeta_{E_v}(1) \cdot \left( 
	\prod_{i=1}^n \frac{\zeta_{E_v}(i)}{L(i, \eta_v^i)}
	\right)^{-1} 
	\sI_{\Pi_v}^\natural(f_v', \phi_v').
\end{align*} 
Then by Propositions \ref{prop:identity}, \ref{prop:Idecomp},
and \ref{prop:localdistrid}, we have
\begin{align*}
	\sJ_{\pi}(f, \phi_1 \otimes \phi_2)
	= \zeta_E^\ast(1)\sI_{\Pi}(f', \phi')
	% &= \sum_{\varphi_1, \varphi_2 \in \OB(\pi)} 
	% Z(\tfrac12, \Pi(f')(\varphi_1 \otimes \sig{\varphi_2^\vee}), \mu, \phi')
	% \langle \varphi_1, \varphi_2 \rangle \\
	% &= \sum_{\varphi \in \OB(\pi)}
	% Z(\tfrac12, \Pi(f')(\varphi \otimes \sig{\varphi^\vee}), \mu, \phi') \\
	&= \zeta_E^\ast(1)\cdot \frac{L(\tfrac12, \pi \times \sig{\pi^\vee} \otimes \mu^{-1})}{\Res_{s=1} L(s, \pi \times \pi^\vee)} 
	\prod_v \sI_{\Pi_v}^\natural(f_v', \phi_v') \\
	&= \L(\tfrac12, \pi)
	% \frac{L(\tfrac12, \pi \times \sig{\pi^\vee} \otimes \mu^{-1}) \cdot \Vol(E^\times \bs \AA_E^{1}) 
	% \cdot \displaystyle\prod_{i=2}^n \zeta_{E}(i)}{n \times 
	% \Res_{s=1} L(s, \pi \times \pi^\vee) 
	% \cdot \displaystyle\prod_{i=1}^n  L(i, \eta^i)} 
	% % \frac{\displaystyle\prod_{i=2}^n \zeta_{E}(i)}{\displaystyle\prod_{i=1}^n  L(i, \eta^i)}
	\prod_v \sJ_{\pi_v}^\natural(f_v, \phi_{1,v} \otimes \phi_{2,v}).
\end{align*}
Now, we also have
\[
	\sJ_\pi(f, \phi_1 \otimes \phi_2) =
	C \times \prod_v \sJ_{\pi_v}^\natural(f_v, \phi_{1,v} \otimes \phi_{2,v}),
\]
where $C$ is the constant in \eqref{eq:factorP}.
Since we chose the test function
$(f, \phi_1 \otimes \phi_2)$ so that
$\prod_v \sJ_{\pi_v}^\natural(f_v, \phi_{1,v} \otimes \phi_{2,v}) \neq 0$, we must have $C = \L(\frac12, \pi)$.
\end{proof}

\bibliography{ggp}{}
\bibliographystyle{alpha}
\end{document}